\documentclass[12pt]{amsart}

\usepackage{amsfonts,amssymb,amsmath,amsthm}
\usepackage{url}
\usepackage[colorlinks=true, urlcolor=black, citecolor=black, linkcolor=black, hyperfootnotes=true]{hyperref}
\usepackage{bbm}
\usepackage{bm}
\usepackage{anysize}
\usepackage{cleveref}

\usepackage{tikz-cd}
\usetikzlibrary{decorations.markings, arrows.meta}
\tikzset{
	marrow/.style={decoration={markings,mark=at position 0.5 with {\arrow{#1}}}, postaction=decorate}
}

\numberwithin{equation}{section}

\theoremstyle{plain}
\newtheorem{theorem}{Theorem}[section]
\newtheorem{lemma}[theorem]{Lemma}
\newtheorem{corollary}[theorem]{\bf Corollary}

\newtheorem{proposition}[theorem]{\bf Proposition}

\theoremstyle{definition}
\newtheorem{definition}[theorem]{\bf Definition}

\newcommand{\lin}{\mathrel{\vphantom{<}\text{\mathsurround=0pt\ooalign{$<$\cr$-$\cr}}}}
\newcommand{\nil}{\mathrel{\vphantom{>}\text{\mathsurround=0pt\ooalign{$>$\cr$-$\cr}}}}

\newcommand{\sqni}{\mathrel{\vphantom{\sqsupset}\text{\mathsurround=0pt\ooalign{$\sqsupset$\cr$-$\cr}}}}
\newcommand{\sqnii}{\mathrel{\vphantom{\sqsupset}\text{\mathsurround=0pt\ooalign{$\sqsupset$\cr$=$\cr}}}}

\newcommand{\sqin}{\mathrel{\vphantom{\sqsubset}\text{\mathsurround=0pt\ooalign{$\sqsubset$\cr$-$\cr}}}}

\newcommand{\oc}{\mathop{\varpropto}}
\newcommand{\co}{\mathop{\reflectbox{$\oc$}}}

\newcommand{\Dashv}{\mathrel{\reflectbox{$\vDash$}}}

\author[T. Bice]{Tristan Bice}
\address{Institute of Mathematics of the Czech Academy of Sciences, \v{Z}itn\'a 25, Prague}
\email{bice@math.cas.cz}
\thanks{Research of Tristan Bice was supported by GAČR project 22-07833K and RVO: 67985840.}

\author[M. Malicki]{Maciej Malicki}
\address{Institute of Mathematics, Polish Academy of Sciences}
\email{mamalicki@gmail.com}


\begin{document}

\title{Homeomorphisms of the Pseudoarc}

\begin{abstract}
    We construct homeomorphisms of compacta from relations between finite graphs representing their open covers.  Applied to the pseudoarc, this yields simple Fra\"iss\'e theoretic proofs of several important results, both old and new.  Specifically, we recover Bing's classic results on the uniqueness and homogeneity of the pseudoarc.  We also show that the autohomeomorphism group of the pseudoarc has a dense conjugacy class, thus confirming a conjecture of Kwiatkowska.
\end{abstract}

\maketitle

\setcounter{tocdepth}{2}
\tableofcontents

\section{Introduction}

Here we focus on dynamical systems coming from a single homeomorphism $\phi$ acting on a topological space $X$.  While such systems have been studied for a long time, they are best understood when the underlying space $X$ is particularly `nice'.  For example, when $X$ is a manifold one can bring to bear various tools from analysis and geometry.  At the other extreme, dynamical systems on zero dimensional spaces are particularly amenable to combinatorial methods, owing to the arbitrarily fine clopen partitions such spaces possess.  Indeed, this led Shimomura to develop a very general method for constructing and analysing zero dimensional systems from `graph covers' in \cite{Shimomura2014}, building on earlier work on Cantor systems in \cite{GambaudoMartens2006} and \cite{BernadesDarji2012}.  This method has much potential to push our understanding of zero dimensional systems even further, both improving on previous results and solving open problems, as already demonstrated in \cite{BoronskiKupkaOprocha2019}, for example.

In contrast, much less is known about dynamical systems on spaces lying between these two extremes, like those often considered in continuum theory.  By definition, continua are connected but can have local branching and twisting behaviour that sets them far apart from any kind of manifold.  On the other hand, the lack of non-trivial clopen partitions means that continua lack any obvious combinatorial structure.  And without a suitable combinatorial description of the spaces, there would appear to be little hope of finding some graph-like description of the homeomorphisms upon them.

Fortunately, more combinatorial methods of constructing continua have been developed over the past couple of decades.  The general idea is to obtain continua from appropriate limits of finite graphs, as first demonstrated by Irwin and Solecki in \cite{IrwinSolecki2006}.  There they showed that the pseudoarc, a space of central interest in continuum theory, can indeed be constructed as a projective Fra\"iss\'e limit in a category of edge-preserving surjections between path graphs.  Taking the same kinds of Fra\"iss\'e limits  in different categories of finite graphs yields other important continua, as demonstrated in \cite{BaKw} and \cite{PanagiotopoulosSolecki2022}, for example.  Potential for combining the Irwin-Solecki approach to constructing spaces with the Shimomura approach to constructing homeomorphisms has even been demonstrated recently in \cite{Kucharski2022}, which may well be a promising line of future research.

One issue with the Irwin-Solecki approach, however, is that the limit in the usual sense is only a `pre-space', i.e. a Cantor space together with an equivalence relation -- it is only after taking the quotient that one obtains the desired continuum.  As a result the graphs get represented as closed covers on the space whose individuals only touch not overlap (i.e.~their interiors are disjoint).  These are somewhat more difficult to work with than the open covers usually employed in topological arguments.  This is perhaps best illustrated by attempts to prove in this framework that the autohomeomorphism group of the pseudoarc has a dense conjugacy class, which have so far all been flawed -- see \cite{Kwiatkowska2023}.

To resolve this issue, the authors of \cite{BaBiVi2} developed a slightly different method of constructing compacta using more general relations between the graphs, building on the theory developed in \cite{BaBiVi}.  The key advantage here is that then the graphs do indeed correspond to open covers of the resulting space.  This approach could thus be viewed as an abstract generalisation of Bing's construction of the pseudoarc in \cite{Bing1948} via crooked chains of open sets, while the Irwin-Solecki approach is more like Moise's construction of the pseudoarc in \cite{Moise1948}.  Indeed, while the Irwin-Solecki construction of the pseudoarc has provided a much needed new perspective, simpler combinatorial proofs of Bing's original results have so far remained elusive.

In the present paper we show how Bing's results can indeed be recovered within the framework of \cite{BaBiVi2}.  The key is to first extend the Shimomura theory to more general compacta arising from relations between finite graphs.  Applying this to the pseudoarc, we are indeed able to recover Bing's results stating that the pseudoarc is unique up to homeomorphism and pointwise homogeneous, while also verifying Kwiatkowska's conjecture that the autohomeomorphism group of the pseudoarc has a dense conjugacy class.  However, our methods could be applied much more broadly and, just like the Irwin-Solecki theory before it, have the potential to reveal even more about various other kinds of continua and their dynamical systems.

\section{Posets, Spectra and Continuous Functions}

\subsection{Preliminaries}

Here we review some notation and terminology from \cite{BaBiVi}.

First let us denote the power set of any set $X$ by
\[\mathsf{P}X=\{W:W\subseteq X\}.\]
The product with any other set $Y$ will be denoted by
\[Y\times X=\{\langle y,x\rangle:y\in Y\text{ and }x\in X\}.\]

\subsubsection{Relations}
A \emph{relation} with domain $X$ and codomain $Y$ is a subset of $Y\times X$, i.e. an element of $\mathsf{P}(Y\times X)$.  We use the standard infix notation for relations ${\sqsupset}\subseteq Y\times X$, i.e.
\[y\sqsupset x\qquad\Leftrightarrow\qquad\langle y,x\rangle\in{\sqsupset}.\]
We denote the complement of any relation ${\sqsupset}\subseteq Y\times X$ by ${\not\sqsupset}={\sqsupset^\mathsf{c}}={\sqsupset}\setminus(Y\times X)$, i.e.
\[y\not\sqsupset x\qquad\Leftrightarrow\qquad\neg(y\sqsupset x).\]
We also denote the \emph{inverse} of any relation ${\sqsupset}\subseteq Y\times X$ by ${\sqsubset}={\sqsupset^{-1}}\subseteq X\times Y$, i.e.
\[x\sqsubset y\qquad\Leftrightarrow\qquad y\sqsupset x.\]
The \emph{image} of any $W\subseteq X$ under $\sqsupset$ will be denoted by
\[W^\sqsubset={\sqsupset}[W]=\{y\in Y:\exists w\in W\ (y\sqsupset w)\}.\]
The \emph{preimage} of any $Z\subseteq Y$ under $\sqsupset$ is just the image under the inverse $\sqsubset$, i.e.
\[Z^\sqsupset={\sqsubset}[Z]=\{x\in X:\exists z\in Z\ (z\sqsupset x)\}.\]
The \emph{demonic image} of $W\subseteq X$ and the \emph{demonic preimage} of $Z\subseteq Y$ under $\sqsupset$ are defined as above with $\forall$ instead of $\exists$, which we denote by subscripts rather than superscripts, i.e.
\begin{align*}
    W_\sqsubset&=\{y\in Y:\forall w\in W\ (y\sqsupset w)\}=\{y\in Y:W\subseteq y^\sqsupset\}.\\
    Z_\sqsupset&=\{x\in X:\forall z\in Z\ (z\sqsupset x)\}=\{x\in X:Z\subseteq x^\sqsubset\}.
\end{align*}
The \emph{restriction} of $\sqsupset$ to $Z\times W$ is the relation ${{\sqsupset}|_W^Z}={\sqsupset}\cap (Z\times W)$, i.e.
\[y\mathrel{{\sqsupset}|_W^Z}x\qquad\Leftrightarrow\qquad Z\ni y\sqsupset x\in W.\]
Likewise, the individual restrictions to $W$ and $Z$ respectively will be denoted by
\[{{\sqsupset}|_W}={{\sqsupset}|_W^Y}={\sqsupset}\cap(Y\times W)\qquad\text{and}\qquad{{\sqsupset}|^Z}={{\sqsupset}|_X^Z}={\sqsupset}\cap(Z\times X).\]
We will also use the same symbol $\sqsubset$ to denote the relation on $\mathsf{P}X\times\mathsf{P}Y$ defined by
\[W\sqsubset Z\qquad\Leftrightarrow\qquad W\subseteq Z^\sqsupset\qquad\Leftrightarrow\qquad\forall w\in W\ \exists z\in Z\ (w\sqsubset z).\]
We call $\sqsupset$ \emph{surjective} if $X^\sqsubset=Y$ and \emph{co-surjective} if $Y^\sqsupset=X$.  We call $\sqsupset$ \emph{co-injective} if, for all $y\in Y$, we have $x\in X$ with $x^\sqsubset=\{y\}$ (so co-injectivity implies surjectivity).  We call $\sqsupset$ \emph{co-bijective} if it is both co-surjective and co-injective.

The \emph{composition} of ${\sqni}\subseteq Z\times Y$ and ${\sqsupset}\subseteq Y\times X$ is the relation ${\sqni\circ\sqsupset}\subseteq Z\times X$ where
\[z\sqni\circ\sqsupset x\qquad\Leftrightarrow\qquad\exists y\in Y\ (z\sqni y\sqsupset x)\qquad\Leftrightarrow\qquad z^{\sqni}\cap x^\sqsubset\neq\emptyset.\]
The \emph{demonic composition} is the relation ${\sqni\co\sqsupset}=(\sqni^\mathsf{c}\circ\sqsupset)^\mathsf{c}\cap Z\times Y^\sqsupset$, i.e.
\[z\sqni\co\sqsupset x\qquad\Leftrightarrow\qquad\exists y\sqsupset x\text{ and }\forall y\sqsupset x\ (z\sqni y)\qquad\Leftrightarrow\qquad\emptyset\neq x^\sqsubset\subseteq z^{\sqni}.\]
The \emph{codemonic composition} is the relation ${\sqni\oc\sqsupset}=(\sqin\co\sqsubset)^{-1}$ so
\[z\sqni\oc\sqsupset x\qquad\Leftrightarrow\qquad\exists y\sqin z\text{ and }\forall y\sqin z\ (y\sqsupset x)\qquad\Leftrightarrow\qquad\emptyset\neq z^{\sqni}\subseteq x^\sqsubset.\]

\subsubsection{Posets}\label{Posets}

A relation ${\geq}\subseteq\mathbb{P}\times\mathbb{P}$ is \emph{transitive} if ${\geq\circ\geq}\subseteq{\geq}$.  We denote the equality relation/identity function on $\mathbb{P}$ by $=_\mathbb{P}$ so $\geq$ is \emph{reflexive} if ${=_\mathbb{P}}\subseteq{\geq}$ and \emph{antiysymmetric} if ${\geq}\cap{\leq}\subseteq{=_\mathbb{P}}$.  If $\geq$ is transitive and reflexive then $\geq$ is a \emph{preorder}, while if $\geq$ is also antisymmetric then $\geq$ is a \emph{partial order}.  A \emph{poset} is a set $\mathbb{P}$ on which we have a partial order $\geq_\mathbb{P}$, often written just as $\geq$ when no confusion will result.

Now fix a poset $\mathbb{P}$.  The \emph{strict order} is defined by ${>}={\geq}\cap{\neq}$, i.e.
\[p<q\qquad\Leftrightarrow\qquad p\leq q\neq p.\]
The \emph{common lower bound relation} on a poset $\mathbb{P}$ is defined by ${\wedge}={\geq\circ\leq}$ so
\[p\wedge q\qquad\Leftrightarrow\qquad\exists r\in\mathbb{P}\ (p,q\geq r).\]
For any $C\subseteq\mathbb{P}$, the \emph{$C$-star-below} relation is then defined by ${\vartriangleleft_C}={({\wedge}\circ{=_C})\oc{\leq}}$ so
\[p\vartriangleleft_Cq\qquad\Leftrightarrow\qquad Cp\subseteq q^\geq,\]
where $Cp=C\cap p^\wedge$.  The \emph{$C$-star-above} relation is the inverse ${\vartriangleright_C}=(\vartriangleleft_C)^{-1}={{\geq}\co{({=_C}\circ{\wedge})}}$.

We call $p\in\mathbb{P}$ an \emph{atom} if it is minimal, i.e. if $p^>=\emptyset$.  We call $\mathbb{P}$ \emph{atomless} if it has no atoms.  Likewise, $p\in\mathbb{P}$ is a \emph{coatom} if it is maximal.  We denote the coatoms of $\mathbb{P}$ by
\[\mathbb{P}^0=\{p\in P:p^<=\emptyset\}.\]
The $n^\mathrm{th}$ \emph{cone} of $\mathbb{P}$ is then defined recursively for $n>0$ by
\[\mathbb{P}^n=\{p\in\mathbb{P}:p^<\subseteq\mathbb{P}^{n-1}\}.\]
The atoms of the $n^\mathrm{th}$ cone form the $n^\mathrm{th}$ \emph{level} of $\mathbb{P}$, which we denote by
\[\mathbb{P}_n=\{p\in\mathbb{P}^n:p^>\cap\mathbb{P}^n=\emptyset\}.\]
The \emph{caps} of $\mathbb{P}$ are the subsets refined by levels, which we denote by
\[\mathsf{C}\mathbb{P}=\{C\subseteq\mathbb{P}:\exists n\in\mathbb{P}_n\ (\mathbb{P}_n\leq C)\}.\]
The \emph{star-below} relation is defined by
\[{\vartriangleleft}=\bigcup_{n\in\omega}\vartriangleleft_{\mathbb{P}_n}=\bigcup_{C\in\mathsf{C}\mathbb{P}}\vartriangleleft_C.\]

We call $\mathbb{P}$ an \emph{$\omega$-poset} if every level is finite and $\mathbb{P}=\bigcup_{n\in\omega}\mathbb{P}^n(=\bigcup_{n\in\omega}\mathbb{P}_n)$.
We call an $\omega$-poset $\mathbb{P}$ \emph{regular} if, for all $m\in\omega$, we have $n>m$ such that $\mathbb{P}_n\vartriangleleft\mathbb{P}_m$.  In this case we can make $n$ larger if necessary to even obtain $\mathbb{P}_n\vartriangleleft_{\mathbb{P}_n}\mathbb{P}_m$.  For $m\leq n$ we define
\[{\geq^m_n}={\geq}|^{\mathbb{P}_m}_{\mathbb{P}_n},\qquad{>^m_n}={>}|^{\mathbb{P}_m}_{\mathbb{P}_n}\qquad\text{and}\qquad{\vartriangleright^m_n}={\vartriangleright_{\mathbb{P}_n}}|^{\mathbb{P}_m}_{\mathbb{P}_n}.\]
We call an $\omega$-poset \emph{graded} if its levels are disjoint and ${>^l_n}={>^l_m}\circ{>^m_n}$ whenever $l<m<n$ (which is equivalent to being atomless and graded in the sense of \cite[\S1.5]{BaBiVi}).

\subsection{Spectra}

In \cite{BaBiVi}, a canonical way of constructing compacta from any $\omega$-poset $\mathbb{P}$ was investigated.  To describe this, let us first call $S\subseteq\mathbb{P}$ a \emph{selector} if $S\cap C\neq\emptyset$, for all $C\in\mathsf{C}\mathbb{P}$.  The \emph{spectrum} of $\mathbb{P}$ is then defined by
\[\mathsf{S}\mathbb{P}=\{S\subseteq\mathbb{P}:S\text{ is a minimal selector}\}\]
considered as a subspace of the power-space $\mathsf{P}\mathbb{P}$, i.e.~with the open subbasis $(p_\mathsf{S})_{p\in\mathbb{P}}$ where
\[p_\mathsf{S}=\{S\in\mathsf{S}\mathbb{P}:p\in S\}.\]
If $\mathbb{P}$ is an $\omega$-poset then every $S\in\mathsf{S}\mathbb{P}$ will be a filter and so $(p_\mathsf{S})_{p\in\mathbb{P}}$ will in fact form basis for the topology on $\mathsf{S}\mathbb{P}$.  In \cite{BaBiVi}, it is shown that every second countable $\mathsf{T}_1$ compactum is homeomorphic to $\mathsf{S}\mathbb{P}$, for some $\omega$-poset $\mathbb{P}$.

Several results in \cite{BaBiVi} required the $\omega$-poset $\mathbb{P}$ in question to be \emph{prime} in the sense that $p_\mathsf{S}\neq\emptyset$, for all $p\in\mathbb{P}$.  While this is a relatively weak assumption, it can sometimes be eliminated by making use of the following closed subsets of $\mathsf{S}\mathbb{P}$.

\begin{proposition}
    For any $\omega$-poset $\mathbb{P}$, each $p\in\mathbb{P}$ defines a closed subset of $\mathsf{S}\mathbb{P}$ by
    \[\overline{p_\mathsf{S}}=\{S\in\mathsf{S}\mathbb{P}:S\subseteq p^\wedge\}.\]
    Moreover, $\mathrm{cl}(p_\mathsf{S})\subseteq\overline{p_\mathsf{S}}$ always holds and if $\mathbb{P}$ is prime then $\mathrm{cl}(p_\mathsf{S})=\overline{p_\mathsf{S}}$.
\end{proposition}

\begin{proof}
    If $S\in\mathsf{S}\mathbb{P}\setminus\overline{p_\mathsf{S}}$ then we must have $q\in S\setminus p^\wedge$ so $S\in q_\mathsf{S}$ and $\overline{p_\mathsf{S}}\cap q_\mathsf{S}=\emptyset$, showing that $\mathsf{S}\mathbb{P}\setminus\overline{p_\mathsf{S}}$ is open and hence $\overline{p_\mathsf{S}}$ is closed.  If $S\in\mathrm{cl}(p_\mathsf{S})$ then, for every $q\in S$, we have $T\in p_\mathsf{S}\cap q_\mathsf{S}$ and hence $p\wedge q$, as $T$ is a filter.  Thus $S\subseteq p^\wedge$ and hence $S\in\overline{p_\mathsf{S}}$, showing that $\mathrm{cl}(p_\mathsf{S})\subseteq\overline{p_\mathsf{S}}$.  Finally, if $\mathbb{P}$ is prime and $S\in\overline{p_\mathsf{S}}$ then, for every $q\in S\subseteq p^\wedge$, we have $r\in p^\geq\cap q^\geq$ so $\emptyset\neq r_\mathsf{S}\subseteq p_\mathsf{S}\cap q_\mathsf{S}$, showing that $S\in\mathrm{cl}(p_\mathsf{S})$.
\end{proof}

For example, we can eliminate the prime assumption in \cite[Proposition 2.33]{BaBiVi} and characterise the relation $\vartriangleleft$ on general $\omega$-posets in terms of the spectrum as follows.

\begin{proposition}\label{VarChar}
    For any $\omega$-poset $\mathbb{P}$ and $p,q\in\mathbb{P}$,
    \[p\vartriangleleft q\qquad\Leftrightarrow\qquad\overline{p_\mathsf{S}}\subseteq q_\mathsf{S}.\]
\end{proposition}

\begin{proof}
    If $p\vartriangleleft q$ then, for any $S\in\mathsf{S}\mathbb{P}$,
    \[S\in\overline{p_\mathsf{S}}\quad\Leftrightarrow\quad p\in S_\wedge\quad\Rightarrow\quad q\in (S_\wedge)^\vartriangleleft\subseteq S^\leq\subseteq S\quad\Rightarrow\quad S\in q_\mathsf{S},\]
    showing that $\overline{p_\mathsf{S}}\subseteq q_\mathsf{S}$.  Conversely, if $p\not\vartriangleleft q$ then $Cp\nleq q$, for all $C\in\mathsf{C}\mathbb{P}$.  This means that $Cp\setminus q^\geq$ is a selector and hence contains a minimal selector $S\in\overline{p_\mathsf{S}}\setminus q_\mathsf{S}$, i.e.~$\overline{p_\mathsf{S}}\nsubseteq q_\mathsf{S}$.
\end{proof}

In \cite{BaBiVi}, we in fact showed that all second countable $\mathsf{T}_1$ compacta arise as spectra of $\omega$-posets that are also graded and \emph{predetermined} in the sense that, for any non-atomic $p\in\mathbb{P}$, we have some $q\in\mathbb{P}$ with $q^<=p^\leq$.  Predetermined graded $\omega$-posets will be of particular interest to us as these are precisely the posets formed from sequences of finite co-bijective relations.  With this in mind, let us note that predetermined $\omega$-posets can be characterised via the spectrum as follows.  For any $Q\subseteq\mathbb{P}$, let
\[Q_\mathsf{S}=\bigcup_{q\in Q}q_\mathsf{S}\qquad\text{and}\qquad\overline{Q_\mathsf{S}}=\bigcup_{q\in Q}\overline{q_\mathsf{S}}.\]

\begin{proposition}\label{PreChars}
    For any $\omega$-poset $\mathbb{P}$, the following are equivalent.
    \begin{enumerate}
        \item\label{Pre} $\mathbb{P}$ is predetermined.
        \item\label{pSp} For every $p\in\mathbb{P}$, we have $S\in p_\mathsf{S}$ with $S\subseteq p^\leq\cup p^\geq$.
        \item\label{SubLeq} For any $n\in\omega$, $A\subseteq\mathbb{P}_n$ and $B\subseteq\mathbb{P}^n$,
        \[A_\mathsf{S}\subseteq B_\mathsf{S}\qquad\Rightarrow\qquad A\leq B.\]
    \end{enumerate}
\end{proposition}

\begin{proof}\
    \begin{itemize}
        \item[\eqref{Pre}$\Rightarrow$\eqref{pSp}]  Assume $\mathbb{P}$ is predetermined.  If $p\in\mathbb{P}$ is an atom then $p^\leq$ is already a minimal selector.  If $p\in\mathbb{P}$ is not an atom then, for any $p\in\mathbb{P}$, we can recursively define $(p_n)$ with $p_0=p$ and $p_{k+1}^<=p_k^\leq$, for all $k\in\omega$.  Setting $S=\bigcup_{k\in\omega}p_k^\leq\subseteq p^\leq\cup p^\geq$, we see that $S\cap\mathbb{P}_{\mathsf{r}(p_k)}=\{p_k\}$, where $\mathsf{r}(q)=\min\{r\in\omega:q\in\mathbb{P}^r\}$, and hence $S$ is again a minimal selector.

        \item[\eqref{pSp}$\Rightarrow$\eqref{SubLeq}]  Assume \eqref{pSp} holds and that we have $n\in\omega$, $A\subseteq\mathbb{P}_n$ and $B\subseteq\mathbb{P}^n$ with $A_\mathsf{S}\subseteq B_\mathsf{S}$.  For any $a\in A$, we then have $S\in a_\mathsf{S}\subseteq B_\mathsf{S}$ with $S\subseteq a^\leq\cup a^\geq$.  Taking any $b\in B\cap S\subseteq\mathbb{P}^n$, we know that $b\not<a$, as $a\in\mathbb{P}_n$, so the only remaining possibility is $a\leq b$.  As $a$ was arbitrary, this shows that $A\leq B$.

        \item[\eqref{SubLeq}$\Rightarrow$\eqref{Pre}]  Say $\mathbb{P}$ is not predetermined, so we have some non-atomic $p\in\mathbb{P}$ with $q^<\nsubseteq p^\leq$, for all $q<p$.  As $p$ is not an atom, we have some $n\in\omega$ with $p\in\mathbb{P}_n\setminus\mathbb{P}_{n+1}$.  We claim that $\mathbb{P}_{n+1}\leq B$, where $B=\mathbb{P}^n\setminus p^\leq$.  To see this, take any $q\in\mathbb{P}_{n+1}$.  Then we have $r\in\mathbb{P}_n$ with $q\leq r$.  If $r\neq p$ then $r\ngeq p$ so $r\in B$.  On the other hand, if $r=p$ then $q<p$ so, by our assumption on $p$, we can pick some other $s\in q^<\setminus p^\leq\subseteq B$.  In either case the claim is proved.  In particular, $B\in\mathsf{C}\mathbb{P}$ so $B_\mathsf{S}=\mathsf{S}\mathbb{P}=\mathbb{P}_{n\mathsf{S}}$ even though $p\notin B^\geq$, showing that \eqref{Pre} fails with $A=\mathbb{P}_n$.\qedhere
    \end{itemize}
\end{proof}

In particular, every predetermined $\omega$-poset is prime.  In this case, the spectrum $\mathsf{S}\mathbb{P}$ is Hausdorff and hence metrisable precisely when $\mathbb{P}$ is also regular (see \cite[Corollary 2.40]{BaBiVi}).  If an $\omega$-poset $\mathbb{P}$ is regular then the points of the spectrum can also be characterised in different ways.  First call $R\subseteq\mathbb{P}$ \emph{round} if $R\subseteq R^\vartriangleleft$ and call $U\subseteq\mathbb{P}$ \emph{unbounded} if $U\cap C^\geq\neq\emptyset$, for all $C\in\mathsf{C}\mathbb{P}$.  By \cite[Proposition 2.41]{BaBiVi}, regularity implies that
\[\mathsf{S}\mathbb{P}=\{S\subseteq\mathbb{P}:S\text{ is an unbounded round filter}\}.\]

If we look at different kinds of unbounded subsets, we can obtain a slightly different `clique-spectrum' which is still homeomorphic to $\mathsf{S}\mathbb{P}$.  To describe this, first let us say that $p,q\in\mathbb{P}$ are \emph{adjacent} if, for all $C\in\mathsf{C}\mathbb{P}$, we have $c\in C$ with $p\wedge c\wedge q$ (actually it suffices to consider $C=\mathbb{P}_n$, for $n\in\omega$).  The adjacency relation will be denoted by $\barwedge$ so
\[p\barwedge q\qquad\Leftrightarrow\qquad\forall C\in\mathsf{C}\mathbb{P}\ (p^\wedge\cap C\cap q^\wedge\neq\emptyset).\]
This relation can also be characterised via the spectrum as follows.

\begin{proposition}\label{barwedgeX}
    If $\mathbb{P}$ is an $\omega$-poset then, for all $p,q\in\mathbb{P}$,
    \[p\barwedge q\qquad\Leftrightarrow\qquad\overline{p_\mathsf{S}}\cap\overline{q_\mathsf{S}}\neq\emptyset.\]
\end{proposition}

\begin{proof}
    If $p\barwedge q$ then $p^\wedge\cap q^\wedge$ is a selector and hence contains a minimal selector $S\in\mathsf{S}\mathbb{P}$, by \cite[Proposition 2.2]{BaBiVi}.  Thus $S\in\overline{p_\mathsf{S}}\cap\overline{q_\mathsf{S}}$, proving the $\Rightarrow$ part.

    Conversely, if we have $S\in\overline{p_\mathsf{S}}\cap\overline{q_\mathsf{S}}$ then, for any $C\in\mathsf{C}\mathbb{P}$, we have $c\in C\cap S\subseteq p^\wedge\cap C\cap q^\wedge$, showing that $p\mathrel{\barwedge}q$, proving the $\Leftarrow$ part.
\end{proof}

We observe that $\barwedge$ is weaker than $\wedge$, thanks to \cite[Proposition 2.31]{BaBiVi}, and we can also replace $\wedge$ with $\barwedge$ on the left in \cite[(2.7)]{BaBiVi}, i.e.
\begin{equation}\label{BarwedgeVar}
    {{\barwedge}\circ{\vartriangleleft}}\ \ \subseteq\ {\wedge}\ \ \subseteq\ \ {\barwedge}.
\end{equation}
Indeed, if $p\barwedge q\vartriangleleft_Cr$ then we have $c\in C$ with $p\wedge c\wedge q$ so $p\wedge c\leq r$ and hence $p\wedge r$.

\begin{proposition}\label{RegularBar}
    If $\mathbb{P}$ is a regular $\omega$-poset then, for all $p\in\mathbb{P}$,
    \[\overline{p_\mathsf{S}}=\{S\in\mathsf{S}\mathbb{P}:S\subseteq p^\barwedge\}.\]
\end{proposition}

\begin{proof}
    If $S\in\overline{p_\mathsf{S}}$ then $S\subseteq p^\wedge\subseteq p^\barwedge$.  Conversely, if $\mathbb{P}$ is regular then every $S\in\mathsf{S}\mathbb{P}$ is round so $S\subseteq p^\barwedge$ implies $S\subseteq S^\vartriangleleft\subseteq p^{\barwedge\vartriangleleft}\subseteq p^\wedge$, by \eqref{BarwedgeVar}, and hence $S\in\overline{p_\mathsf{S}}$.
\end{proof}

Here is a further illustration of the close relationship between $\wedge$ and $\barwedge$.

\begin{proposition}\label{DenseCovers}
    For any $\omega$-poset $\mathbb{P}$ and finite $E,F\subseteq\mathbb{P}$,
    \begin{align*}
        E^\wedge=\mathbb{P}\text{ and }E\vartriangleleft F\quad\Rightarrow\quad F\in\mathsf{C}\mathbb{P}\quad\Rightarrow\quad F^\wedge=\mathbb{P}\quad&\Rightarrow\quad\overline{F_\mathsf{S}}=\mathsf{S}\mathbb{P}\quad\Rightarrow\quad F^\barwedge=\mathbb{P}.\\
        \mathbb{P}\text{ is regular and }F^\barwedge=\mathbb{P}\quad&\Rightarrow\quad\overline{F_\mathsf{S}}=\mathsf{S}\mathbb{P}.\\
        \mathbb{P}\text{ is prime and }\overline{F_\mathsf{S}}=\mathsf{S}\mathbb{P}\quad&\Rightarrow\quad F^\wedge=\mathbb{P}.
    \end{align*}
    Consequently, if $\mathbb{P}$ is regular, $E^\barwedge=\mathbb{P}$ and $E\vartriangleleft F$ then again $F\in\mathsf{C}\mathbb{P}$.
\end{proposition}

\begin{proof}
    As $E$ is finite, $E\vartriangleleft F$ implies that we have a single cap $C\in\mathsf{C}\mathbb{P}$ with $E\vartriangleleft_CF$.  Thus $C=E^\wedge\cap C\subseteq F^\geq$ and hence $F\in\mathsf{C}\mathbb{P}$ as well.

    If $F\in\mathsf{C}\mathbb{P}$ then $F^\wedge=\mathbb{P}$ by \cite[Proposition 2.32]{BaBiVi}.
    
    Now assume $F^\wedge=\mathbb{P}$ and take $S\in\mathsf{S}\mathbb{P}$.  As $F$ is finite, we must have $f\in F$ such that $f^\wedge\cap S$ is unbounded.  For any $s\in S$, minimality means we have $C\in\mathsf{C}\mathbb{P}$ with $C\cap S=\{s\}$.  As $f^\wedge\cap S$ is unbounded, we have $t\in C^\geq\cap f^\wedge\cap S$, which means we have $c\in C\cap t^\leq\subseteq C\cap S=\{s\}$, i.e. $s\geq t\mathrel{\wedge}f$ and hence $s\mathrel{\wedge}f$.  This shows that $S\subseteq f^\wedge$ and hence $S\in\overline{f_\mathsf{S}}$, which in turn shows that $\overline{F_\mathsf{S}}=\mathsf{S}\mathbb{P}$.

    Next assume $\overline{F_\mathsf{S}}=\mathsf{S}\mathbb{P}$.  For any $p\in\mathbb{P}$, \Cref{barwedgeX} yields $S\in\overline{p_\mathsf{S}}$ (as $p\mathrel{\barwedge}p$, by \cite[Proposition 2.32]{BaBiVi}).  As $S\in\overline{f_\mathsf{S}}$, for some $f\in F$, \Cref{barwedgeX} again then yields $p\mathrel{\barwedge}f$, thus showing that $F^\barwedge=\mathbb{P}$.

    If $F^\barwedge=\mathbb{P}$ then we can apply the same argument as above show that, for any $S\in\mathsf{S}\mathbb{P}$, we have $f\in F$ with $S\subseteq f^\barwedge$.  If $\mathbb{P}$ is regular, \Cref{RegularBar} then implies $S\in\overline{f_\mathsf{S}}$, thus again showing that $\overline{F_\mathsf{S}}=\mathsf{S}\mathbb{P}$.

    Now say $\mathbb{P}$ is prime and $\overline{F_\mathsf{S}}=\mathsf{S}\mathbb{P}$.  For every $p\in\mathbb{P}$, this means we have $S\in p_\mathsf{S}$.  We then have $f\in F$ with $S\in\overline{f_\mathsf{S}}$ and hence $p\in S\subseteq f^\wedge$, showing that $F^\wedge=\mathbb{P}$.

    Finally say $\mathbb{P}$ is regular, $E^\barwedge=\mathbb{P}$ and $E\vartriangleleft F$.  Then, by what we just showed, $\mathsf{S}\mathbb{P}=\overline{E_\mathsf{S}}\subseteq F_\mathsf{S}$, by \Cref{VarChar}, and hence $F\in\mathsf{C}\mathbb{P}$, by \cite[Proposition 2.8]{BaBiVi}.
\end{proof}

In regular $\omega$-posets we can also replace $\wedge$ in the definition of $\barwedge$ with $\barwedge$ itself.

\begin{proposition}\label{RegularBarwedge}
    If $\mathbb{P}$ is a regular $\omega$-poset then, for all $p,q\in\mathbb{P}$,
    \[p\barwedge q\qquad\Leftrightarrow\qquad\forall C\in\mathsf{C}\mathbb{P}\ (p^\barwedge\cap C\cap q^\barwedge\neq\emptyset).\]
\end{proposition}

\begin{proof}
    For the $\Rightarrow$ part it suffices to note that ${\wedge}\subseteq{\barwedge}$, which follows immediately from \cite[Proposition 2.31]{BaBiVi}.  Conversely, say we have $p,q\in\mathbb{P}$ with $p^\barwedge\cap C\cap q^\barwedge\neq\emptyset$, for all $C\in\mathsf{C}\mathbb{P}$.  For any $C\in\mathsf{C}\mathbb{P}$, regularity yields $B\in\mathsf{C}\mathbb{P}$ with $B\vartriangleleft_BC$.  Thus $p\barwedge b\barwedge q$, for some $b\in B$, and hence we have $d,e\in B$ with $p\wedge d\wedge b\wedge e\wedge q$.  Taking $c\in C$ with $b^\wedge\cap B\leq c$, it follows that $d,e\leq c$ and hence $p\wedge c\wedge q$, thus showing that $p\barwedge q$.
\end{proof}

Let us call $X\subseteq\mathbb{P}$ a \emph{clique} if $p\barwedge q$, for all $p,q\in X$.

\begin{proposition}\label{UnboundedCliques}
    If $U$ is an unbounded clique in a regular $\omega$-poset $\mathbb{P}$ then $U_\barwedge$ is the largest clique containing $U$.  Moreover, $U^\vartriangleleft$ is then the smallest selector contained in $U_\barwedge$.
\end{proposition}

\begin{proof}
    Take any $p,q\in U_\barwedge$.  As $U$ is unbounded, for any $C\in\mathsf{C}\mathbb{P}$ we have $u\in C^\geq\cap U$.  Taking $c\in C$ with $u\leq c$, it follows that $p\barwedge c\barwedge q$.  Thus $p\barwedge q$, by \Cref{RegularBarwedge}, showing that $U_\barwedge$ is a clique.  Also $U\subseteq U_\barwedge$, as $U$ itself is a clique.  Certainly any other clique containing $U$ must be contained in $U_\barwedge$ so it is the largest clique containing $U$.

    To see that $U^\vartriangleleft$ is a selector, take any $C\in\mathsf{C}\mathbb{P}$.  By regularity, we have $B\in\mathsf{C}\mathbb{P}$ with $B\vartriangleleft_BC$.  As $U$ is unbounded, we have $u\in B^\geq\cap U$.  Then we have $b\in B$ and $c\in C$ with $u\leq b\vartriangleleft c$ and hence $c\in C\cap u^\vartriangleleft\subseteq C\cap U^\vartriangleleft$, as required.  By \eqref{BarwedgeVar}, we also know that $p\wedge q$, for all $p,q\in U^\vartriangleleft$.  To show that $U^\vartriangleleft$ is a minimal selector, it thus suffices to show it is round, by \cite[Proposition 2.39]{BaBiVi}, i.e. that $U^\vartriangleleft\subseteq U^{\vartriangleleft\vartriangleleft}$.  Accordingly, say $U\ni u\vartriangleleft_Cp$ and take $A,B\in\mathsf{C}\mathbb{P}$ with $A\vartriangleleft B\vartriangleleft C$.  As $U$ is unbounded, we have $t\in U\cap A^\geq$.  Then we have $a\in A$, $b\in B$ and $c\in C$ with $t\leq a\vartriangleleft b\vartriangleleft c$ and hence $b\in t^\vartriangleleft\subseteq U^\vartriangleleft$.  Also $u\barwedge t\vartriangleleft c$ so $u\wedge c$, by \eqref{BarwedgeVar}, and hence $c\leq p$.  Thus $b\vartriangleleft p$ and hence $p\in U^{\vartriangleleft\vartriangleleft}$, as required.  So $U^\vartriangleleft$ is indeed a minimal selector, which is contained in $U_\wedge\subseteq U_\barwedge$, again by \eqref{BarwedgeVar} and the fact that ${\wedge}\subseteq{\barwedge}$, by \cite[Proposition 2.31]{BaBiVi}.  Likewise, $(U_\barwedge)^\vartriangleleft$ is a minimal selector contained in $U_{\barwedge\barwedge}=U_\barwedge$, which certainly contains $U^\vartriangleleft$ and must therefore be the same as $U^\vartriangleleft$, by minimality.  Any other selector in $U_\barwedge$ contains a minimal selector $S=S^\vartriangleleft\subseteq(U_\barwedge)^\vartriangleleft$, by \cite[Propositions 2.2 and 2.39]{BaBiVi}, and hence minimality again yields $S=(U_\barwedge)^\vartriangleleft=U^\vartriangleleft$.  This shows that $U^\vartriangleleft$ is indeed the smallest selector contained in $U_\barwedge$.
\end{proof}

We define the \emph{clique-spectrum} of $\mathbb{P}$ as the subspace
\[\mathsf{X}\mathbb{P}=\{X\subseteq\mathbb{P}:X\text{ is an unbounded maximal clique}\}\]
of the complementary power-space, i.e.~with the closed subbasis $(p_\mathsf{X})_{p\in\mathbb{P}}$ where
\[p_\mathsf{X}=\{X\in\mathsf{X}\mathbb{P}:p\in X\}.\]

\begin{theorem}\label{SvsX}
    If $\mathbb{P}$ is a regular $\omega$-poset then we have mutually inverse homeomorphisms $\mathsf{s}:\mathsf{X}\mathbb{P}\rightarrow\mathsf{S}\mathbb{P}$ and $\mathsf{x}:\mathsf{S}\mathbb{P}\rightarrow\mathsf{X}\mathbb{P}$ defined by
    \[\mathsf{x}(S)=S_\wedge=S_\barwedge\qquad\text{and}\qquad\mathsf{s}(X)=X^\vartriangleleft.\]
\end{theorem}

\begin{proof}
    If $S\in\mathsf{S}\mathbb{P}$ then, in particular, $S$ is unbounded and hence $\mathsf{x}(S)=S_\barwedge\in\mathsf{X}\mathbb{P}$, by \Cref{UnboundedCliques}.  Moreover, $S_\wedge\subseteq S_\barwedge$, as ${\wedge}\subseteq{\barwedge}$.  Conversely, take any $p\in S_\barwedge$.  For any $s\in S$, we have $t\in S\cap s^\vartriangleright$, as $S$ is round by \cite[Proposition 2.39]{BaBiVi}.  Then $p\barwedge t\vartriangleleft s$ so $p\wedge s$, showing that $p\in S_\wedge$.  This in turn shows that $S_\barwedge\subseteq S_\wedge$ and hence $S_\wedge=S_\barwedge$.  Moreover, roundness again yields $S=S^\vartriangleleft=(S_\barwedge)^\vartriangleleft$, again by \Cref{UnboundedCliques}, so $\mathsf{s}\circ\mathsf{x}=\mathrm{id}_{\mathsf{S}\mathbb{P}}$.

    Similarly, if $X\in\mathsf{X}\mathbb{P}$ then $\mathsf{s}(X)=X^\vartriangleleft\in\mathsf{S}\mathbb{P}$, by \Cref{UnboundedCliques}.  As ${\barwedge}\circ{\vartriangleleft}\subseteq{\wedge}\subseteq{\barwedge}$, we see that $X^\vartriangleleft\subseteq X_\barwedge=X$, by maximality, and hence $\mathsf{X}\mathbb{P}\ni(X^\vartriangleleft)_\barwedge\subseteq X_\barwedge=X$ so $X=(X^\vartriangleleft)_\barwedge$, again by maximalilty.  This shows that $\mathsf{x}\circ\mathsf{s}=\mathrm{id}_{\mathsf{X}\mathbb{P}}$.

    To see that $\mathsf{x}$ is continuous, say $S\in\mathsf{S}\mathbb{P}$ and $S_\barwedge\notin p_\mathsf{X}$.  This means we have $s\in S$ with $p\mathrel{\not\hspace{-3pt}\barwedge}s$.  And this means $S\in s_\mathsf{S}$ and $T_\barwedge\notin p_\mathsf{X}$, for all $T\in s_\mathsf{S}$, thus verifying continuity.  To see that $\mathsf{s}$ is also continuous, say $X\in\mathsf{X}\mathbb{P}$ and $X^\vartriangleleft\in p_\mathsf{S}$.  This means we have $x\in X$ and $C\in\mathsf{C}\mathbb{P}$ with $x\vartriangleleft_Cp$.  By regularity, we then have finite $B\in\mathsf{C}\mathbb{P}$ with $B\vartriangleleft C$.  Let $O=\bigcap_{b\in B\setminus x^\barwedge}\mathsf{X}\mathbb{P}\setminus b_\mathsf{X}$, which is an open subset of $\mathsf{X}\mathbb{P}$ containing $X$.  Any $Y\in O$ is a selector (containing the minimal selector $Y^\vartriangleleft$, for example) and hence we must have some $b\in Y\cap B\cap x^\barwedge$.  Taking any $c\in C$ with $b\vartriangleleft c$, it follows that $x\wedge c$ and hence $c\leq p$.  Thus $p\in b^\vartriangleleft\subseteq Y^\vartriangleleft$, i.e.~$Y^\vartriangleleft\in p_\mathsf{S}$, showing that $\mathsf{s}[O]\subseteq p_\mathsf{S}$.  Thus $\mathsf{s}$ is also continuous so both $\mathsf{x}$ and $\mathsf{s}$ are indeed homeomorphisms.
\end{proof}

In particular, if $\mathbb{P}$ is a regular $\omega$-poset then $\mathsf{x}[\overline{p_\mathsf{S}}]=\{S_\wedge:p^\wedge\supseteq S\in\mathsf{S}\mathbb{P}\}=p_\mathsf{X}$ and hence
\[\overline{p_\mathsf{S}}=\mathsf{s}[p_\mathsf{X}],\]
for all $p\in\mathbb{P}$.  Thus in this case $(\overline{p_\mathsf{S}})_{p\in\mathbb{P}}$ constitutes a closed subbasis for $\mathsf{S}\mathbb{P}$.

\subsection{Refiners}\label{Refiners}

Continuous functions between spectra of regular $\omega$-posets $\mathbb{P}$ and $\mathbb{Q}$ can be encoded by certain relations between $\mathbb{P}$ and $\mathbb{Q}$.  Specifically, let us call ${\sqsupset}\subseteq\mathbb{Q}\times\mathbb{P}$ a \emph{refiner} if, for every $C\in\mathsf{C}\mathbb{Q}$, we have $B\in\mathsf{C}\mathbb{P}$ with $B\sqsubset C$.  We call $\sqsupset$ \emph{$\barwedge$-preserving} if
\[\tag{$\barwedge$-preserving}{\sqsupset}\circ{\barwedge}\circ{\sqsubset}\ \ \subseteq\ \ {\barwedge}.\]
Likewise $\sqsupset$ is \emph{$\wedge$-preserving} if we can replace $\barwedge$ above with $\wedge$.  These $\wedge$-preserving refiners were the relations considered in \cite{BaBiVi}, but here we show how to extend the theory to $\barwedge$-preserving refiners.  First we note that these are indeed more general.
\begin{proposition}
    Any $\wedge$-preserving refiner ${\sqsupset}\subseteq\mathbb{Q}\times\mathbb{P}$ is automatically $\barwedge$-preserving.
\end{proposition}

\begin{proof}
    Say $q\sqsupset p\mathrel{\barwedge}p'\sqsubset q'$ and take any $C\in\mathsf{C}\mathbb{Q}$.  As $\sqsupset$ is a refiner, we have $B\in\mathsf{C}\mathbb{P}$ with $B\sqsubset C$.  As $p\mathrel{\barwedge}p'$, we have $b\in B$ with $p\mathrel{\wedge}b\mathrel{\wedge}p'$.  Then we have $c\in C$ with $b\leq c$ and hence $q\mathrel{\wedge}c\mathrel{\wedge}q'$, as $\sqsupset$ is $\wedge$-preserving.  This shows that $q\mathrel{\barwedge}q'$, as required.
\end{proof}

Continuous maps can be defined from $\barwedge$-preserving refiners as follows.

\begin{proposition}
    For any $\barwedge$-preserving refiner ${\sqsupseteq}\subseteq\mathbb{Q}\times\mathbb{P}$ between regular $\omega$-posets $\mathbb{P}$ and $\mathbb{Q}$, we have a continuous map $\mathsf{X}_\sqsupseteq:\mathsf{X}\mathbb{P}\rightarrow\mathsf{X}\mathbb{Q}$ defined by
    \[\mathsf{X}_\sqsupseteq(Y)=(Y^\sqsubseteq)_\barwedge.\]
\end{proposition}

\begin{proof}
    For any $Y\in\mathsf{X}\mathbb{P}$, we know that $Y^\sqsubseteq$ is a clique, as $\sqsupseteq$ is $\barwedge$-preserving.  Also, as $\sqsupseteq$ is a refiner, for any $C\in\mathsf{C}\mathbb{Q}$, we have $B\in\mathsf{C}\mathbb{P}$ with $B\sqsubseteq C$.  As $Y$ is a selector, we then have $b\in B\cap Y$.  Taking any $c\in C$ with $b\sqsubseteq c$, it follows that $c\in C\cap Y^\sqsubseteq$.  This shows $Y^\sqsubseteq$ is also a selector.  In particular, $Y^\sqsubseteq$ is unbounded so $(Y^\sqsubseteq)_\barwedge\in\mathsf{X}\mathbb{Q}$, by \Cref{UnboundedCliques}, showing that $\mathsf{X}_\sqsupseteq$ does indeed map $\mathsf{X}\mathbb{P}$ to $\mathsf{X}\mathbb{Q}$.

    To prove continuity, take any $Y\in\mathsf{X}\mathbb{P}$ and $q\in\mathbb{Q}$ such that $\mathsf{X}_\sqsupseteq(Y)\notin q_\mathsf{X}$.  This means we have some $r\in Y^\sqsubseteq$ with $q\mathrel{\not\hspace{-3pt}\barwedge}r$ and hence we have some $C\in\mathsf{C}\mathbb{Q}$ with $q^\barwedge\cap C\cap r^\barwedge=\emptyset$, by \Cref{RegularBarwedge}.  Taking finite $B\in\mathsf{C}\mathbb{P}$ with $B\sqsubseteq C$, it follows that
    \[q^\barwedge\cap C\cap (B\cap Y)^\sqsubseteq\subseteq q^\barwedge\cap C\cap r^\barwedge=\emptyset.\]
    Note that the open set $O=\bigcap_{b\in B\setminus Y}(\mathsf{X}\mathbb{P}\setminus b_\mathsf{X})$ includes $Y$.  Moreover, any other $Z\in O$ is a selector and thus contains some $b\in B\cap Y$.  Taking any $c\in C$ with $b\sqsubseteq c$, it follows that $q\mathrel{\not\hspace{-3pt}\barwedge}c$ and hence $q\notin c_\barwedge\supseteq (Z^\sqsubseteq)_\barwedge=\mathsf{X}_\sqsupseteq(Z)$.  In other words $Y\in O\subseteq\mathsf{X}_\sqsupseteq^{-1}[\mathsf{X}\mathbb{Q}\setminus q_\mathsf{X}]$, showing that $\mathsf{X}_\sqsupseteq$ is indeed continuous.
\end{proof}

\begin{corollary}
    The map $\mathsf{X}$ above defines a functor from $\mathbf{R}$ to $\mathbf{C}$ where
    \begin{enumerate}
        \item $\mathbf{R}$ is the category of regular $\omega$-posets with $\barwedge$-preserving refiners as morphisms.
        \item $\mathbf{C}$ is the category of metrisable compacta with continuous maps as morphisms.
    \end{enumerate}
\end{corollary}

\begin{proof}
    To verify that $\mathsf{X}$ is a functor we just have to show that, for any regular $\omega$-posets $\mathbb{P}$, $\mathbb{Q}$ and $\mathbb{R}$ and $\barwedge$-preserving refiners ${\sqsupseteq}\subseteq\mathbb{Q}\times\mathbb{P}$ and ${\sqsupset}\subseteq\mathbb{R}\times\mathbb{Q}$,
    \[\mathsf{X}_{\sqsupset\circ\sqsupseteq}=\mathsf{X}_\sqsupset\circ\mathsf{X}_\sqsupseteq.\]
    To see this, note $Y\in\mathsf{X}\mathbb{P}$ implies $Y^\sqsubseteq\subseteq (Y^\sqsubseteq)_\barwedge=\mathsf{X}_\sqsupseteq(Y)$ so $Y^{\sqsubseteq\sqsubset}\subseteq\mathsf{X}_\sqsupseteq(Y)^\sqsubset$ and hence $\mathsf{X}_{\sqsupset\circ\sqsupseteq}(Y)=(Y^{\sqsubseteq\sqsubset})_\barwedge\supseteq(\mathsf{X}_\sqsupseteq(Y)^\sqsubset)_\barwedge=\mathsf{X}_\sqsupset(\mathsf{X}_\sqsupseteq(Y))$.  Now equality follows, by maximality.
\end{proof}

The following different but isomorphic functor $\mathsf{S}$ was investigated in \cite[\S3]{BaBiVi}, at least on the slightly smaller class of $\wedge$-preserving refiners.

\begin{corollary}
    We have another functor $\mathsf{S}:\mathbf{R}\rightarrow\mathbf{C}$ taking each regular $\omega$-poset $\mathbb{P}$ to $\mathsf{S}\mathbb{P}$ and each $\barwedge$-preserving refiner ${\sqsupseteq}\subseteq\mathbb{Q}\times\mathbb{P}$ to $\mathsf{S}_\sqsupseteq$ where
    \[\mathsf{S}_\sqsupseteq(Y)=Y^{\sqsubseteq\vartriangleleft}.\]
    Furthermore, $\mathsf{s}$ and $\mathsf{x}$ are natural isomorphisms between $\mathsf{S}$ and $\mathsf{X}$.
\end{corollary}

\begin{proof}
    For any regular $\omega$-posets $\mathbb{P}$ and $\mathbb{Q}$, any $\barwedge$-preserving refiner ${\sqsupseteq}\subseteq\mathbb{Q}\times\mathbb{P}$ and any $S\in\mathsf{S}\mathbb{P}$, \Cref{UnboundedCliques} tells us that
    \[\mathsf{s}\circ\mathsf{X}_\sqsupseteq\circ\mathsf{x}(S)=(((S_\barwedge)^\sqsubseteq)_\barwedge)^\vartriangleleft=((S^\sqsubseteq)_\barwedge)^\vartriangleleft=S^{\sqsubseteq\vartriangleleft}=\mathsf{S}_\sqsupseteq(S).\]
    Thus $\mathsf{S}_\sqsupseteq=\mathsf{s}\circ\mathsf{X}_\sqsupseteq\circ\mathsf{x}$ from which it follows that $\mathsf{S}$ is a functor which is naturally isomorphic to $\mathsf{X}$ via $\mathsf{s}:\mathsf{X}\rightarrow\mathsf{S}$ and $\mathsf{x}:\mathsf{S}\rightarrow\mathsf{X}$.
\end{proof}

We note that these functors are far from being faithful, i.e.~the same continuous map can arise from various different $\barwedge$-preserving refiners.  For example, extending or reducing a given $\barwedge$-preserving refiner to another one will not change the resulting continuous map, i.e.~if ${\sqsupset},{\sqsupseteq}\subseteq\mathbb{Q}\times\mathbb{P}$ are $\barwedge$-preserving refiners between regular $\omega$-posets $\mathbb{P}$ and $\mathbb{Q}$ then
\begin{equation}\label{Subrefiners}
    {\sqsupset}\ \subseteq\ {\sqsupseteq}\qquad\Rightarrow\qquad\mathsf{S}_\sqsupset=\mathsf{S}_\sqsupseteq.
\end{equation}
Indeed, if ${\sqsupset}\subseteq{\sqsupseteq}$ then, for any $S\in\mathsf{S}\mathbb{P}$, it follows that $\mathsf{S}_\sqsupset(S)\subseteq\mathsf{S}_\sqsupseteq(S)$ and hence $\mathsf{S}_\sqsupset(S)=\mathsf{S}_\sqsupseteq(S)$, by the minimality of the selectors in $\mathsf{S}\mathbb{Q}$.

However they are full, as we can show by associating an appropriate $\barwedge$-preserving refiner to each continuous map $\phi:\mathsf{S}\mathbb{P}\rightarrow\mathsf{S}\mathbb{Q}$.  In fact, there are several canonical options.  Here we consider the relations $\sqsupset_\phi$ and $\sqsupseteq_\phi$ defined on $\mathbb{Q}\times\mathbb{P}$ by
\begin{align*}
    q\sqsupseteq_\phi p\qquad&\Leftrightarrow\qquad\overline{q_\mathsf{S}}\supseteq\phi[\overline{p_\mathsf{S}}].\\
    q\sqsupset_\phi p\qquad&\Leftrightarrow\qquad q_\mathsf{S}\supseteq\phi[\overline{p_\mathsf{S}}].
\end{align*}
For general $\omega$-posets $\mathbb{P}$ and $\mathbb{Q}$ we see that, setting ${\sqsubseteq_\phi}={\sqsupseteq_\phi^{-1}}$ and ${\sqsubset_\phi}={\sqsupset_\phi^{-1}}$,
\begin{align*}
    {{\sqsupseteq_\phi}\circ{\barwedge}\circ{\sqsubseteq_\phi}}\ \ &\subseteq\ \ {\barwedge}.\\
    {{\sqsupset_\phi}\circ{\barwedge}\circ{\sqsubset_\phi}}\ \ &\subseteq\ \ {\wedge}.
\end{align*}
Indeed, if $p\mathrel{\barwedge}o$ then \Cref{barwedgeX} yields $S\in\overline{p_\mathsf{S}}\cap\overline{o_\mathsf{S}}$.  If $q\sqsupseteq_\phi p$ and $r\sqsupseteq_\phi o$ then $\phi(S)\in\overline{q_\mathsf{S}}\cap\overline{r_\mathsf{S}}$ and hence $q\mathrel{\barwedge}r$, again by \Cref{barwedgeX}.  On the other hand, if $q\sqsupset_\phi p$ and $r\sqsupset_\phi o$ then $\phi(S)\in q_\mathsf{S}\cap r_\mathsf{S}$ and hence $q\wedge r$, as $\phi(S)$ is a filter.  So $\sqsupseteq_\phi$ is $\barwedge$-preserving and $\sqsupset_\phi$ is $\wedge$-preserving.  When $\mathbb{P}$ is regular, these relations are also refiners.

\begin{proposition}\label{SqsupseteqPhi}
    Say we are given $\omega$-posets $\mathbb{P}$ and $\mathbb{Q}$ and continuous $\phi:\mathsf{S}\mathbb{P}\rightarrow\mathsf{S}\mathbb{Q}$.  If $\mathbb{P}$ is regular then $\sqsupset_\phi$ and $\sqsupseteq_\phi$ are refiners.  If $\mathbb{Q}$ is also regular then $\phi=\mathsf{S}_{\sqsupseteq_\phi}=\mathsf{S}_{\sqsupset_\phi}$.
\end{proposition}

\begin{proof}
    Assume $\mathbb{P}$ is regular.  As ${\sqsupset_\phi}\subseteq{\sqsupseteq_\phi}$, it suffices to show that $\sqsupset_\phi$ is a refiner.  Accordingly, take any $C\in\mathsf{C}\mathbb{Q}$.  As $\{c_\mathsf{S}:c\in C\}$ is an open cover of $\mathsf{S}\mathbb{Q}$, we have an open cover of $\mathsf{S}\mathbb{P}$ given by $\{\phi^{-1}[c_\mathsf{S}]:c\in C\}$, which is thus refined by an open cover of the form $\{b_\mathsf{S}:b\in B\}$, for some $B\in\mathsf{C}\mathbb{P}$.  As $\mathbb{P}$ is regular, we have some $A\in\mathsf{C}\mathbb{P}$ with $A\vartriangleleft B$.  For every $a\in A$, we therefore have $b\in B$ with $a\vartriangleleft b$ and hence $\overline{a_\mathsf{S}}\subseteq b_\mathsf{S}$, by \Cref{VarChar}.  We then also have $c\in C$ with $b_\mathsf{S}\subseteq\phi^{-1}[c_\mathsf{S}]$ and hence $\phi[\overline{a_\mathsf{S}}]\subseteq c_\mathsf{S}$, i.e.~$a\sqsubset_\phi c$.  This shows that $A\sqsubset_\phi C$, which in turn shows that $\sqsupset_\phi$ is a refiner.

    Now if $\mathsf{S}\mathbb{P}\ni S\ni p\sqsubseteq_\phi q\vartriangleleft r$ then $S\in p_\mathsf{S}\subseteq\overline{p_\mathsf{S}}\subseteq\phi^{-1}[\overline{q_\mathsf{S}}]\subseteq\phi^{-1}[r_\mathsf{S}]$ and hence $r\in\phi(S)$, showing that $S^{\sqsubseteq_\phi\vartriangleleft}\subseteq\phi(S)$.  If $\mathbb{Q}$ is also regular then $S^{\sqsubseteq_\phi\vartriangleleft}\in\mathsf{S}\mathbb{Q}$ and hence $S^{\sqsubseteq_\phi\vartriangleleft}=\phi(S)$, by minimality, showing that $\mathsf{S}_{\sqsupseteq_\phi}=\phi$.  As ${\sqsupset_\phi}\subseteq{\sqsupseteq_\phi}$, minimality again yields $\mathsf{S}_{\sqsupset_\phi}=\phi$.
\end{proof}

What distinguishes ${\sqsupseteq_\phi}$ here is that it is actually the largest $\barwedge$-preserving refiner mapped to $\phi$ by $\mathsf{S}$.  Put another way, ${\sqsupseteq}\subseteq{\sqsupseteq_{\mathsf{S}_\sqsupseteq}}$, for every $\barwedge$-preserving refiner $\sqsupseteq$.  We can also express ${\sqsupseteq_{\mathsf{S}_\sqsupseteq}}$ in terms of $\barwedge$, $\sqsupseteq$ and $\wedge$ as follows.

\begin{proposition}\label{sqXsq}
    For any $\barwedge$-preserving refiner ${\sqsupseteq}\subseteq\mathbb{Q}\times\mathbb{P}$ between regular $\omega$-posets,
    \[{\sqsupseteq}\ \subseteq\ {{\barwedge}\co{({\sqsupseteq}\circ{\barwedge})}}\ =\ {{\barwedge}\co{({\sqsupseteq}\circ{\wedge})}}\ =\ {\sqsupseteq_{\mathsf{S}_\sqsupseteq}}.\]
\end{proposition}

\begin{proof}
    First note ${\sqsupseteq}\subseteq{{\barwedge}\co{({\sqsupseteq}\circ{\barwedge})}}$ is just saying that ${\sqsupseteq}\subseteq\mathbb{Q}\times\mathbb{P}$ is $\barwedge$-preserving.

    Now say $q\not\sqsupseteq_{\mathsf{S}_\sqsupseteq}p$, i.e.~$\overline{q_\mathsf{S}}\nsupseteq\mathsf{S}_\sqsupseteq[\overline{p_\mathsf{S}}]$ so we have $T\in \overline{p_\mathsf{S}}$ with $T^{\sqsubseteq\vartriangleleft}=\mathsf{S}_\sqsupseteq(T)\nsubseteq q^\wedge$ and hence $T^\sqsubseteq\nsubseteq q^\barwedge$, by \eqref{BarwedgeVar}.  Thus we have $t\in T\subseteq p^\wedge$ with $p\mathrel{\wedge}t\sqsubseteq r\mathrel{\not\hspace{-3pt}\barwedge}q$, for some $r$, showing that $(q,p)\notin{{\barwedge}\co{({\sqsupseteq}\circ{\wedge})}}$ and hence
    \[{{\barwedge}\co{({\sqsupseteq}\circ{\barwedge})}}\ \subseteq\ {{\barwedge}\co{({\sqsupseteq}\circ{\wedge})}}\ \subseteq\ {\sqsupseteq_{\mathsf{S}_\sqsupseteq}}.\]
    
    Conversely, if $p\mathrel{\barwedge}t\sqsubseteq r\mathrel{\not\hspace{-3pt}\barwedge}q$ then \Cref{barwedgeX} yields $T\in\overline{p_\mathsf{S}}\cap\overline{t_\mathsf{S}}$.  Taking $C\in\mathsf{C}\mathbb{Q}$ with $r^\wedge\cap C\cap q^\wedge=\emptyset$, we then have $c\in C\cap T^{\sqsubseteq\vartriangleleft}\subseteq((t^\sqsubseteq)_\barwedge)^\vartriangleleft\subseteq r^{\barwedge\vartriangleleft}\subseteq r^\wedge$, by \eqref{BarwedgeVar}.  Thus $c\in\mathsf{S}_\sqsupseteq(T)\setminus q^\wedge$ so $\mathsf{S}_\sqsupseteq(T)\in\mathsf{S}_\sqsupseteq[\overline{p_\mathsf{S}}]\setminus\overline{q_\mathsf{S}}$, showing that $\mathsf{S}_\sqsupseteq[\overline{p_\mathsf{S}}]\nsubseteq\overline{q_\mathsf{S}}$ and hence $q\not\sqsupseteq_{\mathsf{S}_\sqsupseteq}p$.
\end{proof}

Thus if we wanted to make $\mathsf{S}$ faithful, we could proceed as follows.  First call $\sqsupseteq$ a \emph{weak refiner} if $\sqsupseteq$ is a refiner satisfying ${\sqsupseteq}={{\barwedge}\co{({\sqsupseteq}\circ{\barwedge})}}$.  Let $\mathbf{W}$ be the category of regular $\omega$-posets with weak refiners as morphisms under a composition operation $\overline{\circ}$ defined by
\[{\sqsupseteq\overline{\circ}\sqsupset}\ =\ {{\barwedge}\co{({\sqsupseteq\circ\sqsupset}\circ{\barwedge})}}.\]
Then $\mathsf{S}$ is a fully faithful functor from $\mathbf{W}$ to $\mathbf{C}$, as can be shown via \Cref{sqXsq}.  In this way we see that the category $\mathbf{C}$ is equivalent to the more combinatorial category $\mathbf{W}$.

\subsection{Strong Refiners}

On the other hand, what distinguishes $\sqsupset_\phi$ above is that it is a `strong refiner', like those considered in \cite{BaBiVi}.  Indeed, these strong refiners are more closely related to the natural topology on continuous functions that we are interested in.

First, for any relation ${\sqsupset}\subseteq\mathbb{Q}\times\mathbb{P}$ between $\omega$-posets $\mathbb{P}$ and $\mathbb{Q}$, let us define another relation ${\sqsupset^*}\subseteq\mathbb{Q}\times\mathbb{P}$ by
\[q\sqsupset^*p\qquad\Leftrightarrow\qquad\exists C\in\mathsf{C}\mathbb{P}\ (Cp\sqsubset q)\]
(where again $Cp=C\cap p^\wedge$).  We define \emph{star-composition} with any other ${\sqni}\subseteq\mathbb{R}\times\mathbb{Q}$ by
\[{\sqni*\sqsupset}=(\sqni\circ\sqsupset)^*.\]

\begin{proposition}\label{StarVar}
    If $\mathbb{P}$ and $\mathbb{Q}$ are regular $\omega$-posets then, for any $\barwedge$-preserving refiner ${\sqsupseteq}\subseteq\mathbb{Q}\times\mathbb{P}$,
    \[{\vartriangleright*\sqsupseteq}={\sqsupset_{\mathsf{S}_\sqsupseteq}}.\]
\end{proposition}

\begin{proof}
    Say $q\vartriangleright*\sqsupseteq p$, so we have $C\in\mathsf{C}\mathbb{P}$ with $Cp\sqsubseteq\circ\vartriangleleft q$.  For any $S\in\overline{p_\mathsf{S}}$, we have $c\in S\cap C\subseteq Cp$ and then $q\in c^{\sqsubseteq\vartriangleleft}\subseteq S^{\sqsubseteq\vartriangleleft}=\mathsf{S}_\sqsupseteq(S)$, showing that $q_\mathsf{S}\supseteq\mathsf{S}_\sqsupseteq[\overline{p_\mathsf{S}}]$, i.e.~$q\sqsupset_{\mathsf{S}_\sqsupseteq}p$.

    Conversely, if $q\vartriangleright*\sqsupseteq p$ does not hold then $Cp\nsubseteq q^{\vartriangleright\sqsupseteq}$, for every $C\in\mathsf{C}\mathbb{P}$.  This means $p^\wedge\setminus q^{\vartriangleright\sqsupseteq}$ is a selector and hence contains a minimal selector $S\in\overline{p_\mathsf{S}}$, as $S\subseteq p^\wedge$.  Also $q\notin S^{\sqsubseteq\vartriangleleft}=\mathsf{S}_\sqsupseteq(S)$ and hence $\mathsf{S}_\sqsupseteq(S)\in\mathsf{S}_\sqsupseteq[\overline{p_\mathsf{S}}]\setminus q_\mathsf{S}$, showing that $q_\mathsf{S}\nsupseteq\mathsf{S}_\sqsupseteq[\overline{p_\mathsf{S}}]$, i.e.~$q\not\sqsupset_{\mathsf{S}_\sqsupseteq}p$.
\end{proof}

Let us call ${\sqsupset}\subseteq\mathbb{Q}\times\mathbb{P}$ a \emph{strong refiner} if $\sqsupset$ is a $\wedge$-preserving refiner satisfying ${\sqsupset}={\vartriangleright*\sqsupset}$.

\begin{corollary}\label{UniqueStrong}
    For any regular $\omega$-posets $\mathbb{P}$ and $\mathbb{Q}$ and any continuous $\phi:\mathsf{S}\mathbb{P}\rightarrow\mathsf{S}\mathbb{Q}$, ${\sqsupset_\phi}$ is the unique strong refiner mapped to $\phi$ by $\mathsf{S}$.
\end{corollary}

\begin{proof}
    By \Cref{SqsupseteqPhi,StarVar}, ${\sqsupset_\phi*\vartriangleright}={\sqsupset_{\mathsf{S}_{\sqsupset_\phi}}}={\sqsupset_\phi}$ so $\sqsupset_\phi$ is indeed a strong refiner.  Also, for any $\barwedge$-preserving refiner ${\sqsupseteq}\subseteq\mathbb{Q}\times\mathbb{P}$, \Cref{SqsupseteqPhi,StarVar} yield
    \[\mathsf{S}_{\vartriangleright*\sqsupseteq}=\mathsf{S}_{\sqsupset_{\mathsf{S}_\sqsupseteq}}=\mathsf{S}_\sqsupseteq.\]
    Thus, for any $\barwedge$-preserving refiners ${\sqsupset},{\sqsupseteq}\subseteq\mathbb{Q}\times\mathbb{P}$,
    \[{\vartriangleright*\sqsupset}={\vartriangleright*\sqsupseteq}\qquad\Leftrightarrow\qquad\mathsf{S}_\sqsupset=\mathsf{S}_\sqsupseteq.\]
    Indeed, ${\vartriangleright*\sqsupset}={\vartriangleright*\sqsupseteq}$ implies $\mathsf{S}_\sqsupset=\mathsf{S}_{\vartriangleright*\sqsupset}=\mathsf{S}_{\vartriangleright*\sqsupseteq}=\mathsf{S}_\sqsupseteq$, while the converse is immediate from \Cref{StarVar}.  This means that any strong refiner $\sqsupset$ with $\mathsf{S}_\sqsupset=\phi=\mathsf{S}_{\sqsupset_\phi}$ must satisfy ${\sqsupset}={\vartriangleright*\sqsupset}={\vartriangleright*\sqsupset_\phi}={\sqsupset_\phi}$, showing that $\sqsupset_\phi$ is unique.
\end{proof}

Star-composition of strong refiners between regular $\omega$-posets is associative, by \cite[Proposition 3.24]{BaBiVi}.  Also $\vartriangleright$ is a strong refiner on any regular $\omega$-poset $\mathbb{P}$ which is an identity with respect to star-composition, as noted in \cite[\S3.3]{BaBiVi}.  We thus have a category $\mathbf{S}$ of regular $\omega$-posets with strong refiners as morphisms under star-composition.  Moreover, invoking \cite[Proposition 3.17]{BaBiVi} again, we see that $\mathsf{S}_{\sqsupset*\sqni}=\mathsf{S}_{\sqsupset\circ\sqni}=\mathsf{S}_\sqsupset\circ\mathsf{S}_{\sqni}$. 
 So $\mathsf{S}|_\mathbf{S}$ is still a functor from $\mathbf{S}$ to $\mathbf{C}$, a fully faithful functor now which thus witnesses the categorical equivalence of $\mathbf{S}$ and $\mathbf{C}$.  Also, as at the start of \cite[Theorem 3.25]{BaBiVi}, we see that ${\sqsupseteq}\mapsto{\vartriangleleft*\sqsupseteq}$ defines a quotient functor $\mathsf{Q}:\mathbf{R}\rightarrow\mathbf{S}$ such that $\mathsf{S}=\mathsf{S}|_\mathbf{S}\circ\mathsf{Q}$.

Strong refiners and continuous functions also carry natural topologies.  Specifically, for any regular $\omega$-posets $\mathbb{P}$ and $\mathbb{Q}$, we consider
\[\mathbf{S}_\mathbb{P}^\mathbb{Q}=\{{\sqsupset}\subseteq\mathbb{Q}\times\mathbb{P}:{\sqsupset}\text{ is a strong refiner}\}\]
again as a subspace of the power space $\mathsf{P}(\mathbb{Q}\times\mathbb{P})$, i.e.~with the subbasis $(\langle q,p\rangle^\mathbf{S})_{p\in\mathbb{P}}^{q\in\mathbb{Q}}$ where
\[{\langle q,p\rangle^\mathbf{S}}=\{{\sqsupset}\in\mathbf{S}_\mathbb{P}^\mathbb{Q}:q\sqsupset p\}.\]
Each finite ${\sqni}\subseteq\mathbb{Q}\times\mathbb{P}$ determines an open set in $\mathbf{S}_\mathbb{P}^\mathbb{Q}$ defined by
\begin{equation}\label{sqniS}
    {\sqni^\mathbf{S}}=\bigcap_{q\sqni p}\langle q,p\rangle^\mathbf{S}=\{{\sqsupset}\in\mathbf{S}_\mathbb{P}^\mathbb{Q}:{\sqni}\subseteq{\sqsupset}\}
\end{equation}
and these then form a basis for the topology on $\mathbf{S}_\mathbb{P}^\mathbb{Q}$.

On the other hand, for any topological spaces $X$ and $Y$, we consider
\[\mathbf{C}_X^Y=\{\phi\in Y^X:\phi\text{ is continuous}\}\]
as a subspace of the Vietoris hyperspace of closed subsets of $Y\times X$.  Then every relation ${\leftarrow}\subseteq\mathbb{Q}\times\mathbb{P}$ determines an open set in $\mathsf{S}\mathbb{Q}\times\mathsf{S}\mathbb{P}$ given by
\[{\leftarrow_\mathsf{S}}=\bigcup_{q\leftarrow p}q_\mathsf{S}\times p_\mathsf{S},\]
which turn defines an open set in $\mathbf{C}_{\mathsf{S}\mathbb{P}}^{\mathsf{S}\mathbb{Q}}$ given by
\begin{equation}\label{leftarrowC}
    {\leftarrow_\mathbf{C}}=\{\phi\in\mathbf{C}_{\mathsf{S}\mathbb{P}}^{\mathsf{S}\mathbb{Q}}:\phi\subseteq{\leftarrow_\mathsf{S}}\}=\{\phi\in\mathbf{C}_{\mathsf{S}\mathbb{P}}^{\mathsf{S}\mathbb{Q}}:\forall S\in\mathsf{S}\mathbb{P}\ \exists p,q\ (\phi(S)\ni q\leftarrow p\in S)\},
\end{equation}
and these again form a basis for the topology on $\mathbf{C}_{\mathsf{S}\mathbb{P}}^{\mathsf{S}\mathbb{Q}}$.

\begin{proposition}
    The map ${\sqsupset}\mapsto\mathsf{S}_\sqsupset$ is a homeomorphism from $\mathbf{S}_\mathbb{P}^\mathbb{Q}$ onto $\mathbf{C}_{\mathsf{S}\mathbb{P}}^{\mathsf{S}\mathbb{Q}}$.
\end{proposition}

\begin{proof}
    By \Cref{StarVar} and \Cref{UniqueStrong}, ${\sqsupset}\mapsto\mathsf{S}_\sqsupset$ is a homeomorphism from $\mathbf{R}_\mathbb{P}^\mathbb{Q}$ onto $\mathbf{C}_{\mathsf{S}\mathbb{P}}^{\mathsf{S}\mathbb{Q}}$ with the compact-open topology.  But as $\mathsf{S}\mathbb{P}$ is compact, this coincides with the Vietoris topology on $\mathbf{C}_{\mathsf{S}\mathbb{P}}^{\mathsf{S}\mathbb{Q}}$ (which also coincides with the topology of uniform convergence with respect to any metric compatible with the topology of $\mathsf{S}\mathbb{Q}$).
\end{proof}

Before moving on, let us make a couple more simple observations about $\leftarrow_\mathsf{S}$.

\begin{proposition}\label{SubfactorOpen}
    For any continuous $\phi:\mathsf{S}\mathbb{P}\rightarrow\mathsf{S}\mathbb{Q}$, ${\leftarrow}\subseteq{\mathbb{Q}\times\mathbb{P}}$ and $C\in\mathsf{C}\mathbb{P}$,
    \[{=_C}\subseteq{{\sqsubset_\phi}\circ{\leftarrow}\circ{\geq}}\qquad\Rightarrow\qquad\phi\subseteq{\leftarrow_\mathsf{S}}.\]
\end{proposition}

\begin{proof}
    For any $S\in\mathsf{S}\mathbb{P}$, we have $c\in C\cap S$ and then ${=_C}\subseteq{{\sqsubset_\phi}\circ{\leftarrow}\circ{\geq}}$ implies we have $q\sqsupset_\phi c$ and $p\geq c$ with $q\leftarrow p$.  But then $S\in c_\mathsf{S}\subseteq p_\mathsf{S}$ and $\phi(S)\in\phi[c_\mathsf{S}]\subseteq q_\mathsf{S}$ and hence $(\phi(S),S)\in q_\mathsf{S}\times p_\mathsf{S}\subseteq{\leftarrow_\mathsf{S}}$.  As $S$ was arbitrary, this shows that $\phi\subseteq{\leftarrow_\mathsf{S}}$.
\end{proof}

\begin{proposition}\label{ConjugateArrow}
    For any homeomorphism $\phi:\mathsf{S}\mathbb{P}\rightarrow\mathsf{S}\mathbb{Q}$ and ${\leftarrow}\subseteq{\mathbb{Q}\times\mathbb{Q}}$,
    \[{\twoheadleftarrow}\subseteq{{\sqsubset_\phi}\circ{\leftarrow}\circ{\sqsupset_\phi}}\qquad\Rightarrow\qquad\phi\circ\overline{\twoheadleftarrow_\mathsf{S}}\circ\phi^{-1}\subseteq{\leftarrow_\mathsf{S}}.\]
\end{proposition}

\begin{proof}
    Say ${\twoheadleftarrow}\subseteq{{\sqsubset_\phi}\circ{\leftarrow}\circ{\sqsupset_\phi}}$ and  $S'\mathrel{\phi\circ\overline{\twoheadleftarrow_\mathsf{S}}\circ\phi^{-1}}S$ so $\phi^{-1}(S')\mathrel{\overline{\twoheadleftarrow_\mathsf{S}}}\phi^{-1}(S)$ and hence we have $p\in\phi^{-1}(S)_\wedge$ and $p'\in\phi^{-1}(S')_\wedge$ with $p'\twoheadleftarrow p$.  Then we have $q\sqsupset_\phi p$ and $q'\sqsupset_\phi p'$ such that $q'\leftarrow q$.  But this means $S\in\phi[\overline{p_\mathsf{S}}]\subseteq q_\mathsf{S}$ and $S'\in\phi[\overline{p'_\mathsf{S}}]\subseteq q'_\mathsf{S}$ and hence $S'\leftarrow_\mathsf{S}S$.
\end{proof}

\subsection{Arrows}\label{Arrows}

Analogously to $\leftarrow_\mathsf{S}$, for any $\omega$-posets $\mathbb{P}$ and $\mathbb{Q}$ and any ${\leftarrow}\subseteq\mathbb{Q}\times\mathbb{P}$, let
\[\overline{\leftarrow_\mathsf{S}}=\bigcup_{q\leftarrow p}\overline{q_\mathsf{S}}\times\overline{p_\mathsf{S}}.\]
If ${\leftarrow}\subseteq\mathbb{Q}\times\mathbb{P}$ is finite then, by \Cref{DenseCovers}, $\overline{(\mathbb{Q}^\leftarrow)_\mathsf{S}}=\mathsf{S}\mathbb{P}$ implies $\mathbb{Q}^{\leftarrow\barwedge}=\mathbb{P}$, i.e.
\[\overline{\leftarrow_\mathsf{S}}\text{ is co-surjective}\qquad\Rightarrow\qquad{{\leftarrow}\circ{\barwedge}}\text{ is co-surjective},\]
and conversely when $\mathbb{P}$ is regular.  Accordingly, we call ${\leftarrow}\subseteq\mathbb{Q}\times\mathbb{P}$ an \emph{arrow} if $\leftarrow$ is finite and ${{\leftarrow}\circ{\barwedge}}$ is co-surjective.  We view these arrows as approximations to continuous functions, as we will elaborate on in the next section when we consider arrow-sequences.

In this section we consider general properties of individual arrows and more general relations.  For example, if $\leftarrow$ approximates a continuous function then we can also approximate the corresponding strong refiner by the relation ${\langle\!\leftarrow}\subseteq\mathbb{Q}\times\mathbb{P}$ defined by
\[{\langle\!\leftarrow}={{\vartriangleright}\co{({\leftarrow}\circ{\barwedge})}}.\]
In other words, for all $p\in\mathbb{P}$ and $q\in\mathbb{Q}$,
\[q\mathrel{\langle\!\leftarrow}p\qquad\Leftrightarrow\qquad p^{\barwedge\rightarrow}\vartriangleleft q.\]

\begin{proposition}
    For any $\omega$-posets $\mathbb{P}$ and $\mathbb{Q}$, continuous $\phi:\mathsf{S}\mathbb{P}\rightarrow\mathsf{S}\mathbb{Q}$ and ${\leftarrow}\subseteq\mathbb{Q}\times\mathbb{P}$,
    \begin{equation}\label{ArrowModification}
        \phi\ \subseteq\ \overline{\leftarrow_\mathsf{S}}\qquad\Rightarrow\qquad{\langle\!\leftarrow}\ \subseteq\ {\sqsupset_\phi}.
    \end{equation}
\end{proposition}

\begin{proof}
    Note $S\in\overline{p_\mathsf{S}}\cap\overline{r_\mathsf{S}}$ implies $p\barwedge r$.  So if $\phi\subseteq\overline{\leftarrow_\mathsf{S}}$ and $S\in\overline{p_\mathsf{S}}$ then
    \[\phi(S)
    \in\bigcup\{\overline{q_\mathsf{S}}:\exists r\rightarrow q\ (S\in\overline{r_\mathsf{S}})\}\subseteq\bigcup\{\overline{q_\mathsf{S}}:q\leftarrow\circ\mathrel{\barwedge}p\}.\]
    This means $\phi[\overline{p_\mathsf{S}}]\subseteq\bigcup_{q\in p^{\barwedge\rightarrow}}\overline{q_\mathsf{S}}\subseteq\bigcap_{r\mathrel{\langle\!\leftarrow}\,p}r_\mathsf{S}$, showing that ${\langle\!\leftarrow}\subseteq{\sqsupset_\phi}$.
\end{proof}

Arrows always give rise to $\wedge$-preserving relations in this way.

\begin{proposition}\label{DoubleArrowWedgePreserving}
    For any $\omega$-posets $\mathbb{P}$ and $\mathbb{Q}$ and any arrow ${\leftarrow}\subseteq\mathbb{Q}\times\mathbb{P}$,
    \[{{\langle\!\leftarrow}\circ{\wedge}\circ{\rightarrow\!\rangle}}\ \ \subseteq\ \ {\wedge}.\]
\end{proposition}

\begin{proof}
    If $q\mathrel{\langle\!\leftarrow}p\mathrel{\wedge}p'\mathrel{\rightarrow\!\rangle}q'$ then we have $p''\in\mathbb{P}$ with $p,p'\geq p''$.  As ${{\leftarrow}\circ{\barwedge}}$ is co-surjective, we then have $c$ and $d$ with $d\leftarrow c\mathrel{\barwedge} p''$ and hence $p\barwedge c\barwedge p'$.  The definition of $\langle\!\leftarrow$ then means that $q,q'\vartriangleright d$.  In particular $q\wedge q'$, as required.
\end{proof}

When the arrow itself is $\wedge$-preserving, we have the following observation.

\begin{proposition}\label{DoubleTriangle}
    Assume $\mathbb{P}$ and $\mathbb{Q}$ are $\omega$-posets and ${\leftarrow}\subseteq\mathbb{Q}\times\mathbb{P}$ is a finite $\wedge$-preserving relation such that $\mathbb{Q}^\leftarrow\in\mathsf{C}\mathbb{P}$.  Then, for any $A,B\subseteq\mathbb{Q}$ with $\mathbb{Q}^\rightarrow\subseteq A\vartriangleleft_AB$,
    \[{{\vartriangleright_B}\circ{\vartriangleright_A}\circ{\leftarrow}}\ \subseteq\ {\langle\!\leftarrow}.\]
\end{proposition}

\begin{proof}
    Say $c\vartriangleright_Bb\vartriangleright_Aa\leftarrow p$.  If $a'\leftarrow p'\mathrel{\barwedge}p$ then we must have $r\in\mathbb{Q}^\leftarrow$ with $p'\wedge r\wedge p$, as $\mathbb{Q}^\leftarrow\in\mathsf{C}\mathbb{P}$.  If $q\leftarrow r$ then $a'\mathrel{\wedge}q\mathrel{\wedge}a$, as $\leftarrow$ is $\wedge$-preserving.  For any $b'\in B$ with $b'\vartriangleright_Aa'$, it follows that $b,b'\geq q$ so $b\mathrel{\wedge}b'$ and hence $a'\vartriangleleft b'\leq c$, as $b\vartriangleleft_Bc$, showing that $c\mathrel{\langle\!\leftarrow}p$.
\end{proof}

In a similar vein, we have the following.

\begin{proposition}\label{DoubleTriangle2}
    Assume $\mathbb{P}$ and $\mathbb{Q}$ are $\omega$-posets, ${\leftarrow},{\sqsupset}\subseteq\mathbb{Q}\times\mathbb{P}$ and $\sqsupset$ is $\wedge$-preserving.  Then, for any $C\subseteq\mathbb{P}$ and $B\subseteq\mathbb{Q}$ with $\mathbb{Q}^\sqsupset\subseteq C\mathrel{\rightarrow\!\rangle}B$ and ${=_C}\subseteq{{\wedge}\circ{\sqsubset}\circ{\leftarrow}\circ{\barwedge}}$,
    \[{{\vartriangleright_B}\circ{\sqsupset}}\ \subseteq\ {\langle\!\leftarrow}.\]
\end{proposition}

\begin{proof}
    Say $q\vartriangleright_Ba\sqsupset c$.  Take $b\in B$ with $b\mathrel{\langle\!\leftarrow}c$.  Take $c'\in\mathbb{P}$ and $a'\in\mathbb{Q}$ such that $c\mathrel{\wedge}c'\sqsubset a'\leftarrow\circ\mathrel{\barwedge}c$ and hence $a\mathrel{\wedge}a'$, as $\sqsupset$ is $\wedge$-preserving.  For any other $q'\in\mathbb{Q}$ with $q'\leftarrow\circ\mathrel{\barwedge}c$, we see that $q',a'\vartriangleleft b$, as $b\mathrel{\langle\!\leftarrow}c$.  Thus $a\mathrel{\wedge}b$ and hence $q'\vartriangleleft b\leq q$, as $a\vartriangleleft_Bq$, showing that $q\mathrel{\langle\!\leftarrow}c$.
\end{proof}

For any $\omega$-posets $\mathbb{P}$ and $\mathbb{Q}$, let us define a preorder on $\mathsf{P}(\mathbb{Q}\times\mathbb{P})$ by
\[{\leftarrow}\ \leq\ {\twoheadleftarrow}\qquad\Leftrightarrow\qquad{\leftarrow}\ \subseteq\ {\leq\circ\twoheadleftarrow\circ\geq}.\]

\begin{proposition}\label{leqSubs}
    For any $\omega$-posets $\mathbb{P}$ and $\mathbb{Q}$ and ${\leftarrow},{\twoheadleftarrow}\subseteq\mathbb{Q}\times\mathbb{P}$,
    \[{\leftarrow}\ \leq\ {\twoheadleftarrow}\qquad\Rightarrow\qquad{\leftarrow_\mathsf{S}}\ \subseteq\ {\twoheadleftarrow_\mathsf{S}},\quad{\overline{\leftarrow_\mathsf{S}}}\ \subseteq\ {\overline{\twoheadleftarrow_\mathsf{S}}}\quad\text{and}\quad{\langle\!\twoheadleftarrow}\ \subseteq\ {\langle\!\leftarrow}.\]
\end{proposition}

\begin{proof}
    Assume ${\leftarrow}\leq{\twoheadleftarrow}$.  For any $q\leftarrow p$, this means we have $q'\geq q$ and $p'\geq p$ with $q'\twoheadleftarrow p'$ and hence $q_\mathsf{S}\times p_\mathsf{S}\subseteq q'_\mathsf{S}\times p'_\mathsf{S}\subseteq{\twoheadleftarrow_\mathsf{S}}$.  Thus ${\leftarrow_\mathsf{S}}\subseteq{\twoheadleftarrow_\mathsf{S}}$ and, likewise, ${\overline{\leftarrow_\mathsf{S}}}\subseteq{\overline{\twoheadleftarrow_\mathsf{S}}}$.
    
    Now say $q\mathrel{\langle\!\twoheadleftarrow}p$.  Whenever $q'\leftarrow p'\mathrel{\barwedge}p$, we then have $q''\geq q'$ and $p''\geq p'$ with $q''\twoheadleftarrow p''$.  But $p''\geq p'\mathrel{\barwedge}p$ implies $p\mathrel{\barwedge}p''\twoheadrightarrow q''$ and hence $q'\leq q''\vartriangleleft q$, as $q\mathrel{\langle\!\twoheadleftarrow}p$.  This shows that $q\mathrel{\langle\!\leftarrow}p$, which in turn shows that ${\langle\!\twoheadleftarrow}\subseteq{\langle\!\leftarrow}$.
\end{proof}

For predetermined posets, we also have a converse.

\begin{proposition}
    For any predetermined $\omega$-posets $\mathbb{P}$ and $\mathbb{Q}$ and any $n\in\omega$,
    \[{\leftarrow}\subseteq\mathbb{Q}_n\times\mathbb{P}_n,\quad{\twoheadleftarrow}\subseteq\mathbb{Q}^n\times\mathbb{P}^n\quad\text{and}\quad{\leftarrow_\mathsf{S}}\subseteq{\twoheadleftarrow_\mathsf{S}}\qquad\Rightarrow\qquad{\leftarrow}\leq{\twoheadleftarrow}.\]
\end{proposition}

\begin{proof}
    We argue as in the proof of \Cref{PreChars} \eqref{pSp}$\Rightarrow$\eqref{SubLeq}.  Specifically, say ${\leftarrow}\subseteq\mathbb{Q}_n\times\mathbb{P}_n$, ${\twoheadleftarrow}\subseteq\mathbb{Q}^n\times\mathbb{P}^n$, ${\leftarrow_\mathsf{S}}\subseteq{\twoheadleftarrow_\mathsf{S}}$ and $q\leftarrow p$.  As $\mathbb{P}$ and $\mathbb{Q}$ are predetermined, \Cref{PreChars} \eqref{pSp} yields $S\in p_\mathsf{S}$ and $T\in q_\mathsf{S}$ with $S\subseteq p^\leq\cap p^\geq$ and $T\subseteq q^\leq\cap q^\geq$.  Thus $(T,S)\in{\leftarrow_\mathsf{S}}\subseteq{\twoheadleftarrow_\mathsf{S}}$ so we have $p'\in S$ and $q'\in T$ with $q'\twoheadleftarrow p'$.  As $\mathbb{P}_n$ and $\mathbb{Q}_n$ consist of atoms of $\mathbb{P}^n$ and $\mathbb{Q}^n$ respectively, we know that $p'\not<p$ and $q'\not<q$.  The only remaining possibility is $p\leq p'$ and $q\leq q'$.  As $(q,p)$ was an arbitrary element of $\leftarrow$, this shows that ${\leftarrow}\leq{\twoheadleftarrow}$.
\end{proof}

Let us define another relation $\vartriangleleft$ on $\mathsf{P}(\mathbb{Q}\times\mathbb{P})$ by
\[{\twoheadleftarrow}\ \vartriangleleft\ {\leftarrow}\qquad\Leftrightarrow\qquad{\twoheadleftarrow}\ \subseteq\ {{\vartriangleleft}\circ{\leftarrow}\circ{\vartriangleright}}.\]
We can also define a weaker and slightly simpler variant of ${\langle\!\leftarrow}$ by
\[{[\!\leftarrow}={{\geq}\co({\leftarrow}\circ{\wedge})}.\]

\begin{proposition}\label{SquareArrow}
    For any $\omega$-posets $\mathbb{P}$ and $\mathbb{Q}$ and ${\leftarrow},{\twoheadleftarrow}\subseteq\mathbb{Q}\times\mathbb{P}$,
    \[\twoheadleftarrow\ \vartriangleleft\ \leftarrow\qquad\Rightarrow\qquad{[\!\leftarrow}\subseteq{\langle\!\twoheadleftarrow}.\]
\end{proposition}

\begin{proof}
    Say ${\twoheadleftarrow}\vartriangleleft{\leftarrow}$ and $q\mathrel{[\!\leftarrow}p$, i.e.~$p^{\wedge\rightarrow}\subseteq q^\geq$.  Then $p^{\barwedge\twoheadrightarrow}\subseteq p^{\barwedge\vartriangleleft\rightarrow\vartriangleright}\subseteq p^{\wedge\rightarrow\vartriangleright}\subseteq q^{\geq\vartriangleright}\subseteq q^\vartriangleright$, showing that $q\mathrel{\langle\!\twoheadleftarrow}p$.  This in turn shows that ${[\!\leftarrow}\subseteq{\langle\!\twoheadleftarrow}$.
\end{proof}

\subsection{Arrow-Sequences}\label{ArrowSequences}

We call a decreasing sequence $(\leftarrow_n)\subseteq\mathsf{P}(\mathbb{Q}\times\mathbb{P})$ of arrows (i.e.~where ${\leftarrow_{n+1}}\leq{\leftarrow_n}$, for all $n\in\omega$) an \emph{arrow-sequence}.  By the \Cref{DoubleArrowWedgePreserving,leqSubs}, we then have a $\wedge$-preserving relation ${\sqsupset}\subseteq\mathbb{Q}\times\mathbb{P}$ given by
\[{\sqsupset}=\bigcup_{n\in\omega}\langle\!\leftarrow_n.\]
Moreover, this relation $\sqsupset$ will be a (strong) refiner precisely when the subsets $(\overline{\leftarrow_{n\mathsf{S}}})$ intersect to define a unique (continuous) function $\phi:\mathsf{S}\mathbb{P}\rightarrow\mathsf{S}\mathbb{Q}$.

\begin{theorem}
For any arrow-sequence $(\leftarrow_n)\subseteq\mathsf{P}(\mathbb{Q}\times\mathbb{P})$ between regular $\omega$-posets,
\[{\sqsupset}=\bigcup_{n\in\omega}\langle\!\leftarrow_n\text{ is a refiner}\qquad\Leftrightarrow\qquad\phi=\bigcap_{n\in\omega}\overline{\leftarrow_{n\mathsf{S}}}\text{ is a function}.\]
In this case, $\sqsupset$ is a strong refiner with corresponding continuous function $\phi=\mathsf{S}_\sqsupset$.
\end{theorem}

\begin{proof}
    Say $\sqsupset$ is not a refiner, so we have $C\in\mathsf{C}\mathbb{Q}$ such that $B\nsubseteq C^\sqsupset$, for all $B\in\mathsf{C}\mathbb{P}$.  So $\mathbb{P}\setminus C^\sqsupset$ is a selector and hence contains a minimal selector $S$.  Then $S^{\overline{\rightarrow_{n\mathsf{S}}}}$ defines a decreasing sequence of non-empty closed subsets of $\mathsf{S}\mathbb{Q}$ and hence we have some $T\in\bigcap_{n\in\omega}S^{\overline{\rightarrow_{n\mathsf{S}}}}$, i.e.~satisfying $T\mathrel{\phi}S$.  Taking any $c\in C\cap T$, we know that $S\cap c^\sqsupset=\emptyset$.  For each $s\in S$ and $n\in\omega$, we thus have some $q\in s^{\barwedge\rightarrow_n}\setminus c^\vartriangleright$.  As $S$ is directed and $\leftarrow_n$ is finite, we must then have $(q,p)\in{\leftarrow_n}$ with $p\in S_\barwedge$ and $q\not\vartriangleleft c$ and hence $\emptyset\neq\overline{q_\mathsf{S}}\setminus c_\mathsf{S}\subseteq S^{\overline{\leftarrow_{n\mathsf{S}}}}\setminus c_\mathsf{S}$.  We then have $U\in\bigcap_{n\in\omega}S^{\overline{\rightarrow_{n\mathsf{S}}}}\setminus c_\mathsf{S}$ and hence $T\neq U\mathrel{\phi}S$, showing that $\phi$ is not a function.

    Now say $\sqsupset$ is a refiner and take $S\in\mathsf{S}\mathbb{P}$ and $T\in\mathsf{S}\mathbb{Q}$ with $(T,S)\in\phi$.  For any $C\in\mathsf{C}\mathbb{Q}$, this means we have $B\in\mathsf{C}\mathbb{P}$ and $n\in\omega$ with $B\mathrel{{}_n\!\!\rightarrow\!\rangle}C$.  As $(T,S)\in\phi$, we also have $s\in S_\wedge$ and $t\in T_\wedge$ with $t\leftarrow_ns$.  Taking any $b\in B\cap S$ and $c\in C$ with $c\mathrel{\langle\!\leftarrow_n}b$, it follows that $s\barwedge b$ so $t\vartriangleleft c$ and hence $c\in T\cap\mathsf{S}_\sqsupset(S)$.  This shows that $U=T\cap\mathsf{S}_\sqsupset(S)$ is a selector and hence $T=\mathsf{S}_\sqsupset(X)$, by minimality.  This shows that $\phi=\mathsf{S}_\sqsupset$.
    
    We next claim ${\sqsupset}\subseteq{\vartriangleright*\sqsupset}$.  To see this, say $q\mathrel{\langle\!\leftarrow_m}p$ and hence $p^{\barwedge\rightarrow_m}\vartriangleleft_Cq$, for some $C\in\mathsf{C}\mathbb{Q}$.  As $\mathbb{Q}$ is regular, we have $B\in\mathsf{C}\mathbb{Q}$ with $B\vartriangleleft C$.  As $\sqsupset$ is a refiner, we have $n>m$ and $A\in\mathsf{C}\mathbb{P}$ with $A\mathrel{{}_n\!\!\rightarrow\!\rangle}B$.  So whenever $p\mathrel{\wedge}a\in A$, we have $b\in B$ and $c\in C$ with $a\mathrel{{}_n\!\!\rightarrow\!\rangle}b\vartriangleleft c$.  Take $p'\in\mathbb{Q}^{\leftarrow_n}$ with $p\mathrel{\barwedge}p'\mathrel{\barwedge}a$, so we have $q'\in\mathbb{Q}$ with $p'\rightarrow_nq'$, necessarily with $q'\vartriangleleft b$.  Take $p''$ and $q''$ with $p'\leq p''\rightarrow_mq''\geq q'$.  Then $p\mathrel{\barwedge}p''$ so $q''\vartriangleleft_Cq$, but also $c\vartriangleright b\mathrel{\wedge}q''$ so this implies $b\vartriangleleft c\leq q$ and hence $a\mathrel{{}_n\!\!\rightarrow\!\rangle}b\vartriangleleft q$.  This shows that $q\mathrel{{\vartriangleright}*{\langle\!\leftarrow_n}}p$, which in turn shows that ${\sqsupset}\subseteq{\vartriangleright*\sqsupset}$.

    On the other hand, say $q\sqsupset^*p$, so we have some finite $C\in\mathsf{C}\mathbb{P}$ with $Cp\sqsubset q$.  This means we must have some $n\in\omega$ with $Cp\mathrel{{}_n\!\!\rightarrow\!\rangle}q$.  Whenever $p\mathrel{\barwedge}p'\rightarrow_nq'$, we have $c\in C$ with $p\mathrel{\wedge}c\mathrel{\wedge}p'$ so $c\in Cp\mathrel{{}_n\!\!\rightarrow\!\rangle}q$ and hence $q'\in c^{\wedge\rightarrow_n}\subseteq c^{\barwedge\rightarrow_n}\vartriangleleft q$.  This shows that $p^{\barwedge\rightarrow_n}\vartriangleleft q$ and hence $p\sqsubset q$, which in turn shows that ${\sqsupset^*}\subseteq{\sqsupset}$.  Also ${{\vartriangleright}\circ{\langle\!\leftarrow}}\subseteq{\langle\!\leftarrow}$, for any ${\leftarrow}\subseteq\mathbb{Q}\times\mathbb{P}$, so ${\vartriangleright\circ\sqsupset}\subseteq{\sqsupset}$ and hence ${\vartriangleright*\sqsupset}\subseteq{\sqsupset^*}\subseteq{\sqsupset}\subseteq{\vartriangleright*\sqsupset}$, i.e.~${\sqsupset}={\vartriangleright*\sqsupset}$.  This shows that $\sqsupset$ is indeed a strong refiner.
\end{proof}

Let us call an arrow-sequence $(\leftarrow_n)$ \emph{thin} in the above situation, i.e.~if ${\sqsupset}=\bigcup_{n\in\omega}{\langle\!\leftarrow_n}$ is a refiner or, equivalently, if $\phi=\bigcap_{n\in\omega}\overline{\leftarrow_{n\mathsf{S}}}$ is a function.  So thin arrow-sequences yield continuous functions and, conversely, every continuous function arises in this way.

To see this first, for any $\omega$-posets $\mathbb{P}$ and $\mathbb{Q}$, $n\in\omega$ and $\phi\subseteq\mathsf{S}\mathbb{Q}\times\mathsf{S}\mathbb{P}$, define
\begin{align*}
    {\leftarrow^\phi_n}&=\{(q,p)\in\mathbb{Q}_n\times\mathbb{P}_n:\phi\cap(\overline{q_\mathsf{S}}\times\overline{p_\mathsf{S}})\neq\emptyset\}\\
    &=\{(q,p)\in\mathbb{Q}_n\times\mathbb{P}_n:\exists S,T\in\mathsf{S}\mathbb{P}\ (q^\wedge\supseteq T\mathrel{\phi}S\subseteq p^\wedge)\}.
\end{align*}
For all $n\in\omega$, note ${\leftarrow^\phi_{n+1}}\leq{\leftarrow^\phi_n}$ and $\phi\subseteq{\leftarrow^\phi_{n\mathsf{S}}}$.  So if $\phi$ is co-surjective then $(\leftarrow^\phi_n)$ is an arrow-sequence.  When $\phi$ is a continuous function, this arrow-sequence will also be thin.

\begin{proposition}\label{ArrowPhi}
    For any regular $\omega$-posets $\mathbb{P}$ and $\mathbb{Q}$, continuous $\phi:\mathsf{S}\mathbb{P}\rightarrow\mathsf{S}\mathbb{Q}$ and $D\in\mathsf{C}\mathbb{Q}$, we have $n\in\omega$ such that $D^{\langle\!\leftarrow^\phi_n}\in\mathsf{C}\mathbb{P}$.
\end{proposition}

\begin{proof}
    Take $C\in\mathsf{C}\mathbb{Q}$ with $C\vartriangleleft_CD$.  As $\{\phi^{-1}[c_\mathsf{S}]:c\in C\}$ covers $\mathsf{S}\mathbb{P}$, it is refined by $\{\overline{b_\mathsf{S}}:b\in B\}$, for some $B\in\mathsf{C}\mathbb{P}$.  Take any $n\in\omega$ and $A\in\mathsf{C}\mathbb{P}$ with $\mathbb{Q}_n\vartriangleleft C$ and $\mathbb{P}_n\vartriangleleft_A\circ\vartriangleleft_{\mathbb{P}_n}B$.  For any $p\in\mathbb{P}_n$, we thus have $p'\in\mathbb{P}$, $b\in B$, $c\in C$ and $d\in D$ such that $p\vartriangleleft_A p'\vartriangleleft_{\mathbb{P}_n}b$, $\overline{b_\mathsf{S}}\subseteq\phi^{-1}[c_\mathsf{S}]$ and $c\vartriangleleft_Cd$.  Now if $q\leftarrow^\phi_np''\mathrel{\barwedge}p$ then we have $a\in A$ with $p''\mathrel{\wedge}a\mathrel{\wedge}p$ so $a\leq p'$ and hence $p''\mathrel{\wedge}p'$ which in turn yields $p''\leq b$ and hence $\phi[\overline{p''_\mathsf{S}}]\subseteq\phi[\overline{b_\mathsf{S}}]\subseteq c_\mathsf{S}$.  Taking any $q'\in C$ with $q\vartriangleleft q'$, it follows that $\emptyset\neq \overline{q_\mathsf{S}}\cap\phi[\overline{p''_\mathsf{S}}]\subseteq q'_\mathsf{S}\cap c_\mathsf{S}$ so $q'\mathrel{\wedge}c$ and hence $q\vartriangleleft q'\leq d$.  This shows that $d\mathrel{\langle\!\leftarrow^\phi_n}p$ and hence $D^{\langle\!\leftarrow^\phi_n}\supseteq\mathbb{P}_n\in\mathsf{C}\mathbb{P}$.
\end{proof}

\begin{corollary}\label{Function->ArrowSequence}
    For any regular $\omega$-posets $\mathbb{P}$ and $\mathbb{Q}$ and any continuous $\phi:\mathsf{S}\mathbb{P}\rightarrow\mathsf{S}\mathbb{Q}$, $(\leftarrow^\phi_n)$ is a thin arrow-sequence such that $\phi=\bigcap\leftarrow^\phi_{n\mathsf{S}}$.
\end{corollary}

\begin{proof}
    The previous result immediately tells us that $(\leftarrow^\phi_n)$ is thin and thus $\bigcap\leftarrow^\phi_{n\mathsf{S}}$ is a function containing and hence equal to $\phi$.
\end{proof}

Let us yet again define a relation $\vartriangleleft_n$ on $\mathsf{P}(\mathbb{Q}\times\mathbb{P})$, for each $n\in\omega$, by
\[{\twoheadleftarrow}\ \vartriangleleft_n\ {\leftarrow}\qquad\Leftrightarrow\qquad{\twoheadleftarrow}\ \subseteq\ {{\vartriangleleft_{\mathbb{Q}_n}}\circ{\leftarrow}\circ{\vartriangleright_{\mathbb{P}_n}}}.\]
If $\mathbb{P}$ and $\mathbb{Q}$ are regular $\omega$-posets then, for each $m\in\omega$, we have $n<m$ with $\mathbb{P}_n\vartriangleleft_{\mathbb{P}_n}\mathbb{P}_m$ and $\mathbb{Q}_n\vartriangleleft_{\mathbb{Q}_n}\mathbb{Q}_m$.  For any $\phi\subseteq\mathsf{S}\mathbb{Q}\times\mathsf{S}\mathbb{P}$ it then follows that
\[{\leftarrow^\phi_n}\ \ \subseteq\ \ {{\vartriangleleft_{\mathbb{Q}_n}}\circ{=_{\mathbb{Q}_m}}\circ{\vartriangleright_{\mathbb{Q}_n}}\circ{\leftarrow^\phi_n}\circ{\vartriangleleft_{\mathbb{P}_n}}}\circ{=_{\mathbb{P}_m}}\circ{\vartriangleright_{\mathbb{P}_n}}\ \ \subseteq\ \ {{\vartriangleleft_{\mathbb{Q}_n}}\circ{\leftarrow^\phi_m}\circ{\vartriangleright_{\mathbb{P}_n}}}\]
and hence ${\leftarrow^\phi_n}\vartriangleleft_n{\leftarrow^\phi_m}$.  We can even replace $\leftarrow^\phi_m$ here with a relation $\leftarrow$ with $\mathrm{cl}(\phi)\subseteq{\leftarrow_\mathsf{S}}$.

\begin{proposition}\label{ArrowPhiLebesgue}
    For any regular $\omega$-posets $\mathbb{P}$ and $\mathbb{Q}$, closed relation $\phi\subseteq\mathsf{S}\mathbb{Q}\times\mathsf{S}\mathbb{P}$ and relation ${\leftarrow}\subseteq\mathbb{Q}\times\mathbb{P}$ such that $\phi\subseteq{\leftarrow_\mathsf{S}}$, we have $n\in\omega$ such that ${\leftarrow^\phi_n}\vartriangleleft_n{\leftarrow}$.
\end{proposition}

\begin{proof}
    This follows from the Lebesgue number lemma combined with \cite[Proposition 1.17 and Theorem 2.29]{BaBiVi}.  For convenience, we also provide a more direct proof.
    
    For any $(T,S)\in\phi\subseteq{\leftarrow_\mathsf{S}}$, we have $p\in S$ and $q\in T$ such that $q\leftarrow p$.  As the spectrum consists of minimal selectors, we then have $n\in\omega$ such that $\mathbb{P}_n\cap S\leq p$ and $\mathbb{Q}_n\cap T\leq q$.  For each $(T,S)\in\phi$, we may thus let $n(T,S)$ be the minimal $n\in\omega$ such that $(T\cap\mathbb{Q}_n)\times(S\cap\mathbb{P}_n)\leq{\leftarrow}$.  We claim that $n[\phi]$ is bounded.

    Indeed, if $n[\phi]$ where unbounded then we would have a sequence $(T_k,S_k)\subseteq\phi$ with $n(T_k,S_k)$ strictly increasing.  Reverting to a subsequence if necessary, we may further assume that $(T_k,S_k)$ converges to some $(T,S)\in\phi$.  Take any $m>n=n(T,S)$ with $\mathbb{P}_m\vartriangleleft_{\mathbb{P}_m}\mathbb{P}_n$ and $\mathbb{Q}_m\vartriangleleft_{\mathbb{Q}_m}\mathbb{Q}_n$.  Taking any $p\in\mathbb{P}_m\cap S$ and $q\in\mathbb{Q}_m\cap T$, we then have $p'\in\mathbb{P}_n$ and $q'\in\mathbb{Q}_n$ with $p\vartriangleleft_{\mathbb{P}_m}p'$ and $q\vartriangleleft_{\mathbb{Q}_m}q'$.  By the definition of $n(T,S)$, we then have $p''\vartriangleright p'$ and $q''\vartriangleright q'$ with $q''\leftarrow p''$.  However, as $(T_k,S_k)$ converges to $(T,S)$, we can pick $k\in\omega$ such that $p\in S_k$, $q\in T_k$ and $n(T_k,S_k)>m$.  For any $p'''\in S_k\cap\mathbb{P}_m$ and $q'''\in T_k\cap\mathbb{Q}_m$, it follows that $p'''\mathrel{\wedge}p$ and $q'''\mathrel{\wedge}q$ so $p'''\leq p'\leq p$ and $q'''\leq q'\leq q$.  This shows that $(\phi(S_k)\cap\mathbb{Q}_m)\times(S_k\cap\mathbb{P}_m)\leq{\leftarrow}$ so $n(S_k)\leq m$, a contradiction.

    This proves the claim that we have $m\in\omega$ such that $(T\cap\mathbb{Q}_m)\times(S\cap\mathbb{P}_m)\leq{\leftarrow}$, for all $(T,S)\in\phi$.  Taking any $n>m$ with $\mathbb{P}_n\vartriangleleft_{\mathbb{P}_n}\mathbb{P}_m$, it follows that $(T_{\wedge}\cap\mathbb{Q}_n)\times(S_{\wedge}\cap\mathbb{P}_n)\leq{\leftarrow}$, for all $(T,S)\in\phi$.  But if $q\leftarrow^\phi_np$ then we have $(T,S)\in\phi\cap(\overline{p_\mathsf{S}}\cap\overline{q_\mathsf{S}})$ which then implies $(q,p)\in(T_{\wedge}\cap\mathbb{Q}_n)\times(S_{\wedge}\cap\mathbb{P}_n)\leq{\leftarrow}$.  This shows that ${\leftarrow^\phi_n}\leq{\leftarrow}$.  Taking $n'>n$ large enough that ${\leftarrow^\phi_{n'}}\vartriangleleft_{n'}{\leftarrow^\phi_n}$ then yields ${\leftarrow^\phi_{n'}}\vartriangleleft_{n'}{\leftarrow}$ as well.
\end{proof}

One immediately corollary is that relations on levels define a basis for $\mathbf{C}^{\mathsf{S}\mathbb{Q}}_{\mathsf{S}\mathbb{P}}$.

\begin{corollary}\label{LevelBasis}
    For any regular $\omega$-posets $\mathbb{P}$ and $\mathbb{Q}$, we have a basis for $\mathbf{C}^{\mathsf{S}\mathbb{Q}}_{\mathsf{S}\mathbb{P}}$ given by
    \[\{{\leftarrow_\mathbf{C}}:n\in\omega\text{ and }{\leftarrow}\subseteq\mathbb{Q}_n\times\mathbb{P}_n\}.\]
\end{corollary}

\begin{proof}
    By definition, $\mathbf{C}^{\mathsf{S}\mathbb{P}}_{\mathsf{S}\mathbb{P}}$ has a basis consisting open sets of the form $\leftarrow_\mathbf{C}$, for ${\leftarrow}\subseteq\mathbb{Q}\times\mathbb{P}$.  But if $\phi\in{\leftarrow_\mathbf{C}}$, i.e.~$\phi\subseteq{\leftarrow_\mathsf{S}}$, then the above result yields $n\in\omega$ such that ${\leftarrow^\phi_n}\leq{\leftarrow}$ and hence $\phi\in(\leftarrow^\phi_n)_\mathsf{S}\subseteq{\leftarrow_\mathsf{S}}$.  As ${\leftarrow^\phi_n}\subseteq\mathbb{Q}_n\times\mathbb{P}_n$, we are done.
\end{proof}

To prove another corollary, let us first observe that
\[\mathbb{Q}_n\times\mathbb{P}_n\ \supseteq\ {\twoheadleftarrow}\ \leq\ {\leftarrow}\qquad\Rightarrow\qquad{\twoheadleftarrow}\ \subseteq\ {\leftarrow_n^{\overline{\leftarrow_\mathsf{S}}}}.\]
Indeed, if $q\twoheadleftarrow p$ then $\overline{q_\mathsf{S}}\times\overline{p_\mathsf{S}}\subseteq\overline{\twoheadleftarrow_\mathsf{S}}\subseteq\overline{\leftarrow_\mathsf{S}}$ so $\emptyset\neq\overline{\leftarrow_\mathsf{S}}\cap\overline{q_\mathsf{S}}\times\overline{p_\mathsf{S}}$ and hence $q\leftarrow_n^{\overline{\leftarrow_\mathsf{S}}}p$.

\begin{corollary}\label{OpenNeighbourhoodSequence}
    Whenever we are given regular $\omega$-posets $\mathbb{P}$ and $\mathbb{Q}$, a strictly increasing sequence $(n_k)\subseteq\omega$, a decreasing sequence ${\leftarrow_k}\subseteq\mathbb{Q}_{n_k}\times\mathbb{P}_{n_k}$ and a relation ${\leftarrow}\subseteq\mathbb{Q}\times\mathbb{P}$ with $\phi\subseteq{\leftarrow_\mathsf{S}}$, where $\phi=\bigcap_{k\in\omega}\overline{\leftarrow_{k\mathsf{S}}}$, there is some $k\in\omega$ with ${\leftarrow_k}\vartriangleleft_k{\leftarrow}$.
\end{corollary}

\begin{proof}
    As $\phi\subseteq{\leftarrow_\mathsf{S}}$ and $\mathsf{S}\mathbb{Q}\times\mathsf{S}\mathbb{P}$ is compact, we have $l\in\omega$ with ${\overline{\leftarrow_{l\mathsf{S}}}}\subseteq{\leftarrow_\mathsf{S}}$.  The previous result then yields $k\in\omega$ with ${\leftarrow^{\overline{\leftarrow_{l\mathsf{S}}}}_k}\vartriangleleft_k{\leftarrow}$.  Enlarging $k$ if necessary, we may assume $k\geq l$ so ${\leftarrow_k}\leq{\leftarrow_l}$ and hence ${\leftarrow_k}\subseteq{\leftarrow^{\overline{\leftarrow_{l\mathsf{S}}}}_k}$, by the above observation, so ${\leftarrow_k}\vartriangleleft_k{\leftarrow}$ as well.
\end{proof}

Let us call $(\leftarrow_n)$ \emph{regular} if, for all $m\in\omega$, we have $n>m$ with ${\leftarrow_n}\vartriangleleft{\leftarrow_m}$.  We already observed that if $\mathbb{P}$ and $\mathbb{Q}$ are regular $\omega$-posets then $(\leftarrow^\phi_n)$ is regular, for any $\phi\subseteq\mathsf{S}\mathbb{Q}\times\mathsf{S}\mathbb{P}$, where $\vartriangleleft$ can even be replaced by the stronger relation $\vartriangleleft_n$.  In fact this last statement remains valid for general regular sequences, which can also be characterised as follows.

\begin{proposition}\label{RegularSequence}
    Given any regular $\omega$-posets $\mathbb{P}$ and $\mathbb{Q}$, any strictly increasing sequence $(n_k)\subseteq\omega$ and any decreasing sequence ${\leftarrow_k}\subseteq\mathbb{Q}_{n_k}\times\mathbb{P}_{n_k}$,
    \[(\leftarrow_k)\text{ is regular}\qquad\Leftrightarrow\qquad{\bigcap_{k\in\omega}\overline{\leftarrow_{k\mathsf{S}}}}={\bigcap_{k\in\omega}\leftarrow_{k\mathsf{S}}}\qquad\Rightarrow\qquad{\bigcup_{k\in\omega}\langle\!\leftarrow_k}={\bigcup_{k\in\omega}[\!\leftarrow_k}.\]
    Moreover, if $(\leftarrow_n)$ regular then
    \begin{enumerate}
        \item for all $m\in\omega$, we can even find $n>m$ such that ${\leftarrow_n}\vartriangleleft_n{\leftarrow_m}$, and
        \item $(\leftarrow_n)$ is an arrow-sequence precisely when $\mathbb{Q}^{\leftarrow_n}\in\mathsf{C}\mathbb{P}$, for all $n\in\omega$.
    \end{enumerate}
\end{proposition}

\begin{proof}
    If $(\leftarrow_k)$ is regular then, for every $m\in\omega$, we have $n>m$ with ${\leftarrow_n}\vartriangleleft{\leftarrow_m}$.  This means $\overline{\leftarrow_{n\mathsf{S}}}\subseteq{\leftarrow_{m\mathsf{S}}}$, by \Cref{VarChar}, ${[\!\leftarrow_m}\subseteq{\langle\!\leftarrow_n}$, by \Cref{SquareArrow}, and $\mathbb{Q}^{\leftarrow_n}\vartriangleleft\mathbb{Q}^{\leftarrow_m}$.  From this it follows that ${\bigcap_{k\in\omega}\overline{\leftarrow_{k\mathsf{S}}}}={\bigcap_{k\in\omega}\leftarrow_{k\mathsf{S}}}$, ${\bigcup_{k\in\omega}\langle\!\leftarrow_k}={\bigcup_{k\in\omega}[\!\leftarrow_k}$ and $\mathbb{Q}^{\leftarrow_m}\in\mathsf{C}\mathbb{P}$ if $\leftarrow_n$ is an arrow, by \Cref{DenseCovers}.

    Conversely, if ${\bigcap_{k\in\omega}\overline{\leftarrow_{k\mathsf{S}}}}={\bigcap_{k\in\omega}\leftarrow_{k\mathsf{S}}}$ then, for every $m\in\omega$, ${\bigcap_{k\in\omega}\overline{\leftarrow_{k\mathsf{S}}}}\subseteq{\leftarrow_{m\mathsf{S}}}$.  Then \Cref{OpenNeighbourhoodSequence} yields $n>m$ with ${\leftarrow_n}\vartriangleleft_n{\leftarrow_m}$.
\end{proof}

For later use let us also prove the following.

\begin{lemma}\label{ConjugacyLemma}
    Assume we are given regular $\omega$-posets $\mathbb{P}$ and $\mathbb{Q}$ together with a continuous function $\phi:\mathsf{S}\mathbb{Q}\rightarrow\mathsf{S}\mathbb{Q}$, a homeomorphism $\theta:\mathsf{S}\mathbb{P}\rightarrow\mathsf{S}\mathbb{Q}$ and a thin arrow-sequence $(\leftarrow_k)\subseteq\mathbb{P}_{n_k}\times\mathbb{P}_{n_k}$, for some strictly increasing $(n_k)\subseteq\omega$, such that
    \[\theta^{-1}\circ\phi\circ\theta=\bigcap\overline{\leftarrow_{k\mathsf{S}}}.\]
    For any $n\in\omega$ and ${\sqsupset}\subseteq{\sqsupset_\theta}$ with $\mathbb{Q}_n^\sqsupset\in\mathsf{C}\mathbb{P}$, we then have $k\in\omega$ with ${\leftarrow_k}\vartriangleleft_k{\sqsubset\circ\leftarrow^\phi_n\circ\sqsupset}$.
\end{lemma}

\begin{proof}
    First we claim that $\theta^{-1}\circ\phi\circ\theta\subseteq(\sqsubset\circ\leftarrow^\phi_n\circ\sqsupset)_\mathsf{S}$.  To see this, take any $S\in\mathsf{S}\mathbb{P}$.  As $\mathbb{Q}_n^\sqsupset\in\mathsf{C}\mathbb{P}$, we have $p\in\mathbb{Q}_n^\sqsupset\cap S$, which means we have $q\in\mathbb{Q}_n$ with $q\sqsupset p$ and hence $q\in\theta(S)$, as ${\sqsupset}\subseteq{\sqsupset_\theta}$.  Likewise, we have $p'\in\theta^{-1}(\phi(\theta(S)))$ and $q'\in\mathbb{Q}_n$ with $q'\sqsupset p'$, necessarily with $q'\in\theta(\theta^{-1}(\phi(\theta(S))))=\phi(\theta(S))$.  Thus $\theta(S)$ witnesses the fact that $q'\leftarrow^\phi_nq$ and hence $(\theta^{-1}(\phi(\theta(S))),S)\in p'_\mathsf{S}\times p_\mathsf{S}\subseteq(\sqsubset\circ\leftarrow^\phi_n\circ\sqsupset)_\mathsf{S}$.  As $S$ was arbitrary, this proves the claim that $\theta^{-1}\circ\phi\circ\theta\subseteq(\sqsubset\circ\leftarrow^\phi_n\circ\sqsupset)_\mathsf{S}$.  Now just apply \Cref{OpenNeighbourhoodSequence}. 
\end{proof}

Another simple observation that will be useful later is the following.

\begin{proposition}\label{ArrowFromRefiner}
    For any continuous $\phi:\mathsf{S}\mathbb{P}\rightarrow\mathsf{S}\mathbb{Q}$ and ${\sqsupset}\subseteq{\mathbb{Q}_m\times\mathbb{P}_n}$, where $m\leq n$,
    \[{\sqsupset}\ \subseteq\ {\sqsupseteq_\phi}\qquad\Rightarrow\qquad{{\sqsupset}\circ{\leq^m_n}}\ \subseteq\ {\leftarrow^\phi_m}.\]
\end{proposition}

\begin{proof}
    Just note that if $\mathbb{Q}_m\ni q\sqsupseteq_\phi r\leq^m_np$ then $\overline{q_\mathsf{S}}\cap\phi[\overline{p_\mathsf{S}}]\supseteq\phi[\overline{r_\mathsf{S}}]\neq\emptyset$ so $q\leftarrow^\phi_mp$.
\end{proof}

Lastly, let us note that arrow-sequences can be composed via $\barwedge$.

\begin{proposition}\label{ArrowComposition}
    Given $\omega$-posets $\mathbb{P}$, $\mathbb{Q}$ and $\mathbb{R}$, any arrow-sequences $(\leftarrow_n)\subseteq\mathsf{P}(\mathbb{Q}\times\mathbb{P})$ and $(\twoheadleftarrow_n)\subseteq\mathsf{P}(\mathbb{R}\times\mathbb{Q})$ define another arrow-sequence \[{\leftarrowtail_n}\ =\ {{\twoheadleftarrow_n}\circ{\barwedge}\circ{\leftarrow_n}}.\]
    Moreover, if $(\leftarrow_n)$ and $(\twoheadleftarrow_n)$ are both thin then so is $(\leftarrowtail_n)$ and the resulting strong refiners ${\leftarrow}={\bigcup_{n\in\omega}{\langle\!\leftarrow_n}}$, ${\twoheadleftarrow}={\bigcup_{n\in\omega}{\langle\!\twoheadleftarrow_n}}$ and ${\leftarrowtail}={\bigcup_{n\in\omega}{\langle\!\leftarrowtail_n}}$ will then satisfy
    \[{\leftarrowtail}\ =\ {\twoheadleftarrow*\leftarrow}.\]
\end{proposition}

\begin{proof}
    As $\mathbb{R}^{\leftarrowtail_n\barwedge}=\mathbb{R}^{\twoheadleftarrow_n\barwedge\leftarrow_n\barwedge}=\mathbb{Q}^{\leftarrow_n\barwedge}=\mathbb{P}$, each $\leftarrowtail_n$ is an arrow.  Also, for each $n\in\omega$,
    \begin{align*}
        {\leftarrowtail_{n+1}}\ =\ {{\twoheadleftarrow_{n+1}}\circ{\barwedge}\circ{\leftarrow_{n+1}}}&\subseteq\ {{\leq}\circ{\twoheadleftarrow_n}\circ{\geq}\circ{\barwedge}\circ{\leq}\circ{\leftarrow_n}\circ{\geq}}\\
        &\subseteq\ {{\leq}\circ{\twoheadleftarrow_n}\circ{\barwedge}\circ{\leftarrow_n}\circ{\geq}}\ =\ {{\leq}\circ{\leftarrowtail_n}\circ{\geq}},
    \end{align*}
    so $(\leftarrowtail_n)$ is an arrow sequence.  Next note that if
    $r\mathrel{\langle\!\twoheadleftarrow_n}q\mathrel{\langle\!\leftarrow_n}p$ then
    \[p^{\barwedge\rightarrowtail_n}=p^{\barwedge\rightarrow_n\barwedge\twoheadrightarrow_n}\subseteq q^{\vartriangleright\barwedge\twoheadrightarrow_n}\subseteq q^{\barwedge\twoheadrightarrow_n}\vartriangleleft r,\]
    so ${\langle\!\twoheadleftarrow_n}\circ{\langle\!\leftarrow_n}\ \subseteq\ {\langle\!\leftarrowtail_n}$.
    If $(\leftarrow_n)$ and $(\twoheadleftarrow_n)$ are thin then this means ${\leftarrowtail'}=\bigcup_{n\in\omega}{\langle\!\twoheadleftarrow_n}\circ{\langle\!\leftarrow_n}$ is a refiner contained in both ${\leftarrowtail}$ and ${{\twoheadleftarrow}\circ{\leftarrow}}\subseteq{{\twoheadleftarrow}*{\leftarrow}}$.  This means $(\leftarrowtail_n)$ is thin and
    \[{\leftarrowtail}={{\vartriangleright}*{\leftarrowtail'}}={{\twoheadleftarrow}*{\leftarrow}}.\qedhere\]
\end{proof}

Let us call an arrow-sequence $(\leftarrow_n)\subseteq\mathsf{P}(\mathbb{Q}\times\mathbb{P})$ \emph{bi-thin} if $(\leftarrow_n)$ and $(\rightarrow_n)$ are thin.  We then have a homeomorphism $\mathsf{S}_\leftarrow=\bigcap_{n\in\omega}\overline{\leftarrow_{n\mathsf{S}}}$ from $\mathsf{S}\mathbb{P}$ to $\mathsf{S}\mathbb{Q}$, where ${\leftarrow}=\bigcup_{n\in\omega}\langle\!\leftarrow$.

\begin{corollary}\label{BackAndForthArrows}
    Assume we have arrow-sequences $(\leftarrow_n)\subseteq\mathsf{P}(\mathbb{P}\times\mathbb{P})$, $(\twoheadleftarrow_n)\subseteq\mathsf{P}(\mathbb{Q}\times\mathbb{Q})$ and $(\sqsupset_n)\subseteq\mathsf{P}(\mathbb{Q}\times\mathbb{P})$ and set ${\leftarrow}=\bigcup_{n\in\omega}\langle\!\leftarrow_n$, ${\twoheadleftarrow}=\bigcup_{n\in\omega}\langle\!\twoheadleftarrow_n$ and ${\sqsupset}={\bigcup_{n\in\omega}\langle\sqsupset_n}$.  Then
    \begin{align*}
        (\sqsupset_n)\text{ is bi-thin, }(\leftarrow_n)\text{ is thin and }\forall n\ ({\twoheadleftarrow_n}\ \subseteq\ {{\sqsubset_n}\circ{\leftarrow_n}\circ{\sqsupset_n}})&\\
        \Rightarrow\quad(\twoheadleftarrow_n)\text{ is thin and }\ \mathsf{S}_{\twoheadleftarrow}={\mathsf{S}_{\sqsupset}^{-1}\circ{\mathsf{S}_\leftarrow}\circ{\mathsf{S}_\sqsupset}}&.
    \end{align*}
\end{corollary}

\begin{proof}
    This follows from \Cref{ArrowComposition} and the observation that \[{{\sqsubset_n}\circ{\leftarrow_n}\circ{\sqsupset_n}}\ \ \subseteq\ \ {{\sqsubset_n}\circ{\barwedge}\circ{\leftarrow_n}\circ{\barwedge}\circ{\sqsupset_n}}.\qedhere\]
\end{proof}

The following shows how a back and forth argument gives rise to bi-thin arrow-sequences.

\begin{lemma}\label{BackAndForth}
Assume that we have regular $\omega$-posets $\mathbb{P}$ and $\mathbb{Q}$ together with 
\begin{enumerate}
    \item coinitial sequences $(C_n)\subseteq\mathsf{C}\mathbb{P}$ and $(D_n)\subseteq\mathsf{C}\mathbb{Q}$ of finite caps, and
    \item $\wedge$-preserving co-surjective $\sqsupset_n\ \subseteq D_n\times C_n$ and $\sqni_n\ \subseteq C_n\times D_{n+1}$ such that
    \begin{equation}\label{subequality}
        \sqsupset_n\circ\sqni_n\ \subseteq\ \geq_\mathbb{Q}\qquad\text{and}\qquad\sqni_n\circ\sqsupset_{n+1}\ \subseteq\ \geq_\mathbb{P},
    \end{equation}
\end{enumerate}
for all $n\in\omega$.  Then we have a bi-thin arrow-sequence $(\leftarrow_n)$ defined by
\[{\leftarrow_{2n}}={\sqsupset_n}\qquad\text{and}\qquad{\leftarrow_{2n+1}}={\sqin_n}.\]
\end{lemma}

\begin{proof}
    The co-surjectivity of $\sqsupset_n$ yields
    \[{\sqin_n}\ \subseteq\ {\sqin_n}\circ{\sqsubset_n}\circ{\sqsupset_n}\ \subseteq\ {\leq}\circ{\sqsupset_n}\ \subseteq\ {\leq}\circ{\sqsupset_n}\circ{\geq},\]
    showing that ${\leftarrow_{2n+1}}={\sqin_n}\leq{\sqsupset_n}={\leftarrow_{2n}}$.  Likewise, the co-surjectivity of $\sqni_n$ yields ${\leftarrow_{2n+2}}={\sqsupset_{n+1}}\leq{\sqin_n}={\leftarrow_{2n+1}}$, showing that $(\leftarrow_n)$ is indeed an arrow-sequence.  For bi-thinness note that, for any $A\in\mathsf{C}\mathbb{Q}$, we have $B\in\mathsf{C}\mathbb{Q}$ and $n\in\omega$ with $D_n\vartriangleleft_{D_n}B\vartriangleleft_BA$, thanks to the regularity of $\mathbb{Q}$ and the co-initiality of $(D_n)$.  Then \Cref{DoubleTriangle} implies that ${\vartriangleright_B}\circ{\vartriangleright_{D_n}}\circ{\sqsupset_n}\subseteq\langle\!\leftarrow_{2n}$ and hence $C_n\subseteq A^{\langle\!\leftarrow_{2n}}$.  Likewise, for every $A\in\mathsf{C}\mathbb{P}$, we have $n\in\omega$ with $D_{n+1}\subseteq A^{\langle\sqni_n}=A^{\langle\rightarrow_{2n+1}}$, showing that $(\leftarrow_n)$ is indeed bi-thin.
\end{proof}

\section{Fra\"iss\'e Theory for Relational Categories of Graphs}

Given any category $\mathbf{K}$, we denote each hom-set of morphisms from $G$ to $H$ by $\mathbf{K}_G^H$.  Let us further denote the family of all morphisms in a category $\mathbf{K}$ having $G$ as their domain or codomain respectively by \[\mathbf{K}_G^\bullet=\bigcup_{H\in\mathbf{K}}\mathbf{K}_G^H\qquad\text{and}\qquad\mathbf{K}^G_\bullet=\bigcup_{H\in\mathbf{K}}\mathbf{K}^G_H.\]
We denote the family of all morphisms in $\mathbf{K}$ with arbitrary domain and codomain by
\[\mathbf{K}^\bullet_\bullet=\bigcup_{G\in\mathbf{K}}\mathbf{K}_G^\bullet=\bigcup_{G\in\mathbf{K}}\mathbf{K}_\bullet^G.\]

We call a collection of morphisms $W\subseteq\mathbf{K}_\bullet^\bullet$ \emph{wide} if every object $G\in\mathbf{K}$ is the codomain of at least one morphism in $W$, i.e. $W\cap\mathbf{K}^G_\bullet\neq\emptyset$.  We call a collection of morphisms $I\subseteq\mathbf{K}_\bullet^\bullet$ an \emph{ideal} if it is closed under composition with arbitrary morphisms, i.e. if \[I=\mathbf{K}_\bullet^\bullet\circ I\circ\mathbf{K}_\bullet^\bullet=\{{\sqsupset\circ\nil\circ\sqni}:{\sqsupset}\in\mathbf{K}^\bullet_G,{\nil}\in I\cap\mathbf{K}^G_H\text{ and }{\sqni}\in\mathbf{K}^H_\bullet\}.\]

We say $\mathbf{K}$ is \emph{directed} if, for all $G,H\in\mathbf{K}$, we have $F\in\mathbf{K}$ with $\mathbf{K}^G_F\neq\emptyset\neq\mathbf{K}^H_F$.  This is equivalent to saying $\bigcup\{\mathbf{K}^\bullet_F:\mathbf{K}^G_F\neq\emptyset\}$ is wide, for all $G\in\mathbf{K}$.  We say a collection of objects $C\subseteq\mathbf{K}$ is \emph{coinitial} if, for all $G\in\mathbf{K}$, we have $H\in C$ with $\mathbf{K}^G_H\neq\emptyset$.  Again this is equivalent to saying that $\bigcup_{G\in C}\mathbf{K}^\bullet_G$ is wide.  Note that if $\mathbf{K}$ has a directed subcategory whose objects are coinitial in $\mathbf{K}$ then $\mathbf{K}$ itself must also be directed.

\subsection{Relational Categories}

We will exclusively consider categories $\mathbf{K}$ where
\begin{enumerate}
    \item each object $G\in\mathbf{K}$ is a finite set $\dot{G}$ together with some relation(s) on $\dot{G}$.
    \item each hom-set $\mathbf{K}_G^H\subseteq \mathsf{P}(\dot{H}\times\dot{G})$ consists of co-bijective relations `from $\dot{G}$ to $\dot{H}$'.
    \item the categorical product is the usual relational composition $\circ$, and
    \item the identity morphism on each $G\in\mathbf{K}$ is the equality relation $=_{\dot{G}}$ on $\dot{G}$.
\end{enumerate}
In particular, there is a canonical ordering on each hom-set $\mathbf{K}_G^H$ given by inclusion $\subseteq$.  As morphisms are co-injective, $\mathbf{K}_G^H\neq\emptyset$ only if $\dot{G}$ has at least as many elements as $\dot{H}$ and, if they have the same number of elements, then $\mathbf{K}_G^H$ consists entirely of functions from $\dot{G}$ to $\dot{H}$.  Following standard practice, from now on we will usually identify the structure $G$ with the underlying set $\dot{G}$ as long as no confusion will result.

For any $G\in\mathbf{K}$, we say a relation ${\leftarrow}\subseteq G\times G$ \emph{subfactors} into ${\sqsupset},{\sqni}\subseteq\mathbf{K}^G_H$ if
\[=_H\ \ \subseteq\ \ \sqsubset\circ\leftarrow\circ\sqni.\]
More explicitly this means that, for all $h\in H$, we have $g\sqsupset h$ and $g'\sqni h$ with $g\leftarrow g'$.  In this case we say that the relation $\leftarrow$ is \emph{$\mathbf{K}$-subfactorisable}.

\begin{proposition}\label{SurjectiveSubfactor}
    Every $\mathbf{K}$-subfactorisable relation is surjective and co-surjective.
\end{proposition}

\begin{proof}
    Take any $G\in\mathbf{K}$ and $\mathbf{K}$-subfactorisable ${\leftarrow}\subseteq G\times G$, so we have ${\sqsupset},{\sqni}\subseteq\mathbf{K}^G_H$ with ${=_H}\subseteq{\sqsubset\circ\leftarrow\circ\sqni}$.  As we are assuming all morphisms in $\mathbf{K}$ are co-injective, for every $g\in G$, we have $h\in H$ with $h^\sqsubset=\{g\}$.  As $h\sqsubset g'\leftarrow g''\sqni h$, for some $g',g''\in G$, the only possibility is $g=g'\leftarrow g''$, showing that $\leftarrow$ is surjective.  Now just note $\mathbf{K}$-subfactorisability is a symmetric condition and so the same applies to $\rightarrow$.
\end{proof}

Let us call $\leftarrow$ \emph{widely $\mathbf{K}$-subfactorisable} if $\sqsubset\circ\leftarrow\circ\sqsupset$ is subfactorisable, for all ${\sqsupset}\in\mathbf{K}^G_\bullet$.

\begin{proposition}\label{EqualitySubfactor}
    For every $G\in\mathbf{K}$, the equality relation $=_G$ is widely $\mathbf{K}$-subfactorisable.
\end{proposition}

\begin{proof}
    As every ${\sqsupset}\in\mathbf{K}^G_H$ is co-surjective, ${=_H}\subseteq{{=_H}\circ{\sqsubset}\circ{=_G}\circ{\sqsupset}\circ{=_H}}$, showing that ${\sqsubset}\circ{=_G}\circ{\sqsupset}$ subfactors into $=_H$ and $=_H$.
\end{proof}

We call a morphism ${\nil}\in\mathbf{K}^E_F$ \emph{amalgamable} if
\[{\sqsupset}\in\mathbf{K}^F_G\ \text{and}\ {\sqni}\in\mathbf{K}^F_H\quad\Rightarrow\quad\exists\,{\sqsupset'}\in\mathbf{K}^G_I\ \exists\,{\sqni'}\in\mathbf{K}^H_I\ ({\nil\circ\sqsupset\circ\sqsupset'}\ =\ {\nil\circ\sqni\circ\sqni'}).\]
More generally, we call ${\nil}\in\mathbf{K}^E_F$ \emph{subamalgamable} if we have ${\overline\nil}\in\mathbf{K}^E_F$ such that ${\nil}\subseteq{\overline\nil}$ and, whenever ${\sqsupset}\in\mathbf{K}^F_G$ and ${\sqni}\in\mathbf{K}^F_H$, we have ${\sqsupset'}\in\mathbf{K}^G_I$ and ${\sqni'}\in\mathbf{K}^H_I$ satisfying
\[{{\nil}\circ{\sqsupset}\circ{\sqsupset'}}\ \subseteq\ {{\overline\nil}\circ{\sqni}\circ{\sqni'}}.\]

\begin{lemma}\label{WideSubfactorisability}
    For any ${\leftarrow}\subseteq F\times F$ and subamalgamable ${\nil}\in\mathbf{K}^F_G$ witnessed by $\overline{\nil}\in\mathbf{K}^F_G$,
    \[\lin\circ\leftarrow\circ\nil\text{ is $\mathbf{K}$-subfactorisable}\qquad\Rightarrow\qquad{{\overline\lin}\circ{\leftarrow}\circ{\overline\nil}}\text{ is widely $\mathbf{K}$-subfactorisable}.\]
\end{lemma}

\begin{proof}
    Assume that $\lin\circ\leftarrow\circ\nil$ is $\mathbf{K}$-subfactorisable and take any ${\sqsupset}\in\mathbf{K}^G_H$.  So we have ${\sqni},{\sqnii}\in\mathbf{K}^G_I$ such that ${=_I}\subseteq{\sqin\circ\lin\circ\leftarrow\circ\nil\circ\sqnii}$.  We then have ${\sqsupset_J}\in\mathbf{K}^H_J$ and ${\sqni_J}\in\mathbf{K}^I_J$ with ${\nil\circ\sqni\circ\sqni_J}\subseteq{{\overline\nil}\circ{\sqsupset}\circ{\sqsupset_J}}$, as well as ${\sqsupset_L}\in\mathbf{K}_L^H$ and ${\sqni_L}\in\mathbf{K}_L^J$ with ${\nil\circ\sqnii\circ\sqni_J\circ\sqni_L}\subseteq{{\overline\nil}\circ{\sqsupset}\circ{\sqsupset_L}}$.

    For every $l\in L$, take $j_l\in J$ and $i_l\in I$ with $i_l\sqni_Jj_l\sqni_Ll$.  As ${=_I}\subseteq{\sqin\circ\lin\circ\leftarrow\circ\nil\circ\sqnii}$, we have $f_l,f_l'\in F$ with $i_l\mathrel{\sqin\circ\lin}f_l\leftarrow f_l'\nil\circ\sqnii i_l$.  Then ${\nil\circ\sqni\circ\sqni_J}\subseteq{{\overline\nil}\circ{\sqsupset}\circ{\sqsupset_J}}$ implies that we have $h_l\in H$ with $f_l\mathrel{{\overline\nil}\circ{\sqsupset}}h_l\sqsupset_Jj_l$, while ${\nil\circ\sqnii\circ\sqni_J\circ\sqni_L}\subseteq{{\overline\nil}\circ{\sqsupset}\circ{\sqsupset_L}}$ implies that we have $h_l'\in H$ with $f_l'\mathrel{{\overline\nil}\circ{\sqsupset}}h'_l\sqsupset_Ll$.  This shows that ${\sqsubset}\circ{\overline\lin}\circ{\leftarrow}\circ{\overline\nil}\circ{\sqsupset}$ subfactors into $\sqsupset_J\circ\sqni_L$ and $\sqsupset_L$ in $\mathbf{K}^H_L$.  As ${\sqsupset}\in\mathbf{K}^G_H$ was arbitrary, this shows that ${{\overline\lin}\circ{\leftarrow}\circ{\overline\nil}}$ is widely $\mathbf{K}$-subfactorisable.
\end{proof}

By a \emph{graph} we mean a finite set $G$ together with a reflexive symmetric binary \emph{edge relation} ${\sqcap_G}$ on $G$, which we denote just by $\sqcap$ when $G$ is clear.  We call a relation ${\sqsupset}\subseteq{H\times G}$ from one graph $G$ to another graph $H$ \emph{edge-preserving} if
\[\tag{Edge-Preserving}{{\sqsupset}\circ{\sqcap}\circ{\sqsubset}}\ \subseteq\ {\sqcap}.\]
Graphs form a category $\mathbf{G}$ with hom-sets
\[\mathbf{G}_F^E=\{{\sqsupset}\subseteq E\times F:{\sqsupset}\text{ is co-bijective and edge-preserving}\}.\]

Our primary interest is in subcategories $\mathbf{K}$ of $\mathbf{G}$ where we restrict both the shape of the graphs and the types of morphisms between them (which represent open covers and refinements between them in certain classes of compacta).  Accordingly, let us make the following standing assumption for the rest of this subsection:
\begin{center}
    \textbf{$\mathbf{K}$ is a subcategory of the graph category $\mathbf{G}$.}
\end{center}

We call a morphism ${\nil}\in\mathbf{K}^E_F$ \emph{near-amalgamable} if
\[{\sqsupset}\in\mathbf{K}^F_G\ \text{and}\ {\sqni}\in\mathbf{K}^F_H\quad\Rightarrow\quad\exists\,{\sqsupset'}\in\mathbf{K}^G_K\ \exists\,{\sqni'}\in\mathbf{K}^H_K\ ({\nil}\circ{\sqsupset}\circ{\sqsupset'}\circ{\sqcap}\circ{\sqin'}\circ{\sqin}\circ{\lin}\ \subseteq\ {\sqcap}).\]

\begin{proposition}\label{SubImpliesNear}
    For any ${\nil}\in\mathbf{K}^F_G$,
    \[{\nil}\text{ is subamalgamable}\quad\Rightarrow\quad{\nil}\text{ is near-amalgamable}.\]
\end{proposition}

\begin{proof}
    Take ${\overline\nil}\in\mathbf{K}^F_G$ with ${\nil}\subseteq{\overline\nil}$ such that, for any ${\sqsupset}\in\mathbf{K}^F_G$ and ${\sqni}\in\mathbf{K}^F_H$, we have ${\sqsupset'}\in\mathbf{K}^G_K$ and ${\sqni'}\in\mathbf{K}^H_K$ with ${\nil\circ\sqsupset\circ\sqsupset'}\ \subseteq\ {\overline{\nil}\circ\sqni\circ\sqni'}$.  Now just note that this implies
        \[{\nil}\circ{\sqsupset}\circ{\sqsupset'}\circ{\sqcap}\circ{\sqin'}\circ{\sqin}\circ{\lin}\quad\subseteq\quad{\overline\nil}\circ{\sqni}\circ{\sqni'}\circ{\sqcap}\circ{\sqin'}\circ{\sqin}\circ{\overline\lin}\quad\subseteq\quad{\sqcap},\]
    as ${\overline\nil}\circ{\sqni}\circ{\sqni'}$ is $\sqcap$-preserving.
\end{proof}

We call ${\nil}\in\mathbf{K}^F_G$ \emph{neighbourly} if we have $\overline\nil\in\mathbf{K}^F_G$ such that
\[{{\nil}\circ{\sqcap}}\ \subseteq\ {\overline\nil}.\]

\begin{proposition}\label{NearImpliesSub}
    For any ${\nil},{\dashv}\in\mathbf{K}^E_F$,
    \[{\nil}\text{ is neighbourly and }{\dashv}\text{ is near-amalgamable}\quad\Rightarrow\quad{\nil\circ\dashv}\text{ is subamalgamable}.\]
\end{proposition}

\begin{proof}
    Assume $\nil$ is neighbourly, as witnessed by ${\overline\nil}$, and that $\dashv$ is near-amalgamable.  Take any ${\sqsupset}\in\mathbf{K}^F_G$ and ${\sqni}\in\mathbf{K}^F_H$, so we have ${\sqsupset'}\in\mathbf{K}^G_K$ and ${\sqni'}\in\mathbf{K}^H_K$ with
    \[{\dashv}\circ{\sqsupset}\circ{\sqsupset'}\circ{\sqcap}\circ{\sqin'}\circ{\sqin}\circ{\vdash}\ \subseteq\ {\sqcap}\qquad\text{and hence}\qquad{\dashv}\circ{\sqsupset}\circ{\sqsupset'}\ \subseteq\ {\sqcap}\circ{\dashv}\circ{\sqni}\circ{\sqni'}.\]
    Indeed, if $f\mathrel{{\dashv}\circ{\sqsupset}\circ{\sqsupset'}}h$ then, as all morphisms are co-surjective, we also have $f'\in F$ with $f'\mathrel{{\dashv}\circ{\sqni}\circ{\sqni'}}h$ and hence $f\mathrel{\sqcap}f'$.  From this it follows that
    \[{\nil}\circ{\dashv}\circ{\sqsupset}\circ{\sqsupset'}\ \subseteq\ {\nil}\circ{\sqcap}\circ{\dashv}\circ{\sqni}\circ{\sqni'}\ \subseteq\ {\overline\nil}\circ{\dashv}\circ{\sqni}\circ{\sqni'},\]
    showing that ${\overline\nil}\circ{\dashv}$ witnesses the subamalgamability of ${\nil}\circ{\dashv}$.
\end{proof}

\begin{proposition}
    If the subamalgamable morphisms are wide in $\mathbf{K}$ then so are the near-amalgamable morphisms.  Conversely, if both the near-amalgamable and neighbourly morphisms are wide in $\mathbf{K}$ then so are the subamalgamable morphisms.
\end{proposition}

\begin{proof}
    Immediate from \Cref{SubImpliesNear} and \Cref{NearImpliesSub}.
\end{proof}

We are particularly interested in graphs arising from levels of an $\omega$-poset $\mathbb{P}$.  Specifically, on each level $\mathbb{P}_n$ we consider the edge relation $\wedge_n$ coming from the restriction of $\wedge$, i.e.
\[{\wedge_n}={\wedge}\cap(\mathbb{P}_n\times\mathbb{P}_n).\]
Each restriction $\geq^m_n$ of the order on $\mathbb{P}$ to levels $m\leq n$ is then an edge-preserving relation.  By a \emph{$\mathbf{K}$-poset} we mean an $\omega$-poset $\mathbb{P}$ such that its levels are all graphs in $\mathbf{K}$ and the restrictions of the order relation to these levels are all morphisms in $\mathbf{K}$, i.e.~such that $(\mathbb{P}_n,\wedge_n)\in\mathbf{K}$ and ${\geq^m_n}\in\mathbf{K}^m_n$ whenever $m\leq n$, where $\mathbf{K}^m_n=\mathbf{K}^{\mathbb{P}_m}_{\mathbb{P}_n}=\mathbf{K}^{(\mathbb{P}_m,\wedge_m)}_{(\mathbb{P}_n,\wedge_n)}$.

\subsection{Subabsorption}

In this subsection we make the following standing assumption:

\begin{center}
    \textbf{$\mathbb{P}$ is a $\mathbf{K}$-poset for some fixed subcategory $\mathbf{K}$ of the graph category $\mathbf{G}$.}
\end{center}

We call ${\nil}\in\mathbf{K}^m_n$ \emph{subabsorbing} if ${\nil}\subseteq{\geq^m_n}$ and, for all ${\sqsupset}\subseteq\mathbf{K}^{\mathbb{P}_n}_G$, we have ${\sqsupset'}\in\mathbf{K}^G_{\mathbb{P}_{n'}}$ with
\[\tag{Subabsorbing}{{\nil}\circ{\sqsupset}\circ{\sqsupset'}}\ \subseteq\ {\geq^m_{n'}}.\]

\begin{proposition}\label{SubNearAbsorbing}
     Every subabsorbing ${\nil}\in\mathbf{K}^m_n$ is near-amalgamable.
\end{proposition}

\begin{proof}
    If ${\nil}\in\mathbf{K}^m_n$ is subabsorbing then, for any ${\sqsupset}\subseteq\mathbf{K}^{\mathbb{P}_n}_G$ and ${\sqni}\subseteq\mathbf{K}^{\mathbb{P}_n}_G$, we have ${\sqsupset'}\in\mathbf{K}^G_{\mathbb{P}_{n'}}$ and ${\sqni'}\in\mathbf{K}^G_{\mathbb{P}_{n''}}$ such that ${{\nil}\circ{\sqsupset}\circ{\sqsupset'}}\subseteq{\geq^m_{n'}}$ and ${{\nil}\circ{\sqni}\circ{\sqni'}}\subseteq{\geq^m_{n''}}$.  Making $n'$ or $n''$ larger if necessary we may assume $n'=n''$ and then
    \[{{\nil}\circ{\sqsupset}\circ{\sqsupset'}\circ{\wedge}\circ{\sqin'}\circ{\sqin}\circ{\lin}}\ \ \subseteq\ \ {\geq^m_{n'}}\circ{\wedge}\circ{\leq^m_{n'}}\ \ \subseteq\ \ {\wedge},\]
    showing that $\nil$ is near-amalgamable.
\end{proof}

Call ${\leftarrow}\subseteq\mathbb{P}_m\times\mathbb{P}_m$ \emph{$\mathbf{K}$-like} if it subfactors into ${\sqsupset}\in\mathbf{K}^m_n$ and $\geq^m_n$, for some $m\geq n$, i.e.
\[\tag{$\mathbf{K}$-Like}{=_{\mathbb{P}_n}}\ \subseteq\ {\sqsubset\circ\leftarrow\circ\geq^m_n}.\]

\begin{lemma}\label{KSubfactorKLike}
    If ${\nil}\in\mathbf{K}^k_m$ is subabsorbing then, for any ${\leftarrow}\subseteq\mathbb{P}_k\times\mathbb{P}_k$,
    \[{\lin\circ\leftarrow\circ\nil}\text{ is $\mathbf{K}$-subfactorisable}\qquad\Rightarrow\qquad\leftarrow\text{ is $\mathbf{K}$-like}.\]
\end{lemma}

\begin{proof}
    Assume ${\lin\circ\leftarrow\circ\nil}$ subfactors into ${\sqsupset'},{\sqsupset}\in\mathbf{K}^{\mathbb{P}_m}_G$.  As $\nil$ is subabsorbing, we have $n>m$ and ${\sqni}\in\mathbf{K}^G_{\mathbb{P}_n}$ such that ${\nil\circ\sqsupset\circ\sqni}\subseteq{\geq^k_n}$ and hence
    \[{=_{\mathbb{P}_n}}\ \ \subseteq\ \ {{\sqin}\circ{=_G}\circ{\sqni}}\ \ \subseteq\ \ {{\sqin}\circ{\sqsubset'}\circ{\lin}\circ{\leftarrow}\circ{\nil}\circ{\sqsupset}\circ{\sqni}}\ \ \subseteq\ \ \sqin\circ\sqsubset'\circ\lin\circ\leftarrow\circ\geq^k_n.\]
    This shows that ${\leftarrow}$ subfactors into $\nil\circ\sqsupset'\circ\sqni$ and $\geq^k_n$.
\end{proof}

We say that $\mathbb{P}$ is \emph{$\mathbf{K}$-coinitial} if its level are coinitial in $\mathbf{K}$ i.e.~ if, for all $G\in\mathbf{K}$, we have $n\in\omega$ with $\mathbf{K}^G_{\mathbb{P}_n}\neq\emptyset$.  We now come the central definition of this section.

\begin{definition}
    We say $\mathbb{P}$ is \emph{$\mathbf{K}$-subabsorbing} if, for all $m\in\omega$, we have subabsorbing ${\nil}\in\mathbf{K}^m_n$, for some $n>m$.  If $\mathbb{P}$ is also $\mathbf{K}$-coinitial then we say $\mathbb{P}$ is \emph{$\mathbf{K}$-sub-Fra\"iss\'e}.
\end{definition}

\begin{proposition}\label{SubabsorbingDirected}
    If $\mathbb{P}$ is $\mathbf{K}$-subabsorbing and $\mathbf{K}$ is directed then $\mathbb{P}$ is also $\mathbf{K}$-coinitial and hence $\mathbf{K}$-sub-Fra\"iss\'e.  In this case, near-amalgamable morphisms are also wide in $\mathbf{K}$.
\end{proposition}

\begin{proof}
    If $\mathbb{P}$ is $\mathbf{K}$-subabsorbing then, for any $k\in\omega$, we have subabsorbing ${\dashv}\in\mathbf{K}^k_m$.  If $\mathbf{K}$ is also directed then, for any $G\in\mathbf{K}$, we have $H\in\mathbf{K}$ such that we have some ${\sqsupset}\in\mathbf{K}^{\mathbb{P}_m}_H$ and ${\sqni}\in\mathbf{K}^G_H$.  As ${\nil}$ is subabsorbing, we then have ${\sqsupset'}\in\mathbf{K}^H_{\mathbb{P}_n}$ and hence ${\sqni\circ\sqsupset'}\in\mathbf{K}^G_{\mathbb{P}_n}$, showing that $\mathbb{P}$ is $\mathbf{K}$-coinitial.

    Now for any $G\in\mathbf{K}$, $\mathbf{K}$-coinitiality yields ${\dashv}\in\mathbf{K}^G_{\mathbb{P}_m}$, for some $m\in\omega$.  Then $\mathbf{K}$-subabsorption yields subabsorbing ${\nil}\in\mathbf{K}^m_n$, for some $n>m$.  By \Cref{SubNearAbsorbing}, $\nil$ is near-amalgamble and hence so is $\dashv\circ\nil$, i.e.~near-amalgamable morphisms are wide.
\end{proof}

If $\geq^m_n$ is a witness for the subamalgamability of ${\nil}\in\mathbf{K}^m_n$, we call $\nil$ \emph{$\mathbb{P}$-subamalgamable}.  Likewise, we say that ${\nil}\in\mathbf{K}^m_n$ is \emph{$\mathbb{P}$-neighbourly} if ${\nil}\circ{\wedge_n}\subseteq{\geq^m_n}$, i.e.~if ${\nil}\subseteq{\vartriangleright_{\mathbb{P}_n}}$.   We say that $\mathbb{P}$ is \emph{$\mathbf{K}$-regular} if, for all $m\in\omega$, we have $\mathbb{P}$-neighbourly ${\nil}\in\mathbf{K}^m_n$, for some $n>m$.  In particular, $\mathbf{K}$-regularity implies regularity as defined earlier for general $\omega$-posets.

\begin{proposition}\label{SubabsorbingRegular}
    Assume $\mathbb{P}$ is $\mathbf{K}$-subabsorbing.
    \begin{enumerate}
        \item If the neighbourly morphisms are wide in $\mathbf{K}$ then $\mathbb{P}$ is $\mathbf{K}$-regular.
        \item If $\mathbb{P}$ is $\mathbf{K}$-regular then, for all $m\in\omega$, we have $\mathbb{P}$-subamalgamable ${\nil}\in\mathbf{K}^m_n$.
    \end{enumerate}
\end{proposition}

\begin{proof}
    If $\mathbb{P}$ is $\mathbf{K}$-subabsorbing then, for any $k\in\omega$, we have subabsorbing ${\dashv}\in\mathbf{K}^k_m$.  If the neighbourly morphisms are also wide in $\mathbf{K}$ then we have neighbourly ${\nil}\in\mathbf{K}^{\mathbb{P}_m}_G$, as witnessed by some larger $\overline{\nil}\in\mathbf{K}^{\mathbb{P}_m}_G$.  As ${\nil}$ is subabsorbing, we then have ${\sqsupset}\in\mathbf{K}^G_{\mathbb{P}_n}$ with ${\dashv}\circ{\overline{\nil}}\circ{\sqsupset}\subseteq{\geq^m_n}$.  As $\sqsupset$ is necessarily co-surjective and $\sqcap$-preserving then
    \[{\dashv}\circ{\nil}\circ{\sqsupset}\circ{\wedge_n}\ \ \subseteq\ \ {\dashv}\circ{\nil}\circ{\sqcap_G}\circ{\sqsupset}\ \ \subseteq\ \ {\dashv}\circ{\overline{\nil}}\circ{\sqsupset}\ \ \subseteq\ \ {\geq^m_n}.\]
    This shows that ${\dashv}\circ{\nil}\circ{\sqsupset}$ is $\mathbb{P}$-neighbourly, which in turn shows that $\mathbb{P}$ is $\mathbf{K}$-regular.

    For (2) just note that, for all $m\in\omega$, we can compose any $\mathbb{P}$-neighbourly ${\dashv}\in\mathbf{K}^m_l$ with any subabsorbing and hence near-amalgamable ${\nil}\in\mathbf{K}^l_n$ to obtain a subamalgamable morphism ${\dashv}\circ{\nil}\in\mathbf{K}^m_n$, by \Cref{SubNearAbsorbing,NearImpliesSub}.  Indeed, by the proof of \Cref{NearImpliesSub}, ${\geq^m_l}\circ{\nil}$ is a witness for the subamalgamability of ${\dashv}\circ{\nil}$, and hence so too is ${\geq^m_n}\supseteq{\geq^m_l}\circ{\nil}$, showing that ${\dashv}\circ{\nil}$ is $\mathbb{P}$-subamalgamable.
\end{proof}

\begin{proposition}\label{SubFraisseImpliesSubAmalgamation}
    If $\mathbb{P}$ is $\mathbf{K}$-sub-Fra\"iss\'e and $\mathbf{K}$-regular then the subamalgamable and neighbourly morphisms are wide in $\mathbf{K}$.
\end{proposition}

\begin{proof}
    For any $G\in\mathbf{K}$, $\mathbf{K}$-coinitiality yields ${\dashv}\in\mathbf{K}^G_{\mathbb{P}_m}$, for some $m\in\omega$.  Then $\mathbf{K}$-regularity yields ($\mathbb{P}$-)neighbourly ${\nil}\in\mathbf{K}^m_n$, for some $n>m$.  Then $\dashv\circ\nil$ is also neighbourly, showing that the neighbourly morphisms are wide.  Likewise, the above result yields subamalgamable ${\nil}\in\mathbf{K}^m_n$, for some $n>m$, and then $\dashv\circ\nil$ is also subamalgamable.
\end{proof}

We call a morphism ${\sqsupset}\in\mathbf{K}^F_G$ \emph{edge-witnessing} if
\[\tag{Edge-Witnessing}{\sqcap_F}\ \ \subseteq\ \ {{\sqsupset}\circ{\sqsubset}}.\]
Note that if ${\sqsupseteq}\supseteq{\sqsupset}$ then $\sqsupseteq$ is necessarily edge-witnessing too.  Moreover, for any ${\sqni}\in\mathbf{K}^G_H$, we see that ${\sqsupset}\circ{\sqni}$ is also edge-witnessing, as $\sqni$ is co-injective and hence surjective.

Let us also call $\mathbf{K}$ \emph{widely disjoint} if, for every finite set of graphs $F,G_0,\ldots,G_n\in\mathbf{K}$, we have another graph $H\in\mathbf{K}$ disjoint from $\bigcup_{0\leq k\leq n}G_k$ such that $\mathbf{K}^F_H\neq\emptyset$. 

\begin{theorem}
    If $\mathbf{K}$ is countable and widely disjoint and its subamalgamble and edge-witnessing morphisms are wide then there is a $\mathbf{K}$-subabsorbing graded $\omega$-poset.
\end{theorem}

\begin{proof}
    Like in the proof of \cite[Theorem 3.7]{Ku}, we first recursively define disjoint graphs $(G_n)\subseteq\mathbf{K}$ together with edge-witnessing subamalgamable morphisms ${\sqsupset^m_{m+1}}\in\mathbf{K}^{G_m}_{G_{m+1}}$, as witnessed by ${\sqsupseteq^m_{m+1}}\supseteq{\sqsupset^m_{m+1}}$, such that, for all ${\sqsupset}\in\mathbf{K}^{G_{m+1}}_G$, we have ${\sqsupset'}\in\mathbf{K}^G_{G_n}$ with ${{\sqsupset^m_{m+1}}\circ{\sqsupset}\circ{\sqsupset'}}\subseteq{\sqsupseteq^m_n}$, where ${\sqsupseteq^m_n}={\sqsupseteq^m_{m+1}}\circ\ldots\circ{\sqsupseteq^{n-1}_n}$.
    
    First let $(m_k)\subseteq\omega\setminus\{0\}$ be any sequence where $m_k\leq k$, for all $k$, and each element of $\omega\setminus\{0\}$ appears infinitely many times.  Next we take any $G_0\in\mathbf{K}$, edge-witnessing ${\sqni}\in\mathbf{K}^{G_0}_G$, subamalgamable ${\sqsupset}\in\mathbf{K}^G_H$ and ${\sqnii}\in\mathbf{K}^H_{G_1}$ with $G_1\cap G_0=\emptyset$.  Then we have edge-witnessing subamalgamable ${\sqsupset^0_1}={{\sqni}\circ{\sqsupset}\circ{\sqnii}}\in\mathbf{K}^{G_0}_{G_1}$.  Once $\sqsupset^{k-1}_k$ has been constructed, we construct $\sqsupset^k_{k+1}$ as follows.  Again take any edge-witnessing subamalgamable ${\sqni}\in\mathbf{K}^{G_k}_G$.  As $\sqsupset^{m_k-1}_{m_k}$ is subamalgamable, for any ${\sqsupset}\in\mathbf{K}^{G_{m_k}}_H$, we have ${\sqni'}\in\mathbf{K}^G_I$ and ${\sqsupset'}\in\mathbf{K}^H_I$ with ${{\sqsupset^{m_k-1}_{m_k}}\circ{\sqsupset}\circ{\sqsupset'}}\subseteq{{\sqsupseteq^{m_k-1}_k}\circ{\sqni}\circ{\sqni'}}$.  Taking any ${\sqnii}\in\mathbf{K}^H_{G_{k+1}}$ with $G_{k+1}\cap\bigcup_{j\leq k}G_j=\emptyset$, we then again have edge-witnessing subamalgamable ${\sqsupset^k_{k+1}}={{\sqni}\circ{\sqni'}\circ{\sqnii}}\in\mathbf{K}^{G_k}_{G_{k+1}}$ such that
    \[{{\sqsupset^{m_k-1}_{m_k}}\circ{\sqsupset}\circ{\sqsupset'}\circ{\sqnii}}\ \ \subseteq\ \ {{\sqsupseteq^{m_k-1}_k}\circ{\sqsupset^k_{k+1}}}\ \ \subseteq\ \ {\sqsupseteq^{m_k-1}_{k+1}}.\]
    As $\mathbf{K}$ is countable we can do this with choices of ${\sqsupset}\in\mathbf{K}^{G_{m_k}}_H$ that exhaust $\bigcup_{n\in\omega\setminus\{0\}}\mathbf{K}^{G_n}_\bullet$.  In this way we obtain a sequence of graphs and morphisms with the desired properties.
    
    Setting $\mathbb{P}=\bigcup_{n\in\omega}G_n$ and ${>_\mathbb{P}}=\bigcup_{m<n}\sqsupseteq^m_n$, it follows that $\mathbb{P}$ is a graded $\omega$-poset with $\mathbb{P}_n=G_n$ and ${>^m_n}={\sqsupseteq^m_n}$, whenever $m<n$.  We further claim that ${\sqcap_{G_m}}={\wedge_m}$, for all $m\in\omega$.  Certainly ${\sqcap_{G_m}}\supseteq{\wedge_m}$, as $G_m\ni p,q>r\in G_n$ implies $p\sqsupseteq^m_nr\mathrel{\sqcap_{G_n}}r\sqsubseteq^m_nq$ and hence $p\mathrel{\sqcap_{G_m}}q$, as $\sqsupseteq^m_n$ is $\sqcap$-preserving.  Conversely, if $p\mathrel{\sqcap_{G_m}}q$ then we have $r\in G_{m+1}$ with $p,q>r$, as $\sqsupset^m_{m+1}$ is edge-witnessing, showing that ${\sqcap_{G_m}}\subseteq{\wedge_m}$.  Thus $\mathbb{P}$ is a $\mathbf{K}$-poset that is $\mathbf{K}$-subabsorbing, by construction.
\end{proof}

\subsection{Compatiblity}\label{Compatibility}

Throughout the rest of this section we assume that
\textbf{
\begin{enumerate}
    \item $\mathbf{K}$ is a fixed subcategory of the graph category $\mathbf{G}$.
    \item $\mathbb{P}$ is a $\mathbf{K}$-regular $\mathbf{K}$-subabsorbing $\mathbf{K}$-poset.
\end{enumerate}
}
In particular, $\mathbb{P}$ is regular so $\mathsf{S}\mathbb{P}$ is Hausdorff, by \cite[Corollary 2.40]{BaBiVi}.  As $\mathbf{K}$-morphisms are co-injective, each level $\mathbb{P}_n$ is a minimal cap, by \cite[Proposition 1.21]{BaBiVi}, which thus yields a minimal open cover $\{p_\mathsf{S}:p\in\mathbb{P}_n\}$ of $\mathsf{S}\mathbb{P}$, by \cite[Proposition 2.8]{BaBiVi}.  In particular, $p_\mathsf{S}\neq\emptyset$, for all $p\in\mathbb{P}$, i.e.~$\mathbb{P}$ is also a prime $\omega$-poset.

We are interested in continuous maps on $\mathsf{S}\mathbb{P}$ that are compatible with the given subcategory $\mathbf{K}$.  First, for any finite ${\sqsupset}\subseteq\mathbb{P}\times\mathbb{P}$, let $\sqsupset^\mathbf{C}$ denote the basic open subset of $\mathbf{C}^{\mathsf{S}\mathbb{P}}_{\mathsf{S}\mathbb{P}}$ corresponding to the open subset $\sqsupset^\mathbf{S}$ of strong refiners defined in \eqref{sqniS}, i.e.~
\[{\sqsupset^\mathbf{C}}=\{\phi\in\mathbf{C}^{\mathsf{S}\mathbb{P}}_{\mathsf{S}\mathbb{P}}:{\sqsupset}\subseteq{\sqsupset_\phi}\}=\{\phi\in\mathbf{C}^{\mathsf{S}\mathbb{P}}_{\mathsf{S}\mathbb{P}}:\forall p,q\ (q\sqsupset p\Rightarrow q_\mathsf{S}\supseteq\phi[\overline{p_\mathsf{S}}])\}.\]

\begin{definition}
    The \emph{$\mathbf{K}$-compatible} functions on $\mathsf{S}\mathbb{P}$ are given by
    \[\mathbf{C}_\mathbf{K}={\bigcap_{m\in\omega}\bigcup_{n\geq m}\bigcup_{{\sqsupset}\in\mathbf{K}_n^m}\!\!\!\sqsupset^\mathbf{C}}=\{\phi\in\mathbf{C}_{\mathsf{S}\mathbb{P}}^{\mathsf{S}\mathbb{P}}:\forall\,m\in\omega\ \exists\,{\sqsupset}\in\mathbf{K}_n^m\ ({\sqsupset}\subseteq{\sqsupset_\phi})\}.\]
\end{definition}

\Cref{LevelBasis} tells us that $\mathbf{C}_\mathbf{K}$ has another basis consisting of open sets $\leftarrow_\mathbf{K}$, for relations ${\leftarrow}\subseteq\mathbb{P}_n\times\mathbb{P}_n$ defined on levels of $\mathbb{P}$, where
\[{\leftarrow_\mathbf{K}}={\leftarrow_\mathbf{C}}\cap\mathbf{C}_\mathbf{K}=\{\phi\in\mathbf{C}_\mathbf{K}:\phi\subseteq{\leftarrow_\mathsf{S}}\}.\]
However, many of these basic sets can be empty and our interest in $\mathbf{K}$-like relations stems from the fact that these are the only ones for which $\leftarrow_\mathsf{S}$ can contain any $\phi\in\mathbf{C}_\mathbf{K}$.  We can even revert to a relation below $\leftarrow$ which not just $\mathbf{K}$-like but widely $\mathbf{K}$-subfactorisable (note every widely $\mathbf{K}$-subfactorisable $\leftarrow$ is $\mathbf{K}$-like, by \Cref{KSubfactorKLike}, as $\mathbb{P}$ is $\mathbf{K}$-subabsorbing).

Let us denote the widely $\mathbf{K}$-subfactorisable relations on any $G\in\mathbf{K}$ by
\[G^\mathbf{K}=\{{\leftarrow}\subseteq G\times G:\forall\,{\nil}\in\mathbf{K}^G_H\ ({\lin\circ\leftarrow\circ\nil}\text{ subfactors into some }{\sqsupset},{\sqni}\in\mathbf{K}^H_I)\}.\]

\begin{lemma}\label{KLikeKSubfactorisable}
For any $m,m'\in\omega$, ${\leftarrow}\subseteq\mathbb{P}_m\times\mathbb{P}_m$, ${\leftarrow'}\subseteq\mathbb{P}_{m'}\times\mathbb{P}_{m'}$ and $\phi\in{\leftarrow_\mathbf{K}\cap\leftarrow'_\mathbf{K}}$, we have some $n>\max(m,m')$ and ${\twoheadleftarrow}\subseteq\mathbb{P}_n^\mathbf{K}$ with $\phi\in{\twoheadleftarrow_\mathbf{K}}$, ${\twoheadleftarrow}\leq{\leftarrow}$ and ${\twoheadleftarrow}\leq{\leftarrow'}$
\end{lemma}

\begin{proof}
    By \Cref{ArrowPhiLebesgue}, we have $k>m$ with ${\leftarrow^\phi_k}\vartriangleleft_k{\leftarrow}$ and ${\leftarrow^\phi_k}\vartriangleleft_k{\leftarrow'}$.  By \Cref{SubabsorbingRegular}, we have $\mathbb{P}$-subamalgamable ${\nil}\in\mathbf{K}^k_l$.  As $\phi\in\mathbf{C}_\mathbf{K}$, we have ${\sqsupset}\in\mathbf{K}^l_n$ with ${\sqsupset}\subseteq{\sqsupset_\phi}$.  By \Cref{SubfactorOpen,ArrowFromRefiner}, $\phi\subseteq(\sqsupset\circ\leq^l_n)_\mathsf{S}$ and ${\sqsupset\circ\leq^l_n}\subseteq{\leftarrow^\phi_l}$.  Now set
    \[{\twoheadleftarrow'}\ =\ {{\lin}\circ{\nil}\circ{\sqsupset}\circ{\leq^l_n}\circ{\lin}\circ{\nil}}\qquad\text{and}\qquad{\twoheadleftarrow}\ =\ {{\leq^k_l}\circ{\nil}\circ{\sqsupset}\circ{\leq^l_n}\circ{\lin}\circ{\geq^k_l}}.\]
    As ${{\sqsupset}\circ{\leq^l_n}}\subseteq{\twoheadleftarrow'}$, it follows that $\twoheadleftarrow'$ is also a $\mathbf{K}$-like relation satisfying $\phi\subseteq{\twoheadleftarrow'_\mathsf{S}}$.  The same then applies to the even larger relation $\twoheadleftarrow$, which is also widely $\mathbf{K}$-subfactorisable, by \Cref{WideSubfactorisability}.  As ${\sqsupset\circ\leq^l_n}\subseteq{\leftarrow^\phi_l}\subseteq{{\leq^k_l}\circ{\leftarrow^\phi_k}\circ{\geq^k_l}}$ and ${\leftarrow^\phi_k}\vartriangleleft_k{\leftarrow}$,
    \[{\twoheadleftarrow}\ \ \subseteq\ \ {{\leq^k_l}\circ{\geq^k_l}\circ{\leq^k_l}\circ{\leftarrow^\phi_k}\circ{\geq^k_l}\circ{\leq^k_l}\circ{\geq^k_l}}\ \ \subseteq\ \ {{\leq^k_l}\circ{\wedge_k}\circ{\leftarrow^\phi_k}\circ{\wedge_k}\circ{\geq^k_l}}\ \ \subseteq\ \ {{\leq^m_l}\circ{\leftarrow}\circ{\geq^m_l}}.\]
    This shows that ${\twoheadleftarrow}\leq{\leftarrow}$ and, likewise, ${\twoheadleftarrow}\leq{\leftarrow'}$ as well.
\end{proof}

\begin{corollary}\label{SubfactorisableBasis}
    The open sets $\{{\leftarrow_\mathbf{K}}:n\in\omega\text{ and }{\leftarrow}\in\mathbb{P}_n^\mathbf{K}\}$ form a basis for $\mathbf{C}_\mathbf{K}$.
\end{corollary}

\begin{proof}
    This is immediate from \Cref{KLikeKSubfactorisable} once we note ${\twoheadleftarrow}\leq{\leftarrow}$ implies ${\twoheadleftarrow_\mathbf{K}}\subseteq{\leftarrow_\mathbf{K}}$.
\end{proof}

Another immediate corollary of \Cref{KLikeKSubfactorisable} is that, for any ${\leftarrow}\subseteq\mathbb{P}_m\times\mathbb{P}_m$,
\begin{equation}\label{NonemptyImpliesSubfactorisable}
    {\leftarrow_\mathbf{K}}\neq\emptyset\qquad\Rightarrow\qquad{\leftarrow}\text{ is $\mathbf{K}$-like}.
\end{equation}
(because if ${\twoheadleftarrow}\leq{\leftarrow}$ and $\twoheadleftarrow$ is $\mathbf{K}$-like then so is $\leftarrow$).  In particular, for any $\phi\in\mathbf{C}_\mathbf{K}$, the corresponding thin arrow-sequence $(\leftarrow^\phi_n)$ consists entirely of $\mathbf{K}$-like relations.  Conversely, all  $\mathbf{K}$-subfactorisable thin arrow-sequences produce $\mathbf{K}$-compatible functions.

\begin{proposition}\label{KLikeArrowSequence}
    If we have a thin arrow-sequence of $\mathbf{K}$-subfactorisable ${\leftarrow_k}\subseteq\mathbb{P}_{n_k}\times\mathbb{P}_{n_k}$, for some strictly increasing $(n_k)\subseteq\omega$, then $\phi=\bigcap_{k\in\omega}\overline{\leftarrow_{k\mathsf{S}}}$ is $\mathbf{K}$-compatible.
\end{proposition}

\begin{proof}
    First assume each $\leftarrow_k$ is $\mathbf{K}$-like.  For any $m\in\omega$, we have neighbourly ${\nil}\in\mathbf{K}^m_{m'}$.  We also have $k>m'$ such that $\mathbb{P}_n\mathrel{{}_k\!\!\rightarrow\!\rangle}\mathbb{P}_{m'}$, for all sufficiently large $n$.  As $\leftarrow_k$ is $\mathbf{K}$-like, we have ${\sqsupset}\in\mathbf{K}^{n_k}_n$ such that ${=_{\mathbb{P}_n}}\subseteq{\sqsubset\circ\leftarrow_k\circ\geq^{n_k}_n}\subseteq{{\wedge}\circ{\sqsubset}\circ{\leftarrow_k}\circ{\barwedge}}$, where we can also ensure that $n$ here is large enough that $\mathbb{P}_n\mathrel{{}_k\!\!\rightarrow\!\rangle}\mathbb{P}_{m'}$.  By \Cref{DoubleTriangle2},
    \[\mathbf{K}^m_n\ \ni\ {{\nil}\circ{\geq^{m'}_{n_k}}\circ{\sqsupset}}\ \ \subseteq\ \ {{\vartriangleright_{\mathbb{P}_{n_k}}}\circ{\sqsupset}}\ \ \subseteq\ \ {\langle\!\leftarrow_k}\ \ \subseteq\ \ {\sqsupset_\phi}.\]
    As $m$ was arbitrary, this shows that ${\phi}\in\mathbf{C}_\mathbf{K}$.

    To prove the general case, first note that $({=_{\mathbb{P}_{n_k}}}\circ{\barwedge}\circ{\leftarrow_k}\circ{\barwedge}\circ{=_{\mathbb{P}_{n_k}}})$ is again a thin arrow-sequence resulting in the same function $\phi$, by \Cref{ArrowComposition}.  For each $k\in\omega$, we have $m>k$ and subabsorbing ${\nil}\in\mathbf{K}^{n_k}_{n_m}$.  Note ${{\lin}\circ{\nil}\circ{\leftarrow_m}\circ{\lin}\circ{\nil}}$ is $\mathbf{K}$-subfactorisable, as it is bigger than $\leftarrow_k$, and hence ${\nil}\circ{\leftarrow_m}\circ{\lin}$ is $\mathbf{K}$-like, by \Cref{KSubfactorKLike}.  Moreover,
    \[{\nil}\circ{\leftarrow_m}\circ{\lin}\ \ \subseteq\ \ {\nil}\circ{\leq}\circ{\leftarrow_k}\circ{\geq}\circ{\lin}\ \ \subseteq\ \ {=_{\mathbb{P}_{n_k}}}\circ{\barwedge}\circ{\leftarrow_k}\circ{\barwedge}\circ{=_{\mathbb{P}_{n_k}}}\]
    so ${=_{\mathbb{P}_{n_k}}}\circ{\barwedge}\circ{\leftarrow_k}\circ{\barwedge}\circ{=_{\mathbb{P}_{n_k}}}$ is also $\mathbf{K}$-like.  Thus ${\phi}\in\mathbf{C}_\mathbf{K}$ as before.
\end{proof}

By definition, $\mathbf{C}_\mathbf{K}$ a $G_\delta$ subsemigroup of $\mathbf{C}^{\mathsf{S}\mathbb{P}}_{\mathsf{S}\mathbb{P}}$.  We can now show it is actually closed.

\begin{corollary}\label{CKClosed}
    $\mathbf{C}_\mathbf{K}$ is a closed subsemigroup of $\mathbf{C}^{\mathsf{S}\mathbb{P}}_{\mathsf{S}\mathbb{P}}$.
\end{corollary}

\begin{proof}
    Take any $(\phi_n)\subseteq\mathbf{C}_\mathbf{K}$ with $\phi_n\rightarrow\phi$ in $\mathbf{C}^{\mathsf{S}\mathbb{P}}_{\mathsf{S}\mathbb{P}}$.  By \Cref{Function->ArrowSequence}, $(\leftarrow^\phi_n)$ is a thin arrow-sequence with
    $\phi=\bigcap_{n\in\omega}{\leftarrow^\phi_{n\mathsf{S}}}$.
    For all $n\in\omega$, $\leftarrow^\phi_{n\mathbf{C}}$ is a neighbourhood of $\phi$ and hence contains some $\phi_m\in\mathbf{C}_\mathbf{K}$ so $\leftarrow^\phi_n$ is $\mathbf{K}$-like, by \eqref{NonemptyImpliesSubfactorisable}.  Thus $\phi\in\mathbf{C}_\mathbf{K}$, by \Cref{KLikeArrowSequence}.
\end{proof}

We can also show that $\mathbf{C}_\mathbf{K}$ is closed under taking inverses, when they exist.

\begin{proposition}\label{InverseCK}
    If $\phi\in\mathbf{C}_\mathbf{K}$ and $\phi^{-1}$ is a function then $\phi^{-1}\in\mathbf{C}_\mathbf{K}$ as well.
\end{proposition}

\begin{proof}
    If $\phi\in\mathbf{C}_\mathbf{K}$ then $(\leftarrow^\phi_n)$ is a $\mathbf{K}$-subfactorisable thin arrow-sequence.  Thus $(\rightarrow^\phi_n)$ is also a $\mathbf{K}$-subfactorisable arrow-sequence, which is also thin if $\phi^{-1}$ is a function, in which case $\phi^{-1}=\bigcap_{n\in\omega}\overline{\rightarrow_{n\mathsf{S}}^\phi}\in\mathbf{C}_\mathbf{K}$, by \Cref{KLikeArrowSequence}.
\end{proof}

Now we restrict our attention to the autohomeomorphisms of $\mathsf{S}\mathbb{P}$ compatible with $\mathbf{K}$ whose inverses are functions and hence also compatible with $\mathbf{K}$, by \Cref{InverseCK}.  Accordingly, let
\[\mathbf{C}_\mathbf{K}^\times=\mathbf{C}_\mathbf{K}\cap\mathbf{C}_{\mathbf{K}}^{-1}=\{\phi\in\mathbf{C}_\mathbf{K}:\phi^{-1}\in\mathbf{C}_\mathbf{K}\}=\{\phi\in\mathbf{C}_\mathbf{K}:\phi^{-1}\text{ is a function}\}.\]
Note that $\mathbf{C}_\mathbf{K}^\times$ is a closed subgroup of the autohomeomorphism group of $\mathsf{S}\mathbb{P}$, either by \Cref{CKClosed} or the general fact that a $G_\delta$ subgroup of a Polish group is always closed.

First we wish to prove a kind of converse to \eqref{NonemptyImpliesSubfactorisable}, namely that widely $\mathbf{K}$-subfactorisable relations yield open sets containing functions in $\mathbf{C}_\mathbf{K}^\times$.  For this we will need the following.
    
\begin{lemma}\label{SKlemma}
    For any ${\nil}\in\mathbf{K}^l_{l'}$, ${\nil'}\in\mathbf{K}^{l'}_m$, ${\dashv}\in\mathbf{K}^m_{m'}$ and ${\sqsupset}\in\mathbf{K}^{m'}_n$,
    \[\text{$\nil$ and $\nil'$ are $\mathbb{P}$-neighbourly and $\dashv$ is subabsorbing}\quad\Rightarrow\quad({\nil}\circ{\nil'}\circ{\dashv}\circ{\sqsupset})^\mathbf{C}\cap\mathbf{C}_\mathbf{K}^\times\ \neq\ \emptyset.\]
\end{lemma}

\begin{proof}
    Take subabsorbing ${\Dashv_0}\in\mathbf{K}^n_{n'}$.  As $\dashv$ is subabsorbing, we have ${\sqni_0}\in\mathbf{K}^{n'}_{m_1}$ with
    \[{{\dashv}\circ{\sqsupset}\circ{\Dashv_0}\circ{\sqni_0}}\ \subseteq\ {\geq^m_{m_1}}.\]
    Again taking subabsorbing ${\dashv_1}\in\mathbf{K}^{m_1}_{m_1'}$, we have ${\sqsupset_1}\in\mathbf{K}^{m_1'}_{n_1}$ such that
    \[{{\Dashv_0}\circ{\sqni_0}\circ{\dashv_1}\circ{\sqsupset_1}}\ \subseteq\ {\geq^{n_0}_{n_1}}.\]
    Continuing in this way we get ${\sqsupset_k'}={{\dashv_k}\circ{\sqsupset_k}}\in\mathbf{K}^{m_k}_{n_k}$ and ${\sqni_k'}={{\Dashv_k}\circ{\sqni_k}}\in\mathbf{K}^{n_k}_{m_{k+1}}$ such that ${\sqsupset_k'\circ\sqni_k'}\subseteq{\geq^{m_k}_{m_{k+1}}}$ and ${\sqni_k'\circ\sqsupset_{k+1}'}\subseteq{\geq^{n_k}_{n_{k+1}}}$, for all $k\in\omega$ (where $m_0=m$, $m_0'=m'$, $n_0=n$, $n_0'=n'$ and ${\sqsupset_0}={\sqsupset}$).  By \Cref{BackAndForth}, we then have a homeomorphism
    \[\phi=\bigcap_{n\in\omega}\overline{\sqsupset_{n\mathsf{S}}'}=\bigcap_{n\in\omega}\overline{\sqin_{n\mathsf{S}}'}.\]
    
    For each $j\in\omega$, we can take $\mathbb{P}$-neighbourly ${\sqnii}\in\mathbf{K}^j_{j'}$ and ${\sqnii'}\in\mathbf{K}^{j'}_{m_k}$, for some $k>j$, so 
    \[{{\sqnii}\circ{\sqnii'}\circ{\sqsupset_k'}}\ \subseteq\ {\sqsupset_\phi},\]
    by \Cref{DoubleTriangle}.  This shows that $\phi\in\mathbf{C}_\mathbf{K}$.  Likewise, for each $j\in\omega$, we can take $\mathbb{P}$-neighbourly ${\sqnii}\in\mathbf{K}^j_{j'}$ and ${\sqnii'}\in\mathbf{K}^{j'}_{n_k}$, for some $k>j$, so \Cref{DoubleTriangle} again yields
    \[{{\sqnii}\circ{\sqnii'}\circ{\sqni_k'}}\ \subseteq\ {\sqsupset_{\phi^{-1}}},\]
    showing that $\phi^{-1}\in\mathbf{C}_\mathbf{K}$ too and hence $\phi\in\mathbf{C}_\mathbf{K}^\times$.  Now \Cref{DoubleTriangle} again tells us that 
    \[{{\nil}\circ{\nil'}\circ{\dashv}\circ{\sqsupset}}\ =\ {{\nil}\circ{\nil'}\circ{\sqsupset_0'}}\ \subseteq\ {\sqsupset_\phi},\]
    i.e.~$\phi$ also lies in $({\nil}\circ{\nil'}\circ{\dashv}\circ{\sqsupset})^\mathbf{C}$, as required.
\end{proof}

Like before, let us denote the basic open set in $\mathbf{C}_\mathbf{K}^\times$ defined by any ${\leftarrow}\subseteq\mathbb{P}_n\times\mathbb{P}_n$ by
\[{\leftarrow^\times_\mathbf{K}}={{\leftarrow_\mathbf{C}}\cap\mathbf{C}_\mathbf{K}^\times}=\{\phi\in\mathbf{C}_\mathbf{K}^\times:\phi\subseteq{\leftarrow_\mathsf{S}}\}.\]

\begin{corollary}\label{SubfactorisableImpliesNonempty}
    If $m\in\omega$ and ${\leftarrow}\in\mathbb{P}_m^\mathbf{K}$ then ${\leftarrow_\mathbf{K}^\times}\neq\emptyset$.
\end{corollary}

\begin{proof}
    Take ${\nil}\in\mathbf{K}^m_{m'}$ that can be written as ${\nil'}\circ{\nil''}\circ{\dashv}$, for some $\mathbb{P}$-neighbourly $\nil'$ and $\nil''$ and subabsorbing $\dashv$.  As ${\leftarrow}$ is widely $\mathbf{K}$-subfactorisable, so is ${\lin}\circ{\leftarrow}\circ{\nil}$ and hence it is $\mathbf{K}$-like, so we have ${\sqsupset}\in\mathbf{K}^{m'}_n$ with ${=_{\mathbb{P}_n}}\subseteq{\sqsubset}\circ{\lin}\circ{\leftarrow}\circ{\nil}\circ{\geq^{m'}_n}$.  It then follows from \Cref{SubfactorOpen} and \Cref{SKlemma} that $\emptyset\neq({\nil}\circ{\sqsupset})^\mathbf{C}\cap\mathbf{C}^\times_\mathbf{K}\subseteq{\leftarrow}_\mathbf{C}\cap\mathbf{C}^\times_\mathbf{K}={\leftarrow}_\mathbf{K}^\times$.
\end{proof}

The following corollary of the above will be needed in the next section.

\begin{corollary}\label{DigraphShrinking}
    If ${\leftarrow}\in\mathbb{P}_m^\mathbf{K}$ then we have ${\twoheadleftarrow}\in\mathbb{P}_n^\mathbf{K}$ with ${\twoheadleftarrow}\vartriangleleft_n{\leftarrow}$ and $\mathbb{P}_n\mathrel{\twoheadrightarrow\!\rangle}\mathbb{P}_m$.
\end{corollary}

\begin{proof}
    Given ${\leftarrow}\in\mathbb{P}_m^\mathbf{K}$, \Cref{SubfactorisableImpliesNonempty} yields $\phi\in{\leftarrow_\mathbf{K}^\times}$.  By \Cref{ArrowPhi,ArrowPhiLebesgue}, we have $k>m$ with ${\leftarrow^\phi_k}\vartriangleleft_k{\leftarrow}$ and $\mathbb{P}_m^{\langle\!\leftarrow^\phi_k}\in\mathsf{C}\mathbb{P}$.  By \Cref{KLikeKSubfactorisable}, we then have $n>k$ with $\mathbb{P}_n\subseteq\mathbb{P}_m^{\langle\!\leftarrow^\phi_k}$ and ${\twoheadleftarrow}\in\mathbb{P}_n^\mathbf{K}$ with ${\twoheadleftarrow}\leq{\leftarrow^\phi_k}\vartriangleleft_n{\leftarrow}$ so ${\langle\!\leftarrow^\phi_k}\subseteq{\langle\!\twoheadleftarrow}$ and hence $\mathbb{P}_n\subseteq\mathbb{P}_m^{\langle\!\twoheadleftarrow}$.
\end{proof}

\subsection{The Digraph Category}

A \emph{$\mathbf{K}$-digraph} is a graph $G\in\mathbf{K}$ together with a widely $\mathbf{K}$-subfactorisable relation ${\leftarrow_G}\in G^\mathbf{K}$.  The $\mathbf{K}$-digraphs form a category $\underline{\mathbf{K}}$ with hom-sets
\[\underline{\mathbf{K}}_H^G=\underline{\mathbf{K}}_{\leftarrow_H}^{\leftarrow_G}=\{{\sqsupset}\in\mathbf{K}_H^G:{\leftarrow_H}\subseteq{\sqsubset\circ\leftarrow_G\circ\sqsupset}\}.\]
By \Cref{EqualitySubfactor}, $G^\mathbf{K}\neq\emptyset$, for all $G\in\mathbf{K}$, so $\underline{\mathbf{K}}$ is an expansion of $\mathbf{K}$, i.e.~the forgetful functor from $\underline{\mathbf{K}}$ to $\mathbf{K}$ is surjective on objects.

We call ${\leftarrow_k}\in\mathbb{P}_{n_k}^\mathbf{K}$ a \emph{$\underline{\mathbf{K}}$-sequence} on $\mathbb{P}$ if, for all $k\in\omega$, $n_k<n_{k+1}$ and ${\leftarrow_{k+1}}\leq{\leftarrow_k}$.  By \Cref{SurjectiveSubfactor}, every ${\leftarrow}\in\mathbb{P}^\mathbf{K}_n$ is co-surjective and hence an arrow, so every $\underline{\mathbf{K}}$-sequence is an arrow-sequence to which we can apply the theory developed thus far.  In particular, every bi-thin $\underline{\mathbf{K}}$-sequence ${\leftarrow_k}\in\mathbb{P}^\mathbf{K}_{n_k}$ defines a function $\phi=\bigcap_{k\in\omega}\overline{\leftarrow_{k\mathsf{S}}}\in\mathbf{C}_\mathbf{K}^\times$, by \Cref{KLikeArrowSequence}.  We will be interested in conjugacy classes $\{\theta^{-1}\circ\phi\circ\theta:\theta\in\mathbf{C}_\mathbf{K}^\times\}$ of such $\phi$ and how they relate to categorical properties of the defining $\underline{\mathbf{K}}$-sequence $(\leftarrow_k)$.

Like before with posets, we say that a $\underline{\mathbf{K}}$-sequence ${\leftarrow_k}\in\mathbb{P}^\mathbf{K}_{n_k}$ on $\mathbb{P}$ is \emph{$\underline{\mathbf{K}}$-coinitial} if the digraphs $(\mathbb{P}_{n_k},\leftarrow_k)$ defined by $(\leftarrow_k)$ are coinitial in $\underline{\mathbf{K}}$.

\begin{proposition}\label{CoinitialFromDirected}
    If $\underline{\mathbf{K}}$ is directed then there is a bi-thin $\underline{\mathbf{K}}$-coinitial $\underline{\mathbf{K}}$-sequence on $\mathbb{P}$.
\end{proposition}

\begin{proof}
    As $\underline{\mathbf{K}}$ is directed, so is $\mathbf{K}$.  As $\mathbb{P}$ is $\mathbf{K}$-subabsorbing, $\mathbb{P}$ must then be $\mathbf{K}$-coinitial.  Thus, for every $(G,\leftarrow_G)\in\underline{\mathbf{K}}$, we have some $n\in\omega$ and ${\sqsupset}\in\mathbf{K}^G_{\mathbb{P}_n}$ and hence ${\sqsupset}\in\underline{\mathbf{K}}^{\leftarrow_G}_{\sqsubset\circ\leftarrow_G\circ\sqsupset}$.  To construct a $\underline{\mathbf{K}}$-coinitial $\underline{\mathbf{K}}$-sequence it thus suffices to consider relations on levels of $\mathbb{P}$.

    Accordingly, take an enumeration ${\leftarrow_k}\in\mathbb{P}_{m_k}^\mathbf{K}$ of all widely $\mathbf{K}$-subfactorisable relations on all levels of $\mathbb{P}$.  We construct a bi-thin $\underline{\mathbf{K}}$-coinitial $\underline{\mathbf{K}}$-sequence $(\twoheadleftarrow_k)$ recursively as follows.  First set $n_0=0$ and take any ${\twoheadleftarrow_0}\in\mathbb{P}_{n_0}^\mathbf{K}$.  Once ${\twoheadleftarrow_k}\in\mathbb{P}_{n_k}$ has been defined, take any subabsorbing ${\nil}\in\mathbf{K}^{n_k}_n$.  As $\underline{\mathbf{K}}$ is directed, we then have some $(G,\leftarrow_G)\in\underline{\mathbf{K}}$ such that there exists ${\sqsupset}\in\mathbf{K}^{\leftarrow_k}_{\leftarrow_G}$ and ${\sqni}\in\mathbf{K}^{{\lin}\circ{\twoheadleftarrow_k}\circ{\nil}}_{\leftarrow_G}$.  As $\nil$ is subabsorbing, we then have ${\sqni'}\in\mathbf{K}^G_{\mathbb{P}_{n'}}$ such that ${\nil}\circ{\sqni}\circ{\sqni'}\subseteq{\geq^n_{n'}}$.  Applying \Cref{DigraphShrinking} to ${\leftarrow}={{\sqin'}\circ{\sqin}\circ{\lin}\circ{\twoheadleftarrow_k}\circ{\nil}\circ{\sqni}\circ{\sqni'}}$, we obtain $n_{k+1}>n'$ and ${\twoheadleftarrow_{k+1}}\in\mathbb{P}_{n_{k+1}}^\mathbf{K}$ such that $\mathbb{P}_k^{\langle\!\twoheadleftarrow_{k+1}},\mathbb{P}_k^{\langle\!\twoheadrightarrow_{k+1}}\in\mathsf{C}\mathbb{P}$, ${\twoheadleftarrow_{k+1}}\leq{\leftarrow}\leq{\twoheadleftarrow_k}$ and ${{\sqsupset}\circ{\sqni'}}\in\mathbf{K}^{\leftarrow_k}_{\twoheadleftarrow_{k+1}}$.  Thus the resulting $\underline{\mathbf{K}}$-sequence is indeed bi-thin $\underline{\mathbf{K}}$-coinitial.
\end{proof}

\begin{theorem}\label{DirectedVsDenseCC}
    If ${\leftarrow_k}\in\mathbb{P}^\mathbf{K}_{n_k}$ is a bi-thin $\underline{\mathbf{K}}$-sequence and $\phi=\bigcap_{k\in\omega}\overline{\leftarrow_{k\mathsf{S}}}$ then
    \[(\leftarrow_k)\text{ is $\underline{\mathbf{K}}$-coinitial}\quad\Leftrightarrow\quad\mathbb{P}\text{ is $\mathbf{K}$-coinitial and the conjugacy class of $\phi$ is dense in }\mathbf{C}_\mathbf{K}^\times.\]
\end{theorem}

\begin{proof}
    If $(\leftarrow_k)$ is not $\mathbf{K}$-coinitial then we have $G\in\mathbf{K}$ and ${\twoheadleftarrow}\in G^\mathbf{K}$ such that $\underline{\mathbf{K}}^{\twoheadleftarrow}_{\leftarrow_k}=\emptyset$, for all $k\in\omega$.  Either $\mathbb{P}$ is not $\mathbf{K}$-coinitial, and we are done, or we have $m\in\omega$ with $\mathbf{K}_{\mathbb{P}_m}^G\neq\emptyset$.  If necessary, we may thus replace $G$ with $\mathbb{P}_m$ and $\twoheadleftarrow$ with ${\sqsubset}\circ{\twoheadleftarrow}\circ{\sqsupset}$, for some ${\sqsupset}\in\mathbf{K}_{\mathbb{P}_m}^G$, to ensure that $G=\mathbb{P}_m$, for some $m\in\omega$.  By \Cref{SubfactorisableImpliesNonempty}, we can then take $\psi\in{\twoheadleftarrow^\times_\mathbf{K}}$.  By \Cref{ArrowPhiLebesgue}, we have $m'>m$ with ${\leftarrow^\psi_{m'}}\leq{\twoheadleftarrow}$.  If we had $\theta\in\mathbf{C}^\times_\mathbf{K}$ with $\phi=\theta^{-1}\circ\psi\circ\theta$ then, by the definition of $\mathbf{C}_\mathbf{K}$, we would have $n>m'$ and ${\sqsupset}\in\mathbf{K}^{m'}_n$ with ${\sqsupset}\subseteq{\sqsupset_\theta}$.  By \Cref{ConjugacyLemma}, we would then have $k\in\omega$ with ${\leftarrow_k}\subseteq{{\sqsubset}\circ{\leftarrow^\psi_{m'}}\circ{\sqsupset}}\subseteq{{\sqsubset}\circ{\twoheadleftarrow}\circ{\sqsupset}}$, i.e.~${\sqsupset}\in\underline{\mathbf{K}}^\twoheadleftarrow_{\leftarrow_k}$, a contradiction.  This shows that $\twoheadleftarrow^\times_\mathbf{K}$ is a non-empty open subset of $\mathbf{C}^\times_\mathbf{K}$ disjoint from the conjugacy class of $\phi$.

    Conversely, if $(\leftarrow_k)$ is $\underline{\mathbf{K}}$-coinitial then certainly $\mathbb{P}$ is $\mathbf{K}$-coinitial.  To see that the conjugacy class of $\phi$ is also dense in $\mathbf{C}_\mathbf{K}^\times$, it suffices to consider open sets of the form $\twoheadleftarrow_\mathbf{K}^\times$, for ${\twoheadleftarrow}\in\mathbb{P}^\mathbf{K}_m$, by \Cref{SubfactorisableBasis}.  For any such $\twoheadleftarrow$, we can take ${\nil}\in\mathbf{K}^m_{m'}$ that can be written as ${\nil'}\circ{\nil''}\circ{\dashv}$, for $\mathbb{P}$-neighbourly $\nil'$ and $\nil''$ and subabsorbing $\dashv$.  Then ${{\lin}\circ{\twoheadleftarrow}\circ{\nil}}\in\mathbb{P}_{m'}^\mathbf{K}$ so, as $(\leftarrow_k)$ is $\underline{\mathbf{K}}$-coinitial, we have ${\sqsupset}\in\mathbf{K}^{m'}_{n_k}$ with ${\leftarrow_k}\subseteq{{\sqsubset}\circ{\lin}\circ{\twoheadleftarrow}\circ{\nil}\circ{\sqsupset}}$.  By \Cref{SKlemma}, we have $\theta\in({\nil}\circ{\sqsupset})^\mathbf{C}\cap\mathbf{C}^\times_\mathbf{K}$ so ${\leftarrow_k}\subseteq{{\sqsubset_\theta}\circ{\twoheadleftarrow}\circ{\sqsupset_\theta}}$ and hence $\theta\circ\phi\circ\theta^{-1}\subseteq\theta\circ\overline{\leftarrow_{k\mathsf{S}}}\circ\theta^{-1}\subseteq{\twoheadleftarrow_\mathsf{S}}$, by \Cref{ConjugateArrow}.  This shows that $\theta\circ\phi\circ\theta^{-1}\in{\twoheadleftarrow_\mathbf{K}^\times}$, as required.
\end{proof}

Given a $\underline{\mathbf{K}}$-sequence ${\leftarrow_k}\in\mathbb{P}_{n_k}^\mathbf{K}$, we say that ${\nil}\in\underline{\mathbf{K}}^{\leftarrow_i}_{\leftarrow_j}$ is \emph{$\underline{\mathbf{K}}$-subabsorbing} if ${\nil}\subseteq{\geq^{n_i}_{n_j}}$ and, for all ${\sqsupset}\subseteq\underline{\mathbf{K}}^{\leftarrow_j}_{\leftarrow_G}$, we have ${\sqsupset'}\in\underline{\mathbf{K}}^{\leftarrow_G}_{\leftarrow_{j'}}$ with ${{\nil}\circ{\sqsupset}\circ{\sqsupset'}}\subseteq{\geq^{n_i}_{n_{j'}}}$.  We call the $\underline{\mathbf{K}}$-sequence $(\leftarrow_k)$ \emph{$\underline{\mathbf{K}}$-subabsorbing} if, for all $i\in\omega$, we have $j>i$ and $\underline{\mathbf{K}}$-subabsorbing ${\nil}\in\underline{\mathbf{K}}^{\leftarrow_i}_{\leftarrow_j}$.

\begin{proposition}\label{BiThinFraisse}
Every $\underline{\mathbf{K}}$-subabsorbing $\underline{\mathbf{K}}$-sequence ${\leftarrow_k}\in\mathbb{P}_{n_k}^\mathbf{K}$ is regular and bi-thin.
\end{proposition}

\begin{proof}
    For any $j\in\omega$, we have some $\underline{\mathbf{K}}$-subabsorbing ${\dashv}\in\underline{\mathbf{K}}^{\leftarrow_j}_{\leftarrow_k}$.  We then have $\mathbb{P}$-neighbourly ${\nil}\in\mathbf{K}^{n_k}_l$.  Setting ${\twoheadleftarrow}={{\lin}\circ{\leftarrow_k}\circ{\nil}}\in\mathbb{P}_l^\mathbf{K}$, it follows that ${\geq^{n_k}_l}\in\mathbf{K}^{\leftarrow_k}_{\twoheadleftarrow}$.  As $\nil$ is $\underline{\mathbf{K}}$-subabsorbing, we then have ${\sqsupset}\in\underline{\mathbf{K}}^\twoheadleftarrow_{\leftarrow_m}$ with ${{\dashv}\circ{\geq^{n_k}_l}\circ{\sqsupset}}\subseteq{\geq^{n_j}_{n_m}}$.  So
    \[{\leftarrow_m}\ \subseteq\ {{\sqsubset}\circ{\twoheadleftarrow}\circ{\sqsupset}}\ =\ {{\sqsubset}\circ{\lin}\circ{\leftarrow_k}\circ{\nil}\circ{\sqsupset}}\ \subseteq\ {{\sqsubset}\circ{\lin}\circ{\vdash}\circ{\leftarrow_j}\circ{\dashv}\circ{\nil}\circ{\sqsupset}}\]
    and, as ${\nil}$ is $\mathbb{P}$-neighbourly and $\sqsupset$ is co-surjective and $\wedge$-preserving,
    \[{{\dashv}\circ{\nil}\circ{\sqsupset}\circ{\wedge_{n_m}}}\ \subseteq\ {{\dashv}\circ{\nil}\circ{\wedge_l}\circ{\sqsupset}}\ \subseteq\ {{\dashv}\circ{\geq^{n_k}_l}\circ{\sqsupset}}\ \subseteq\ {\geq^{n_j}_{n_m}}.\]
    This means ${{\dashv}\circ{\nil}\circ{\sqsupset}}\subseteq{\vartriangleright_{\mathbb{P}_{n_m}}}$ so ${\leftarrow_m}\vartriangleleft_{n_m}{\leftarrow_j}$, showing that $(\leftarrow_k)$ is indeed regular.

    Now again, for any $j\in\omega$, we have some $\underline{\mathbf{K}}$-subabsorbing ${\dashv}\in\underline{\mathbf{K}}^{\leftarrow_j}_{\leftarrow_k}$.  By \Cref{DigraphShrinking}, we have $l>n_k$ and ${\twoheadleftarrow}\in\mathbb{P}_l^\mathbf{K}$ with ${\twoheadleftarrow}\leq{\leftarrow_k}$, $\mathbb{P}_l\mathrel{\twoheadrightarrow\!\rangle}\mathbb{P}_{n_k}$ and $\mathbb{P}_l\mathrel{\twoheadleftarrow\!\rangle}\mathbb{P}_{n_k}$.  As $\dashv$ is $\underline{\mathbf{K}}$-subabsorbing, we then have ${\sqsupset}\in\underline{\mathbf{K}}^{\twoheadleftarrow}_{\leftarrow_m}$ such that ${{\dashv}\circ{\geq^{n_k}_l}\circ{\sqsupset}}\subseteq{\geq^{n_j}_{n_m}}$.

    Given $p\in\mathbb{P}_{n_m}$, take $q,r,s\in\mathbb{P}$ with $s\dashv r\mathrel{\langle\!\twoheadleftarrow}q\sqsupset p$.  Then
    \[p^{\wedge\rightarrow_m}\subseteq p^{\wedge\sqsubset\twoheadrightarrow\sqsupset}\subseteq q^{\wedge\twoheadrightarrow\sqsupset}\subseteq r^{\geq\sqsupset}\subseteq s^{\dashv\geq\sqsupset}\subseteq s^\geq,\]
    showing that $p\mathrel{{}_m\!\!\rightarrow\!]}s$.  This in turn shows that $\mathbb{P}_{n_m}\mathrel{{}_m\!\!\rightarrow\!]}\mathbb{P}_{n_j}$ and, likewise, $\mathbb{P}_{n_m}\mathrel{{}_m\!\!\leftarrow\!]}\mathbb{P}_{n_j}$.  Thus $(\leftarrow_k)$ is bi-thin, by \Cref{RegularSequence}.
\end{proof}

\begin{theorem}
    If $\underline{\mathbf{K}}$ has countably many morphisms and a wide collection of subamalgamable morphisms then there is a $\underline{\mathbf{K}}$-subabsorbing $\underline{\mathbf{K}}$-sequence $(\leftarrow_k)$ on $\mathbb{P}$.
\end{theorem}

\begin{proof}
    This is a variant of the usual construction of a Fra\"iss\'e sequence from amalgamation.
    
    First we claim that, for any ${\twoheadleftarrow}\in\mathbb{P}_m^\mathbf{K}$, we have $n>m$, ${\leftarrow}\in\mathbb{P}_n^\mathbf{K}$ and ${\nil}\in\underline{\mathbf{K}}^\twoheadleftarrow_\leftarrow$ that is subamalgamable in $\underline{\mathbf{K}}$, as witnessed by ${\geq^m_n}\supseteq{\nil}$.  Indeed, as $\mathbb{P}$ is $\mathbf{K}$-subabsorbing, we have $k>m$ and subabsorbing ${\dashv}\in\mathbf{K}^m_k$.  We then have $\underline{\mathbf{K}}$-subamalgamable ${\sqsupset}\in\underline{\mathbf{K}}^{\vdash\circ\twoheadleftarrow\circ\dashv}_{\leftarrow_G}$, as witnessed by some ${\sqsupseteq}\in\underline{\mathbf{K}}^{\vdash\circ\twoheadleftarrow\circ\dashv}_{\leftarrow_G}$ with ${\sqsupseteq}\supseteq{\sqsupset}$.  We then further have ${\sqni}\in\mathbf{K}^G_{\mathbb{P}_n}$ such that ${{\dashv}\circ{\sqsupseteq}\circ{\sqni}}\subseteq{\geq^m_n}$.  Setting ${\leftarrow}={{\sqin}\circ{\leftarrow_G}\circ{\sqni}}$ then ${\nil}={{\dashv}\circ{\sqsupset}\circ{\sqni}}\in\underline{\mathbf{K}}^{\twoheadleftarrow}_{\leftarrow}$ is indeed subamalgamable in $\underline{\mathbf{K}}$ as witnessed by ${\geq^m_n}\supseteq{{\dashv}\circ{\sqsupseteq}\circ{\sqni}}\supseteq{\nil}$.
    
    Now take any sequence $(m_k)\subseteq\omega$ such that $0<m_k<k$, for all $k\geq2$, and each element of $\omega\setminus\{0\}$ appears infinitely many times.  Next take any $n_0\in\omega$ and let ${\leftarrow_0}=\mathbb{P}_{n_0}\times\mathbb{P}_{n_0}$.  By the claim just proved, we have $n_1>n_0$, ${\leftarrow_1}\in\mathbb{P}_{n_1}^\mathbf{K}$ and ${\nil_1}\in\underline{\mathbf{K}}^{\leftarrow_0}_{\leftarrow_1}$ that is subamalgamable in $\underline{\mathbf{K}}$, as witnessed by ${\geq^{n_0}_{n_1}}$.  Likewise, we can then take $n_2>n_1$, ${\leftarrow_2}\in\mathbb{P}_{n_2}^\mathbf{K}$ and ${\nil_2}\in\underline{\mathbf{K}}^{\leftarrow_1}_{\leftarrow_2}$ that is subamalgamable in $\underline{\mathbf{K}}$, as witnessed by ${\geq^{n_1}_{n_2}}$.
    
    Once $n_k\in\omega$, ${\leftarrow_k}\in\mathbb{P}_{n_k}^\mathbf{K}$ and ${\nil_k}\in\mathbf{K}^{\leftarrow_{k-1}}_{\leftarrow_k}$ have been defined, we define $n_{k+1}\in\omega$, ${\leftarrow_{k+1}}\in\mathbb{P}_{n_{k+1}}^\mathbf{K}$ and ${\nil_{k+1}}\in\mathbf{K}^{\leftarrow_k}_{\leftarrow_{k+1}}$ from any ${\sqsupset_k}\in\underline{\mathbf{K}}^{\leftarrow_{m_k}}_{\leftarrow_G}$ as follows.  First take subabsorbing ${\dashv}\in\mathbf{K}^{n_k}_n$.  As $\nil_{m_k}$ is subamalgamable in $\underline{\mathbf{K}}$, as witnessed by $\geq^{n_{m_k-1}}_{n_{m_k}}$, we have ${\sqsupset}\in\underline{\mathbf{K}}^{\leftarrow_G}_{\leftarrow_H}$ and ${\sqni}\in\underline{\mathbf{K}}^{{\vdash}\circ{\leftarrow_k}\circ{\dashv}}_{\leftarrow_H}$ with ${{\nil_{m_k}}\circ{\sqsupset_k}\circ{\sqsupset}}\subseteq{{\geq^{n_{m_k-1}}_{n_k}}\circ{\dashv}\circ{\sqni}}$.  As $\dashv$ is subabsorbing, we then have $n'>n$ and ${\sqni'}\in\mathbf{K}^H_{\mathbb{P}_{n'}}$ with ${{\dashv}\circ{\sqni}\circ{\sqni'}}\subseteq{\geq^{n_k}_{n'}}$.  Set
    \[{\twoheadleftarrow}\ =\ {\sqin'\circ\leftarrow_H\circ\sqni'}\ \subseteq\ {\sqin'\circ\sqin\circ\vdash\circ\leftarrow_k\circ\dashv\circ\sqni\circ\sqni'}\ \subseteq\ {\leq_{n'}^{n_k}\circ\leftarrow_k\circ\geq_{n'}^{n_k}}\]
    so ${\twoheadleftarrow}\leq{\leftarrow_k}$.  Moreover, ${{\sqsupset}\circ{\sqni'}}\in\mathbf{K}^{\leftarrow_G}_{\twoheadleftarrow}$ witnesses $\underline{\mathbf{K}}$-subabsorption for $\sqsupset_k$ in that
    \[{{\nil_{m_k}}\circ{\sqsupset_k}\circ{\sqsupset}\circ{\sqni'}}\ \subseteq\ {{\geq^{n_{m_k-1}}_{n_k}}\circ{\dashv}\circ{\sqni}\circ{\sqni'}}\ \subseteq\ {\geq^{n_{m_k-1}}_{n'}}.\]
    Finally, by the claim proved above, we have $n_{k+1}>n'$ and ${\leftarrow_{k+1}}\in\mathbb{P}_{n_{k+1}}^\mathbf{K}$ and ${\nil}\in\underline{\mathbf{K}}^{\twoheadleftarrow}_{\leftarrow_{k+1}}$ that is subamalgamable in $\underline{\mathbf{K}}$ and hence so is ${\nil_{k+1}}={\geq^{n_k}_{n'}}\circ{\nil}\in\underline{\mathbf{K}}^{\leftarrow_k}_{\leftarrow_{k+1}}$.
    
    This completes the recursive construction of the $\underline{\mathbf{K}}$-sequence $(\leftarrow_k)$.  By our assumptions on $(m_k)$ and the fact that $\underline{\mathbf{K}}^\bullet_\bullet$ is countable, we can perform the construction with choices of ${\sqsupset_k}\in\underline{\mathbf{K}}^{\leftarrow_{m_k}}_{\leftarrow_G}$ that exhaust all morphisms in $\underline{\mathbf{K}}$ having codomain $\leftarrow_m$, for any $m\in\omega$.  This will ensure that $(\leftarrow_k)$ is a $\underline{\mathbf{K}}$-subabsorbing $\underline{\mathbf{K}}$-sequence on $\mathbb{P}$.
\end{proof}

\begin{theorem}
    Assume ${\leftarrow_k}\in\mathbb{P}_{n_k}^\mathbf{K}$ is a regular bi-thin $\underline{\mathbf{K}}$-sequence and $\phi=\bigcap_{k\in\omega}{\leftarrow_{k\mathsf{S}}}$.
    \begin{enumerate}
        \item If $(\leftarrow_k)$ is $\underline{\mathbf{K}}$-sub-Fra\"iss\'e then the conjugacy class of $\phi$ in $\mathbf{C}_\mathbf{K}^\times$ is comeagre.
        \item If $(\leftarrow_k)$ is not $\underline{\mathbf{K}}$-subabsorbing then the conjugacy class of $\phi$ in $\mathbf{C}_\mathbf{K}^\times$ is meagre.
    \end{enumerate}
\end{theorem}

\begin{proof}
    Both statements are proved via the Banach-Mazur game on $\mathbf{C}_\mathbf{K}^\times$.  Specifically, consider the game with two players, $P_0$ and $P_1$, where $P_0$ first plays any non-empty open $O_0\subseteq\mathbf{C}_\mathbf{K}^\times$, next $P_1$ plays any non-empty open $O_1\subseteq O_0$, then $P_0$ again plays any non-empty open $O_2\subseteq O_1$, etc..  The game is won by $P_1$ if $\bigcap_{k\in\omega}{O_k}$ is contained in some set $C\subseteq\mathbf{C}_\mathbf{K}^\times$ given in advance.  If $P_1$ has a winning strategy then $C$ must be comeagre.
    \begin{enumerate}
        \item We show that $P_1$ has a winning strategy when $C$ is the conjugacy class of $\phi$ in $\mathbf{C}_\mathbf{K}^\times$ and $(\leftarrow_k)$ is $\underline{\mathbf{K}}$-sub-Fra\"iss\'e.

        Accordingly, say $P_0$ plays $O_0$ containing some $\psi\in\mathbf{C}_\mathbf{K}^\times$.  Then take ${\twoheadleftarrow}\subseteq\mathbb{P}_m\times\mathbb{P}_m$ with $\psi\subseteq{\twoheadleftarrow_\mathsf{S}}$ and $\mathrm{cl}(\twoheadleftarrow^\times_\mathbf{K})\subseteq O_{2k}$.  By \Cref{KLikeKSubfactorisable}, we have ${\twoheadleftarrow_0}\in\mathbb{P}_{m_0}^\mathbf{K}$ with ${\twoheadleftarrow_0}\leq{\twoheadleftarrow}$ and $\psi\subseteq{\twoheadleftarrow_{0\mathsf{S}}}$.  Take subabsorbing ${\dashv_0}\in\mathbf{K}^{m_0}_{m_0'}$ and let ${\twoheadleftarrow'_0}={{\vdash_0}\circ{\twoheadleftarrow_0}\circ{\dashv_0}}$.  As $(\leftarrow_k)$ is $\underline{\mathbf{K}}$-sub-Fra\"iss\'e and hence $\underline{\mathbf{K}}$-coinitial, we have $k_0\in\omega$ such that there exists ${\sqsupset_0}\in\underline{\mathbf{K}}^{\twoheadleftarrow'_0}_{\leftarrow_{k_0}}$.  Again because $(\leftarrow_k)$ is $\underline{\mathbf{K}}$-sub-Fra\"iss\'e and hence $\underline{\mathbf{K}}$-subabsorbing, we can take $\underline{\mathbf{K}}$-subabsorbing ${\dashv_1}\in\underline{\mathbf{K}}^{\leftarrow_{k_0}}_{\leftarrow_{k_1}}$.  As $\dashv_0$ is subabsorbing, we then have $m_1>m_0'$ and ${\sqni_0}\in\mathbf{K}^{n_{k_1}}_{m_1}$ such that ${{\dashv_0}\circ{\sqsupset_0}\circ{\dashv_1}\circ{\sqni_0}}\subseteq{\geq^{m_0}_{m_1}}$.  Setting ${\twoheadleftarrow_1}={{\sqin_0}\circ{\leftarrow_{k_1}}\circ{\sqni_0}}$, note ${\twoheadleftarrow_1}\leq{\twoheadleftarrow_0}$ so $P_1$ may play $O_1={\twoheadleftarrow_{1\mathbf{K}}^\times}\subseteq O_0$.

        Say $P_2$ next plays $O_2\subseteq O_1$.  Then again we have $m_2>m_1$ and ${\twoheadleftarrow_2}\in\mathbb{P}_{m_2}^\mathbf{K}$ such that $\emptyset\neq\mathrm{cl}(\twoheadleftarrow_{2\mathbf{K}}^\times)\subseteq O_2$.  Again we can then take subabsorbing ${\dashv_2}\in\mathbf{K}^{m_2}_{m_2'}$ and let ${\twoheadleftarrow'_2}={{\vdash_2}\circ{\twoheadleftarrow_2}\circ{\dashv_2}}$.  As ${\dashv_1}\in\underline{\mathbf{K}}^{\leftarrow_{k_0}}_{\leftarrow_{k_1}}$ is $\underline{\mathbf{K}}$-subabsorbing, we then have $k_2>k_1$ and ${\sqsupset_1}\in\underline{\mathbf{K}}^{\twoheadleftarrow'_2}_{\leftarrow_{k_2}}$ such that ${{\dashv_1}\circ{\sqni_0}\circ{\dashv_2}\circ{\sqsupset_1}}\subseteq{\geq_{n_{k_2}}^{n_{k_0}}}$.  Take $\underline{\mathbf{K}}$-subabsorbing ${\dashv_3}\in\underline{\mathbf{K}}^{\leftarrow_{k_2}}_{\leftarrow_{k_3}}$.  As $\dashv_2$ is subabsorbing, we then have $m_3>m_2'$ and ${\sqni_1}\in\mathbf{K}^{n_{k_3}}_{m_3}$ such that ${{\dashv_2}\circ{\sqsupset_1}\circ{\dashv_3}\circ{\sqni_1}}\subseteq{\geq^{m_2}_{m_3}}$.  Setting ${\twoheadleftarrow_3}={{\sqin_1}\circ{\leftarrow_{k_3}}\circ{\sqni_1}}$, note ${\twoheadleftarrow_3}\leq{\twoheadleftarrow_2}$ so $P_1$ may play $O_3={\twoheadleftarrow_{3\mathbf{K}}^\times}\subseteq O_2$.
        
        If $P_1$ continues to play according to this strategy then we will end up in the situation of \Cref{BackAndForthArrows}.  Indeed, the arrow-sequence defined by ${\sqni_j'}={{\dashv_{2j+1}}\circ{\sqni_j}}$ is bi-thin, by \Cref{BackAndForth}.  As ${\twoheadleftarrow_{2j+1}}\subseteq{{\sqin_j'}\circ{\leftarrow_{k_{2j}}}\circ{\sqni_j'}}$, \Cref{BackAndForthArrows} says that $(\twoheadleftarrow_k)$ is also thin and that the resulting refiners ${\leftarrow}=\bigcup_{j\in\omega}\langle\!\leftarrow_j$, ${\twoheadleftarrow}=\bigcup_{j\in\omega}\langle\!\leftarrow_j$ and ${\sqni}=\bigcup_{j\in\omega}{\langle\sqni_j'}$ satisfy $\mathsf{S}_{\twoheadleftarrow}=\mathsf{S}_{\sqni}^{-1}\circ\phi\circ\mathsf{S}_{\sqni}$.  As $\bigcap_{j\in\omega}O_j=\{\mathsf{S}_{\twoheadleftarrow}\}$, this shows $P_1$ wins the game so the conjugacy class of $\phi$ in $\mathbf{C}_\mathbf{K}^\times$ is comeagre.
        
        \item As $(\leftarrow_k)$ is not $\underline{\mathbf{K}}$-subabsorbing, we have $j\in\omega$ such that there is no $\underline{\mathbf{K}}$-subabsorbing ${\nil}\in\underline{\mathbf{K}}^{\leftarrow_j}_{\leftarrow_k}$, for any $k\geq j$.  Whenever ${\sqsupset}\in\underline{\mathbf{K}}^{\leftarrow_j}_{\twoheadleftarrow}$, for some ${\twoheadleftarrow}\in\mathbb{P}_m^\mathbf{K}$, we claim that we have ${\leftarrow}\in\mathbb{P}_n^\mathbf{K}$ with ${\leftarrow}\leq{\twoheadleftarrow}$ such that there is no ${\sqni}\in\mathbf{K}^\leftarrow_{\leftarrow_k}$, for any $k\in\omega$, with ${{\sqsupset}\circ{\geq^m_n}\circ{\sqni}}\subseteq{\geq^{n_j}_{n_k}}$.

        To see this, first take subabsorbing ${\dashv}\in\mathbf{K}^m_{m'}$.  Note we may assume we have ${\sqsupset'}\in\underline{\mathbf{K}}^{\vdash\circ\twoheadleftarrow\circ\dashv}_{\leftarrow_{j'}}$ with ${{\sqsupset}\circ{\geq^m_{m'}}\circ{\sqsupset'}}\subseteq{\geq^{n_j}_{n_{j'}}}$, otherwise we would already be done by taking ${\leftarrow}={\vdash\circ\twoheadleftarrow\circ\dashv}$ and $n=m'$.  As ${\sqsupset}\circ{\dashv}\circ{\sqsupset'}\in\underline{\mathbf{K}}^{\leftarrow_j}_{\leftarrow{j'}}$ is not $\underline{\mathbf{K}}$-subabsorbing, we then have ${\sqsupset_G}\in\underline{\mathbf{K}}^{\leftarrow_{j'}}_{\leftarrow_G}$ such that there is no ${\sqni}\in\underline{\mathbf{K}}^{\leftarrow_G}_{\leftarrow_k}$, for any $k\in\omega$, with ${{\sqsupset}\circ{\dashv}\circ{\sqsupset'}\circ{\sqsupset_G}\circ{\sqni}}\subseteq{\geq^{n_j}_{n_k}}$.  However, as $\dashv$ is subabsorbing, we do have ${\sqsupset''}\in\mathbf{K}^G_{\mathbb{P}_n}$ with ${{\dashv}\circ{\sqsupset'}\circ{\sqsupset_G}\circ{\sqsupset''}}\subseteq{\geq^m_n}$.  Now set ${\leftarrow}={\sqsubset''}\circ{\leftarrow_G}\circ{\sqsupset''}\leq{\twoheadleftarrow}$ and note that, if we had ${\sqni}\in\mathbf{K}^\leftarrow_{\leftarrow_k}$ with ${{\sqsupset}\circ{\geq^m_n}\circ{\sqni}}\subseteq{\geq^{n_j}_{n_k}}$, then
        \[{{\sqsupset}\circ{\dashv}\circ{\sqsupset'}\circ{\sqsupset_G}\circ{\sqsupset''}\circ{\sqni}}\ \ \subseteq\ \ {{\sqsupset}\circ{\geq^m_n}\circ{\sqni}}\ \ \subseteq\ \ {\geq^{n_j}_{n_k}},\]
        contradicting our assumption on $\sqsupset_G$.  Thus this completes the proof of the claim.

        Now we use the claim to show that $P_1$ has a winning strategy in the Banach-Mazur game where $C$ is the complement of the conjugacy class of $\phi$ in $\mathbf{C}_\mathbf{K}^\times$.  To start with take any sequence $(i_k)\subseteq\omega$ such that each element of $\omega$ appears infinitely many times and $i_k\leq 2k$, for all $k\in\omega$.  All $P_1$'s moves will be of the form $O_{2k+1}={\twoheadleftarrow^\times_{(2k+1)\mathbf{K}}}$, for some ${\twoheadleftarrow_k}\in\mathbb{P}_{m_k}^\mathbf{K}$.  Whenever $P_0$ plays $O_{2k}$, first $P_1$ will also choose $\twoheadleftarrow_{2k}$ with $\mathrm{cl}(\twoheadleftarrow^\times_{2k\mathbf{K}})\subseteq O_{2k}$ and ${\twoheadleftarrow_{2k}}\leq{\twoheadleftarrow_{2k-1}}$.  Only after that will $P_1$ make their own move $O_{2k+1}={\twoheadleftarrow^\times_{(2k+1)\mathbf{K}}}$ with ${\twoheadleftarrow_{2k+1}}\leq{\twoheadleftarrow_{2k}}$.
        
        Accordingly, assume $P_0$ has just played $O_{2k}\subseteq{\twoheadleftarrow_{(2k-1)\mathsf{S}}}$, for some $k\in\omega$, and take $\psi\in O_{2k}$.  Then take ${\twoheadleftarrow}\subseteq\mathbb{P}_m\times\mathbb{P}_m$ such that $\psi\in{\twoheadleftarrow^\times_\mathbf{K}}$ and $\mathrm{cl}(\twoheadleftarrow^\times_\mathbf{K})\subseteq O_{2k}$.  By \Cref{ArrowPhi}, we have $m'\in\omega$ with $\mathbb{P}_k^{\langle\!\leftarrow^\psi_{m'}}\in\mathsf{C}\mathbb{P}$ and $\mathbb{P}_k^{\langle\!\leftarrow^{\psi^{-1}}_{m'}}\in\mathsf{C}\mathbb{P}$.  By \Cref{KLikeKSubfactorisable}, we then have ${\twoheadleftarrow_{2k}}\in\mathbb{P}_{m_{2k}}^\mathbf{K}$ below $\twoheadleftarrow$, $\leftarrow^\psi_{m'}$ and $\twoheadleftarrow_{2k-1}$ with $\psi\in{\twoheadleftarrow^\times_{2k\mathbf{K}}}$.  Given any ${\sqsupset_k}\in\underline{\mathbf{K}}^{\leftarrow_j}_{\twoheadleftarrow_{i_k}}$, the claim above yields ${\twoheadleftarrow_{2k+1}}\in\mathbb{P}_{m_{2k+1}}^\mathbf{K}$ with ${\twoheadleftarrow_{2k+1}}\leq{\twoheadleftarrow_{2k}}$ such that there is no ${\sqni}\in\mathbf{K}^{\twoheadleftarrow_{2k+1}}_{\leftarrow_l}$, for any $l\in\omega$, with ${{\sqsupset_k}\circ{\geq^{m_{i_k}}_{m_{2k+1}}}\circ{\sqni}}\subseteq{\geq^{n_j}_{n_l}}$.  Now just note that we can define $P_1$'s strategy in this way with choices of $\sqsupset_k$ that exhaust all morphisms in $\underline{\mathbf{K}}$ having domain $\twoheadleftarrow_i$, for any $i\in\omega$, as there can only be countably many such morphisms.
        
        It only remains to show that $P_1$ wins any game played in this way, i.e.~that the function $\psi=\bigcap_{k\in\omega}{\twoheadleftarrow_{k\mathsf{S}}}$ resulting from the regular bi-thin $\underline{\mathbf{K}}$-sequence $(\twoheadleftarrow_k)$ is not conjugate to $\phi$.  Indeed, say we had $\theta\in\mathbf{C}_\mathbf{K}^\times$ with $\psi=\theta^{-1}\circ\phi\circ\theta$.  By \Cref{ArrowPhiLebesgue}, we have $n>n_j$ with ${\leftarrow^\phi_n}\leq{\leftarrow_j}$.  As $\mathbb{P}$ is $\mathbf{K}$-regular, we have $\mathbb{P}$-neighbourly ${\dashv}\in\mathbf{K}^n_{n'}$.  As $\theta\in\mathbf{C}_\mathbf{K}$, we have ${\sqsupset}\in\mathbf{K}^{n'}_m$ with ${\sqsupset}\subseteq{\sqsupset_\theta}$.  Then \Cref{ConjugacyLemma} yields $i\in\omega$ with ${\twoheadleftarrow_i}\leq{{\sqsubset}\circ{\leftarrow^\phi_{n'}}\circ{\sqsupset}}$.  For some $k\in\omega$, we must have chosen $i_k=i$ and ${\sqsupset_k}={{\geq^{n_j}_n}\circ{\dashv}\circ{\sqsupset}\circ{\geq^m_{m_i}}}(\in\underline{\mathbf{K}}^{\leftarrow_j}_{\twoheadleftarrow_i})$.  Noting that $\theta\subseteq{\sqsupset_\mathsf{S}}$, \Cref{ArrowPhiLebesgue} again yields $m'>m_{2k+1}$ with ${\leftarrow^\theta_{m'}}\leq{\sqsupset}$ and ${\leftarrow^\psi_{m'}}\leq{\twoheadleftarrow_{2k+1}}$.  As $\theta^{-1}\in\mathbf{C}_\mathbf{K}$, we have ${\sqni}\in\mathbf{K}_{n''}^{m'}$ with ${\sqni}\subseteq{\sqsupset_{\theta^{-1}}}$ and hence ${\sqin}\leq{\leftarrow^\theta_{m'}}\leq{\sqsupset}$.  Then \Cref{ConjugacyLemma} again yields $l$ with ${\leftarrow_l}\leq{\sqin}\circ{\leftarrow^\psi_{m'}}\circ{\sqni}$ so ${{\geq_{m'}^{m_{2k+1}}}\circ{\sqni}\circ{\geq_{n_l}^{n''}}}\in\underline{\mathbf{K}}_{\leftarrow_l}^{\twoheadleftarrow_{2k+1}}$.  As $\sqsupset$ is $\wedge$-preserving and $\dashv$ is $\mathbb{P}$-neighbourly,
        \[\qquad{{\dashv}\circ{\sqsupset}\circ{\geq^m_{m'}}\circ{\sqni}}\ \ \subseteq\ \ {{\dashv}\circ{\sqsupset}\circ{\geq^m_{m'}}\circ{\leq^m_{m'}}\circ{\sqsubset}\circ{\geq^{n'}_{n''}}}\ \ \subseteq\ \ {{\dashv}\circ{\wedge}\circ{\geq^{n'}_{n''}}}\ \ \subseteq\ \ {\geq^n_{n''}}\]
        and hence ${{\sqsupset_k}\circ{\geq^{m_i}_{m_{2k+1}}}\circ({\geq^{m_{2k+1}}_{m'}}\circ{\sqni}\circ{\geq_{n_l}^{n''}})}={{\geq^{n_j}_n}\circ{\dashv}\circ{\sqsupset}\circ{\geq^m_{m'}}\circ{\sqni}\circ{\geq_{n_l}^{n''}}}\subseteq{\geq^{n_j}_{n_l}}$, contradicting our choice of $\twoheadleftarrow_{2k+1}$.  Thus $P_1$ does indeed win the game and this shows that the conjugacy class of $\phi$ in $\mathbf{C}_\mathbf{K}^\times$ is meagre. \qedhere
    \end{enumerate}
\end{proof}

\section{The Pseudoarc}

To describe a natural subcategory of $\mathbf{G}$ that our theory can be applied to, let us first recall some basic graph theoretic terminology.  First recall that for us a graph is a finite non-empty set $G$ of \emph{vertices} together with a reflexive symmetric \emph{edge relation} ${\sqcap}\subseteq G\times G$.  The transitive closure of $\sqcap$ is the \emph{path-relation} $\sim$.  A graph $G$ is \emph{connected} if the path relation is trivial, i.e.~if ${\sim}=G\times G$.  An \emph{edge} is a connected graph with two vertices, i.e.~a graph of the form $E=\{e,e'\}$ with $e\neq e'$ and ${\sqcap}=E\times E$.

We consider any subset $H$ of a graph $G$ as a graph in its own right with respect to the restricted edge relation ${\sqcap}|_H^H$.  The number of edges in a graph $G$ will be denoted by
\[\mathsf{e}(G)=|\{E\subseteq G:E\text{ is an edge}\}|.\]
The \emph{order} of a vertex $g\in G$ is the number of edges it is contained in, which we denote by \[\mathsf{o}(g)=|\{E\subseteq G:E\text{ is an edge containing }g\}|=|g^\sqcap|-1.\]
An \emph{end} is a vertex of order at most $1$.

A \emph{path} is a connected graph $P$ with at least one end in which each vertex has order at most $2$.  In this case $\mathsf{e}(P)$ is the \emph{length} of the path.  So if we fix some distinct potential vertices $(p^n_m)_{m,n\in\omega}$ in advance then every path of length $n$ is isomorphic to the path $P_n=\{p_0^n,\ldots,p_n^n\}$ with the edge relation given by
\[p^n_j\mathrel{\sqcap}p^n_k\qquad\Leftrightarrow\qquad|j-k|\leq1.\]

Given any path $P$, we see that $Q\subseteq P$ is connected if and only if $Q$ is itself a path.  In this case we call $Q$ a \emph{subpath} of $P$.  For any $q,q'\in Q$ there is a minimal subpath containing $q$ and $q'$ which we denote by $[q,q']$, e.g.~if we are considering any vertices $p^n_j$ and $p^n_k$ in the path $P_n$, for some $n\in\omega$, then
\[[p^n_j,p^n_k]=[p^n_k,p^n_j]=\{p^n_l:j\leq l\leq k\text{ or }k\leq l\leq j\}.\]
We further define $(q,q']=[q,q']\setminus\{q\}$, $[q,q')=[q,q']\setminus\{q'\}$ and $(q,q')=[q,q']\setminus\{q,q'\}$.

\subsection{Existence and Uniqueness}

To construct the pseudoarc we consider the full subcategory $\mathbf{P}$ of paths in $\mathbf{G}$.  So each hom-set $\mathbf{P}_Q^R$ consists of all edge-preserving co-bijective relations ${\sqsupset}\subseteq R\times Q$ between the paths $Q$ and $R$.  In particular, for the paths $P_n$ described explicitly above, $\mathbf{P}_{P_n}^{P_n}$ consists of just the identity morphism $=_{P_n}$ and, if $n>0$, the order reversing isomorphism $\iota$, where $\iota(p^n_k)=p^n_{n-k}$ for all $k\leq n$.   On the other hand, if $m<n$ then $\mathbf{P}_{P_m}^{P_n}$ is empty while $\mathbf{P}_{P_n}^{P_m}$ consists entirely of non-injective relations.  In particular, paths are isomorphic if and only if they have the same length.  It follows that every morphism in $\mathbf{P}$ can be decomposed into \emph{prime} morphisms $\sqsupset$, i.e.~such that $\sqsupset$ itself is not an isomorphism but ${\sqsupset}={\sqni}\circ{\sqnii}$ implies either $\sqni$ or $\sqnii$ is an isomorphism.

\begin{proposition}\label{PrimeDecomposition}
    Every non-isomorphism in $\mathbf{P}$ is a finite product of prime morphisms.
\end{proposition}

\begin{proof}
    Take any non-isomorphism ${\sqsupset}\in\mathbf{P}_Q^R$.  If $\sqsupset$ itself is prime we are done.  Otherwise ${\sqsupset}={\sqni}\circ{\sqnii}$ for some non-isomorphisms ${\sqni}\in\mathbf{P}^R_P$ and ${\sqnii}\in\mathbf{P}^P_Q$ where $P$ necessarily has length strictly between that of $Q$ and $R$.  Thus the statement follows by induction on the difference in length between the domain and codomain.
\end{proof}

Let us call a morphism ${\sqsupset}\in\mathbf{P}_Q^R$ \emph{simple} if $\mathsf{e}(Q)=\mathsf{e}(R)-1$.  For example, given any $m\leq n$, we have a simple morphism ${\sqsupset}\in\mathbf{P}^{P_n}_{P_{n+1}}$ defined by
\[p^n_j\sqsupset p^{n+1}_k\qquad\Leftrightarrow\qquad m\geq j=k\quad\text{or}\quad m\leq j=k-1.\]
We could also consider the subrelation obtained by replacing $\geq$ or $\leq$ above with $>$ or $<$ respectively.  Note every simple morphism is prime, by the same argument as in the proof of \Cref{PrimeDecomposition}.  However, $\mathbf{P}$ also has other prime morphisms and to describe these we first need to make a few more definitions and observations.

Let us call a morphism ${\sqsupset}\in\mathbf{P}_Q^R$ \emph{edge-injective} if the restriction ${\sqsupset}|_E$ is injective for any edge $E\subseteq Q$.  So $E=[q,q']=\{q,q'\}$, for some $q,q'\in Q$, and ${\sqsupset}|_E$ being injective means that $q^\sqsubset$ and $q'^\sqsubset$ are distinct singletons in $R$ (e.g.~if $q^\sqsubset$ were not a singleton then, as $\sqsupset$ is $\sqcap$-preserving, we would have $q'^\sqsubset\subseteq q^\sqsubset$ and hence there is no $r\in q^\sqsubset$ with $r^\sqsupset=\{q'\}$, contradicting injectivity).  So an edge-injective morphism is just a function which maps each edge of $Q$ onto an edge of $R$ or, put another way, a function which always changes as you move from one vertex in the domain to an adjacent one.  It follows that a composition of edge-injective morphisms is again edge-injective.  We also have the following.

\begin{proposition}
    The non-edge-injective morphisms form an ideal in $\mathbf{P}$.
\end{proposition}

\begin{proof}
    Take any ${\sqsupset}\in\mathbf{P}_Q^R$ and ${\sqni}\in\mathbf{P}_R^S$.  If $\sqsupset$ is not edge-injective then we have an edge $E\subseteq Q$ such that ${\sqsupset}|_E$ is not injective.  Then $({\sqni}\circ{\sqsupset})|_E={{\sqni}\circ{\sqsupset}|_E}$ is not injective either, showing that ${\sqni}\circ{\sqsupset}$ is also not edge-injective.  On the other hand, if $\sqni$ is not edge-injective then we have an edge $F\subseteq R$ such that ${\sqni}|_F$ is not injective.  Then we have an edge $E\subseteq Q$ with $E^\sqsubset=F$ and so again $({\sqni}\circ{\sqsupset})|_E={{\sqni}|_F\circ{\sqsupset}|_E}$ is not injective.
\end{proof}

Let us call $t\in\mathbf{P}_D^E$ a \emph{turn} if $t$ is edge-injective, $E$ is an edge and $D$ has length two, i.e.~$\mathsf{e}(D)=2$.  So every turn is isomorphic to (i.e.~can be composed with isomorphisms on either side to obtain) the turn $t\in\mathbf{P}_{P_2}^{P_1}$ where $t(p^2_0)=p^1_0=t(p^2_2)$ and $t(p^2_1)=p^1_1$.  The \emph{turning number} of any edge-injective $f\in\mathbf{P}_Q^R$ is defined to be the number $\mathsf{t}(f)$ of turns of the form $f|_D$, for subpaths $D\subseteq Q$ of length two.

\begin{proposition}
    For any edge-injective $f\in\mathbf{P}_Q^R$ and $g\in\mathbf{P}_R^S$,
    \[\mathsf{t}(g\circ f)\geq\mathsf{t}(g)+\mathsf{t}(f).\]
\end{proposition}

\begin{proof}
    If $f|_D$ is a turn then $f[D]$ is an edge and hence so too is $g[f[D]]$ so $(g\circ f)|_D$ is also a turn.  On the other hand, if $g|_C$ is a turn then, taking a minimal subpath $D\subseteq Q$ with $f[D]=C$, we see that $\mathsf{e}(D)=\mathsf{e}(C)=2$ so $(g\circ f)|_D$ is a turn even though $f|_D$ is not.  So in this way we get distinct turns of $g\circ f$ from each of the turns of $f$ and $g$.  In particular, this implies that $\mathsf{t}(g\circ f)\geq\mathsf{t}(g)+\mathsf{t}(f)$.
\end{proof}

An edge-injective $f\in\mathbf{P}_Q^R$ is an isomorphism precisely when it has no turns, i.e.~when $\mathsf{t}(f)=0$.  Call $h\in\mathbf{P}_Q^R$ a \emph{hook} if $h$ is edge-injective with exactly one turn, i.e.~$\mathsf{t}(h)=1$.  So each hook is isomorphic to one of the form $h\in\mathbf{P}_{P_{m+n}}^{P_n}$ where $1\leq m\leq n$ and
\[h(p^{m+n}_k)=\begin{cases}p^n_k&\text{if }k\leq n\\p^n_{2n-k}&\text{if }k\geq n\end{cases}\]

\begin{proposition}
    Every hook $h\in\mathbf{P}_Q^R$ is prime.
\end{proposition}

\begin{proof}
    First note $h$ is not an isomorphism, as $\mathsf{t}(h)=1$.  But if $h=f\circ g$ then either $f$ or $g$ must be an isomorphism, as $\mathsf{t}(f)+\mathsf{t}(g)\leq\mathsf{t}(h)=1$.  Thus $h$ is prime.
\end{proof}

Let us call a morphism ${\sqsupset}\in\mathbf{P}_Q^R$ \emph{proper} if ${\sqsupset}|_S\notin\mathbf{P}_S^R$ for any proper subpath $S\subsetneqq Q$.  If $\mathsf{E}(P)$ denotes the ends of any given path $P$ (e.g.~$\mathsf{E}(P_n)=\{p^n_0,p^n_n\}$) and $\sqsupset$ is a function then we see that $\sqsupset$ is proper precisely when $\mathsf{E}(R)^\sqsupset=\mathsf{E}(Q)$.

Let us also call $s\in\mathbf{P}_Q^R$ a \emph{snake} if $s$ is edge-injective with exactly two turns, i.e.~$\mathsf{t}(s)=2$.  So every proper snake is isomorphic to one of the form $s\in\mathbf{P}^{P_{l+m+n}}_{P_{l-m+n}}$ where $l,n>m$ and
\[s(p^{l+m+n}_k)=\begin{cases}p^{l-m+n}_k&\text{if }k\leq l\\p^{l-m+n}_{2l-k}&\text{if }l\leq k\leq l+m\\p^{l-m+n}_{k-2m}&\text{if }l+m\leq k\leq l+m+n.\end{cases}\]

\begin{proposition}\label{PrimeSnake}
    A snake $s\in\mathbf{P}_Q^R$ is prime precisely when it is proper.
\end{proposition}

\begin{proof}
    Again first note $s$ is not an isomorphism, as $\mathsf{t}(s)=2$.  So if $s$ is not prime then $s=f\circ g$ for some $f$ and $g$ that are not isomorphisms and hence $\mathsf{t}(f),\mathsf{t}(g)\geq1$.  But then both $f$ and $g$ must be hooks, as $\mathsf{t}(f)+\mathsf{t}(g)\leq\mathsf{t}(s)=2$.  But hooks are not proper, which would imply their composition $s$ is also not proper.

    For the converse, first let $Q_0,Q_1,Q_2\subseteq Q$ be the maximal subpaths on which $s$ is injective, numbered so that $Q_0\cap Q_2=\emptyset$.  If $s$ is not proper then $s$ must be surjective on one of them, say $Q_0$, and hence $s|_{Q_0\cup Q_1}\in\mathbf{P}^R_{Q_0\cup Q_1}$.  Define $f:Q\rightarrow Q_0\cup Q_1$ by $f(q)=s(q)$, for all $q\in Q_0\cup Q_1$, and $f(q)=s|_{Q_1}^{-1}(s(q))$, for all $q\in Q_2$.  As $s[Q_1]=s[Q_2]$, it follows that $f$ is a well-defined morphism in $\mathbf{P}^{Q_0\cup Q_1}_Q$ with $s=s|_{Q_0\cup Q_1}\circ f$.  As neither $s|_{Q_0\cup Q_1}$ nor $f$ are isomorphisms, this shows that $s$ is not prime.
\end{proof}

We have now exhibited all the prime morphisms, at least within the wide subcategory $\mathbf{F}$ consisting of all edge-preserving surjective functions between paths.

\begin{proposition}\label{PrimesInF}
    Every prime in $\mathbf{F}$ is either simple, a hook or a proper snake.
\end{proposition}

\begin{proof}
    Say we have prime $f\in\mathbf{F}_Q^R$ that is not edge-injective.  This means we have some $e\in R$ such that $f^{-1}\{e\}$ contains an edge $E$.  Consider the quotient graph
    \[Q'=\{\{q\}:q\in Q\setminus E\}\cup\{E\}\]
    with the canonical quotient function $f':Q'\rightarrow R$ where, for all $S\in Q'$ and $r\in R$,
    \[\phi'(S)=r\qquad\Leftrightarrow\qquad S\subseteq f^{-1}\{r\}.\]
    Now note that the restriction of the usual membership relation ${\ni_Q}\subseteq Q'\times Q$ is a simple morphism in $\mathbf{F}_Q^{Q'}$.  Also $f'\in\mathbf{F}_{Q'}^R$ and $f={f'}\circ{\ni_Q}$ so, as $f$ is prime, $f'$ must be an isomorphism and hence $f$ too is simple.

    By \Cref{PrimeSnake}, the only remaining possibility to consider is that there is some edge-injective $f\in\mathbf{F}^R_Q$ with $\mathsf{t}(f)\geq3$ that could be prime.  To see that such an $f$ could not be prime, first let $Q_0,Q_1,\ldots,Q_{\mathsf{t}(f)}\subseteq Q$ be the maximal subpaths on which $f$ is injective, numbered so that
    \[Q_j\cap Q_k\neq\emptyset\qquad\Leftrightarrow\qquad|j-k|\leq 1.\]
    For each $k<\mathsf{t}(f)$, note that either $f[Q_n]\subseteq f[Q_{n+1}]$ or $f[Q_n]\supseteq f[Q_{n+1}]$.  Also $f[Q_0]\supseteq f[Q_1]$ and $f[Q_{\mathsf{t}(f)-1}]\subseteq f[Q_{\mathsf{t}(f)}]$, from which it follows there must be at least one $k$ with $f[Q_{k-1}]\supseteq f[Q_k]\subseteq f[Q_{k+1}]$.  Then we have a subpath $S=Q_k\cup f|_{Q_{k+1}}^{-1}[f[Q_k]]$ and we may again consider the quotient graph
    \[Q'=\{\{q\}:q\in Q\setminus S\}\cup\{S\}.\]
    Take $q_-\in Q$ with $\{q_-\}=Q_k\cap Q_{k-1}$ and let $q_+$ be the other end of $Q'$, necessarily with $f(q_-)=r=f(q_+)$, for some $r\in R$.  Define $g\in\mathbf{F}^R_{Q'}$ by $g(\{q\})=f(q)$, for all $q\in Q\setminus S$, and $g(S)=r$.  Further define $h\in\mathbf{F}^{Q'}_Q$ by
    \[h(q)=\begin{cases}\{q\}&\text{if }q\in Q\setminus S\\\{f|_{Q_{k-1}}^{-1}(f(q))\}&\text{if }q\in S\setminus\{q_-,q_+\}\\S&\text{if }q\in\{q_-,q_+\}.\end{cases}.\]
    Then $f=g\circ h$, $\mathsf{t}(g)=\mathsf{t}(f)-2\geq1$ and $\mathsf{t}(h)=2$, showing that $f$ is not prime.
\end{proof}

Like in \cite[\S2.7]{BaBiVi}, let us call ${\sqsupset}\in\mathbf{G}_Q^R$ \emph{tangled} if, for all connected $S,T\subseteq Q$,
\[S\cup T\text{ is connected}\qquad\Rightarrow\qquad S\subseteq T^{\sqsubset\sqsupset}\text{ or }T\subseteq S^{\sqsubset\sqsupset}.\]

\begin{proposition}\label{TangledWide}
    The tangled morphisms are wide in $\mathbf{P}$.
\end{proposition}

\begin{proof}
    The unique edge-witnessing morphisms in $\mathbf{P}^{P_0}_{P_0}$ and $\mathbf{P}^{P_1}_{P_2}$ are immediately seen to be tangled.  Now assume we have already exhibited tangled morphisms with codomain $P_{n-2}$ and $P_{n-1}$.  We then have tangled ${\sqsupset}\in\mathbf{P}_Q^{P_n\setminus\{p^n_n\}}$, ${\sqni}\in\mathbf{P}_R^{P_k\setminus\{p^n_0,p^n_n\}}$ and ${\sqsupset'}\in\mathbf{P}_{Q'}^{P_k\setminus\{p_0^n\}}$ for some $Q,R,Q'\in\mathbf{P}$ which we may take to have disjoint sets of vertices.  By composing with further morphisms if necessary, we may also assume that we have $q\in\mathsf{E}(Q)$, $r,r'\in\mathsf{E}(R)$ and $q'\in\mathsf{E}(Q')$ with $q^\sqsubset=\{p^n_{n-1}\}=r^{\sqin}$ and $r'^{\sqin}=\{p^n_1\}=q'^{\sqsubset'}$.  We can then form a path $S=Q\cup\{t\}\cup R\cup\{t'\}\cup Q'$ with $q\mathrel{\sqcap}t\mathrel{\sqcap}r$ and $r'\mathrel{\sqcap}t'\mathrel{\sqcap}q'$ and define tangled ${\nil}\in\mathbf{P}_S^{P_n}$ by
    \[{\nil}={\sqsupset}\cup\{(p^n_{n-1},t),(p^n_n,t)\}\cup{\sqni}\cup\{(p^n_0,t'),(p^n_1,t')\}\cup{\sqsupset'}.\qedhere\]
\end{proof}

If ${\sqsupset}\supseteq{\sqni\circ\sqnii}$ then we call $\sqni$ a \emph{left subfactor} and $\sqnii$ a \emph{right subfactor} of $\sqsupset$.

\begin{theorem}\label{SnakeFactor}
    Every snake $s\in\mathbf{P}_S^R$ is a left subfactor of every tangled ${\sqsupset}\in\mathbf{P}_Q^R$.
\end{theorem}

\begin{proof}
    Let $S_+,S_0,S_-\subseteq S$ be the maximal subpaths on which $s$ is injective, indexed so that $S_-\cap S_+=\emptyset$.  Let $r_+$ and $r_-$ be the vertices of $R$ such that $\{r_+\}=s[S_0\cap S_-]$ and $\{r_-\}=s[S_0\cap S_+]$ and hence $s[S_0]=[r_+,r_-]$.  Further let
    \begin{align*}
        Q_0&=\{q\in Q:s[S_0]\subseteq[q,q']^\sqsubset\text{ whenever }q'^\sqsubset=\{r_+\}\text{ or }q'^\sqsubset=\{r_-\}\}.\\
        Q_+&=\{q\in Q:\exists q'\in Q\ (q'^\sqsubset=\{r_+\}\text{ and }(q,q']\cap Q_0=\emptyset)\}.\\
        Q_-&=\{q\in Q:\exists q'\in Q\ (q'^\sqsubset=\{r_-\}\text{ and }(q,q']\cap Q_0=\emptyset)\}.
    \end{align*}

    First we claim that every edge $E\subseteq Q$ is contained in one of these sets.  To see this, say $E\nsubseteq Q_0$ so we have $e\in E\setminus Q_0$.  This means we have $q'\in Q$ with $q'^\sqsubset=\{r_+\}$ or $q'^\sqsubset=\{r_-\}$ such that $s[S_0]\nsubseteq[e,q']^\sqsubset$.  If $q\in[e,q']$ then $s[S_0]\nsubseteq[q,q']^\sqsubset$ too so $[e,q']\cap Q_0=\emptyset$.  Thus $(f,q']\cap Q_0=\emptyset$, where $f\in E\setminus\{e\}$, so $E\subseteq Q_+$ or $E\subseteq Q_-$, proving the claim.

    If $\mathsf{e}(S_0)=1$ then $r_+\mathrel{\sqcap}r_-$ and $Q_+\cap Q_-\subseteq Q_0=\{q\in Q:q^\sqsubset=\{r_+,r_-\}\}$.  If $\mathsf{e}(S_0)>1$ then we claim that $Q_+\cap Q_-=\emptyset$.  For this it suffices to show that $Q_0\cap[q_+,q_-]$ contains at least two vertices, for any $q_+,q_-\in Q$ with $q_+^\sqsubset=\{r_+\}$ and $q_-^\sqsubset=\{r_-\}$ (because then $Q_0\cap((q,q_+]\cup(q,q_-])\neq\emptyset$, for any $q\in Q$).  Replacing $q_+$ or $q_-$ to define a smaller subpath if necessary, we may further assume that $\{r_+\}\neq q^\sqsubset\neq\{r_-\}$, for all $q\in(q_+,q_-)$.  As $\sqsupset$ is tangled, \cite[Proposition 2.51]{BaBiVi} yields $d_+\sqsubset r_+$ and $d_-\sqsubset r_-$ with $d_+\in[q_-,d_-]$ and $d_-\in[q_+,d_+]$.  In particular, $d_+\neq d_-$ because $\mathsf{e}(S_0)>1$ and hence $r_+\not\mathrel{\sqcap}r_-$.  Moreover, for any $q\in[d_+,d_-]$, if $q'^\sqsubset=\{r_+\}$ or $q'^\sqsubset=\{r_-\}$ then $[q,q']$ contains either $[q,q_+]\supseteq[d_-,q_+]$ or $[q,q_-]\supseteq[d_+,q_-]$.  As $[r_-,r_+]\subseteq[d_-,q_+]^\sqsubset$ and $[r_+,r_-]\subseteq[d_+,q_-]^\sqsubset$, this shows that $[d_+,d_-]\subseteq Q_0$, proving the claim.

    Next note that if $q_0\in Q_-\cap Q_0$ then $q\sqsubset r_+$.  Indeed, $q_0\in Q_-$ implies that we have $q'\in Q$ with $q'^\sqsubset=\{r_-\}$ and $(q_0,q']\cap Q_0=\emptyset$.  Replacing $q'$ if necessary we may assume that $q^\sqsubset\neq\{r_-\}$, for any $q\in(q_0,q')$.  Also $q^\sqsubset\neq\{r_+\}$, for any $q\in(q_0,q')$, by the claim just proved above.  Taking $q''$ with $(q_0,q']=[q'',q']$, it follows that $q''\notin Q_0$ so we have $q'''$ with $q'''^\sqsubset=\{r_+\}$ or $q'''^\sqsubset=\{r_-\}$ such that $s[S_0]\nsubseteq[q'',q''']^\sqsubset$.  But it can not be that $q_0\in[q'',q''']$ because this would also imply $q_0\notin Q_0$.  But we also can not have $q'''\in(q_0,q')$ so the only possibility is $q'\in(q_0,q''']=[q'',q''']$.  Thus $s[S_0]\nsubseteq[q'',q']^\sqsubset$ as well and hence $r_+\in s[S_0]\setminus[q'',q']^\sqsubset$, as $r_-\in q'^\sqsubset\subseteq[q'',q']^\sqsubset$.  Then $q_0\not\sqsubset r_+$ would imply that also $s[S_0]\nsubseteq[q_0,q']^\sqsubset$ and hence $q_0\notin Q_0$, a contradiction.  Likewise, we see that if $q\in Q_+\cap Q_0$ then $q\sqsubset r_-$.  We can therefore define ${\sqni}\in\mathbf{P}_Q^S$ with $s\circ{\sqni}\subseteq{\sqsupset}$ by
    \[q^{\sqin}=
    \begin{cases}
        S_0&\text{if }q\in Q_-\cap Q_0\cap Q_+\\
        S_-\cap S_0&\text{if }q\in Q_-\cap Q_0\setminus Q_+\\
        S_+\cap S_0&\text{if }q\in Q_+\cap Q_0\setminus Q_-\\
        s|_{S_k}^{-1}[q^\sqsubset]&\text{ otherwise for }q\in Q_k.
    \end{cases}\qedhere\]
\end{proof}

\begin{proposition}\label{HookSubfactor}
    Every non-simple hook $h\in\mathbf{P}_H^R$ and every proper simple ${\sqni}\in\mathbf{P}_P^R$ is a left subfactor of every tangled ${\sqsupset}\in\mathbf{P}_Q^R$.
\end{proposition}

\begin{proof}
    Let $H_+,H_-\subseteq H$ be the maximal subpaths on which $h$ is injective, indexed so that $h[H_+]=R$.  Let $Q_-$ be a minimal subpath of $Q$ for which ${\sqsupset}|^{h[H_-]}_{Q_-}$ is co-injective.  So we have $q,d\in Q$ and $r,e\in R$ with $Q_-=[q,d]$, $h[H_-]=[r,e]$, $q^\sqsubset=\{r\}$, $d^\sqsubset=\{e\}$ and $h|_{H_-}^{-1}(e)=h|_{H_+}^{-1}(e)$.  As $\sqsupset$ is tangled, \cite[Proposition 2.51]{BaBiVi} yields $q'\sqsubset r$ and $d'\sqsubset e$ with $d'\in[q,q']$ and $q'\in[d,d']$.  As $h$ is not simple, $\mathsf{e}(Q_-)\geq2$ so $q$ and $d$ are not adjacent and hence $q'\neq d'$.  Thus we can define ${\nil}\in\mathbf{P}_Q^H$ with ${\sqsupset}\supseteq h\circ{\nil}$ by
    \[p^{\nil}=\begin{cases}\{h|_{H_-}^{-1}(r)\}&\text{if }p=q'\\\{h|_{H_-}^{-1}(e)\}&\text{if }p=d'\\h|_{H_-}^{-1}[p^\sqsubset]&\text{if }p\in[d,d')\setminus\{q'\}\\h|_{H_+}^{-1}[p^\sqsubset]&\text{otherwise}.\end{cases}\]

    As ${\sqni}\in\mathbf{P}_P^R$ is proper and simple, we have $p\in P\setminus\mathsf{E}(P)$ such that $p^{\sqcap\sqin}$ is an edge of $R$ and $\sqni$ is a bijective function on $P\setminus\{p\}$.  We can then define ${\nil}\in\mathbf{P}_Q^P$ with ${\sqsupset}\supseteq{{\sqni}\circ{\nil}}$ by
    \[q^{\lin}=\begin{cases}\{p\}&\text{if }q^\sqsubset=p^{\sqcap\sqin}\\q^{\sqsubset\sqni}\setminus\{p\}&\text{if }q^\sqsubset\neq p^{\sqcap\sqin}.\end{cases}\qedhere\]
\end{proof}

Note every tangled morphism is edge-witnessing, to which the following result applies.

\begin{proposition}\label{ImproperSubfactor}
    Every improper simple ${\sqni}\in\mathbf{P}_S^R$ with $\mathsf{e}(R)\geq1$ is a left subfactor of every ${\sqsupset}\in\mathbf{P}_Q^R$ that is a composition of at least two edge-witnessing morphisms.
\end{proposition}

\begin{proof}
    As $\sqni$ is improper and simple, we have $e\in\mathsf{E}(S)$ such that $\sqni$ is a bijective function from $S_+=S\setminus\{e\}$ onto $R$.  Take $s\in S$ with $s\mathrel{\sqcap}e$ so $s^\sqsubset=\{r\}$ for some end $r\in\mathsf{E}(R)$.  As $\mathsf{e}(R)\geq1$, we have a (unique) edge $E\subseteq R$ with $r\in E$, necessarily with $e^\sqsubset\subseteq E$.  Let $Q_E$ be a minimal subpath of $Q$ for which ${\sqsupset}|^E_{Q_E}$ is co-injective, so we have $q_E\in\mathsf{E}(Q_E)$ with $q_E^\sqsubset=\{r\}$.  Let $Q_E'=Q_E\setminus\mathsf{E}(Q_E)$ so $q^\sqsubset=E$, for all $q\in Q_E'$.  Take $q_E'\in Q_E'$ with $q_E'\mathrel{\sqcap}q_E$.  As ${\sqsupset}\in\mathbf{P}_Q^R$ is a composition of at least two edge-witnessing morphisms, $\mathsf{e}(Q'_E)\geq1$ so we can define 
    ${\nil}\in\mathbf{P}_Q^S$ with ${\sqsupset}\supseteq{\sqni}\circ{\nil}$ by
    \[q^{\lin}=\begin{cases}\{e\}&\text{if }q=q_E'\\\{s\}&\text{if }q\in Q_E'\setminus\{q_E'\}\\q^{\sqsubset\sqni}\setminus\{e\}&\text{if }q\in Q\setminus Q_E'.\end{cases}\qedhere\]
\end{proof}

From any path $Q\in\mathbf{P}$ we can define another path $\mathsf{F}Q$ consisting of its cliques, i.e. all edges and singleton subsets, under the comparability relation, where $X,Y\in\mathsf{F}Q$ are adjacent precisely when $X\subseteq Y$ or $Y\subseteq X$.  For any ${\sqsupset}\in\mathbf{P}_Q^R$, we also define $\mathsf{F}^\sqsupset\in\mathbf{F}_{\mathsf{F}Q}^{\mathsf{F}R}$ by
\[\mathsf{F}^\sqsupset(X)=X^\sqsubset.\]
This defines a functor $\mathsf{F}$ from $\mathbf{P}$ to $\mathbf{F}$.  Moreover, for each $Q\in\mathbf{P}$, the restriction of the membership relation ${\in^Q}={\in}|^Q_{\mathsf{F}Q}$ defines a morphism in $\mathbf{P}^Q_{\mathsf{F}Q}$.  This yields a natural transformation from $\mathsf{F}$ to the identify functor -- see \cite[Proposition 4.45]{BaBiVi2}.

Let us call a morphism in $\mathbf{P}$ \emph{decomposable} if it can be written as a composition of hooks, proper snakes and simple morphisms.

\begin{proposition}\label{LeftFactorDecomposable}
    Every ${\sqni}\in\mathbf{P}_L^R$ is a left factor of some decomposable ${\sqsupset}\in\mathbf{P}_D^R$.
\end{proposition}

\begin{proof}
    Certainly $\sqni$ is a left factor of ${{\sqni}\circ{\in^L}}={{\in^R}\circ{\mathsf{F}^{\sqni}}}$.  It is also immediately seen that ${\in^R}$ is a composition of proper simple morphisms.  If ${\mathsf{F}^{\sqni}}$ is invertible then we can just compose further with its inverse and we are done.   Otherwise, as ${\mathsf{F}^{\sqni}}\in\mathbf{F}_{\mathsf{F}L}^{\mathsf{F}R}$, it can be decomposed into prime morphisms in $\mathbf{F}$, by the same counting argument as in the proof of \Cref{PrimeDecomposition}.  By \Cref{PrimesInF}, the primes in $\mathbf{F}$ consist entirely of hooks, proper snakes and possibly more simple morphisms so again we are done.
\end{proof}

We call an $\omega$-poset $\mathbb{P}$ \emph{tangled} if, for every $m\in\omega$, we have $n>m$ such that $\geq^m_n$ is a tangled morphism in $\mathbf{G}^m_n$.  As tangled morphisms form an ideal in $\mathbf{G}$, this is consistent with the notion of a tangled $\omega$-poset in \cite[\S2.7]{BaBiVi2}.

Recall that a \emph{continuum} is a connected metrisable compactum.  A continuum is \emph{indecomposable} if it is not the union of two proper subcontinua.  A continuum $X$ is \emph{hereditarily indecomposable} if all subcontinua are indecomposable.  Equivalently this is saying that, for all subcontinua $Y,Z\subseteq X$ with $Y\cap Z\neq\emptyset$, either $Y\subseteq Z$ or $Z\subseteq Y$.

\begin{theorem}\label{PFraisse}
    For any $\mathbf{P}$-poset $\mathbb{P}$ with $\mathsf{e}(\mathbb{P}_0)\geq1$, the following are equivalent.
    \begin{enumerate}
        \item\label{Ptangled} $\mathbb{P}$ is tangled.
        \item\label{PsubFraisse} $\mathbb{P}$ is $\mathbf{P}$-sub-Fra\"iss\'e.
        \item\label{PHD} $\mathsf{S}\mathbb{P}$ is hereditarily indecomposable.
    \end{enumerate}
\end{theorem}

\begin{proof}\
    \begin{itemize}
        \item[\eqref{Ptangled}$\Rightarrow$\eqref{PsubFraisse}]  Assume $\mathbb{P}$ is tangled.  As $\mathbf{P}$ is directed, it suffices to show that $\mathbb{P}$ is $\mathbf{P}$-subabsorbing.  For this it suffices to show that the identity morphisms on each level of $\mathbb{P}$ are subabsorbing.  Accordingly, take any ${\sqsupset}\in\mathbf{P}^{\mathbb{P}_n}_Q$.  Composing with another morphism if necessary, we may assume that $\sqsupset$ is decomposable, by \Cref{LeftFactorDecomposable}, i.e.~${\sqsupset}={\sqsupset_0}\circ\ldots\circ{\sqsupset_k}$ where each $\sqsupset_j$ is simple, a hook or a proper snake.  Take $n_0>n_0'>n$ such that $\geq^n_{n_0'}$ and $\geq^{n_0'}_{n_0}$ are tangled.  By \Cref{SnakeFactor} and \Cref{HookSubfactor,ImproperSubfactor}, we have ${\sqsupset_0'}\in\mathbf{P}^Q_{\mathbb{P}_{n_0}}$ such that ${{\sqsupset_0}\circ{\sqsupset_0'}}\subseteq{{\geq^n_{n_0'}}\circ{\geq^{n_0'}_{n_0}}}\subseteq{\geq^n_{n_0}}$.  Taking $n_1>n_1'>n_0$ such that $\geq^{n_0}_{n_1'}$ and $\geq^{n_1'}_{n_1}$ are tangled, ${\sqsupset_0'}\circ{\geq^{n_0}_{n_1'}}$ is also tangled, as tangled morphisms form an ideal.  So again we have ${\sqsupset_1'}\in\mathbf{P}^Q_{\mathbb{P}_{n_1}}$ with ${{\sqsupset_1}\circ{\sqsupset_1'}}\subseteq{{\sqsupset_0'}\circ{\geq^{n_0}_{n_1'}}\circ{\geq^{n_1'}_{n_1}}}\subseteq{{\sqsupset_0'}\circ{\geq^{n_0}_{n_1}}}$ and hence ${{\sqsupset_0}\circ{\sqsupset_1}\circ{\sqsupset_1'}}\subseteq{{\sqsupset_0}\circ{\sqsupset_0'}\circ{\geq^{n_0}_{n_1}}}\subseteq{{\geq^n_{n_0}}\circ{\geq^{n_0}_{n_1}}}\subseteq{\geq^n_{n_1}}$.  Continuing in this way we eventually obtain ${\sqsupset_k'}\in\mathbf{P}^Q_{\mathbb{P}_{n_k}}$ such that ${{\sqsupset}\circ{\sqsupset_k'}}\subseteq{\geq^n_{n_k}}$, as required.

        \item[\eqref{PsubFraisse}$\Rightarrow$\eqref{Ptangled}]  Assume $\mathbb{P}$ is $\mathbf{P}$-sub-Fra\"iss\'e so, for every $m\in\omega$, we have subabsorbing ${\nil}\in\mathbf{P}^m_n$.  By \Cref{TangledWide}, we have tangled ${\sqsupset}\in\mathbf{P}^{\mathbb{P}_n}_Q$ and we can then take ${\sqsupset'}\in\mathbf{P}^Q_{\mathbb{P}_{n'}}$ with ${\nil}\circ{\sqsupset}\circ{\sqsupset'}\subseteq{\geq^m_{n'}}$.  As tangled morphisms form ideal and any morphism containing a tangled morphism is itself tangled, ${\geq^m_{n'}}$ is also tangled.  This shows $\mathbb{P}$ is tangled.

        \item[\eqref{Ptangled}$\Rightarrow$\eqref{PHD}]  Immediate from \cite[Theorem 2.52]{BaBiVi}.

        \item[\eqref{PHD}$\Rightarrow$\eqref{Ptangled}]  As $\mathbb{P}$ is a $\mathbf{P}$-poset, $\mathbb{P}$ is necessarily prime, as noted at the start of \S\ref{Compatibility}.  If $\mathsf{S}\mathbb{P}$ is hereditarily indecomposable then, in particular, $\mathsf{S}\mathbb{P}$ is a continuum and hence Hausdorff so $\mathbb{P}$ is regular, by \cite[Corollary 2.40]{BaBiVi}.  The other direction of \cite[Theorem 2.52]{BaBiVi} then says that $\mathbb{P}$ must be tangled.\qedhere
    \end{itemize}
\end{proof}

Recall that a continuum is \emph{chainable} if every open cover is refined by a \emph{chain}, i.e.~an open cover forming a path with respect to the overlap relation $O\cap N\neq\emptyset$.  A \emph{pseudoarc} is a chainable hereditarily indecomposable continuum.  By \Cref{TangledWide} and \Cref{PFraisse}, pseudoarcs exist, as first shown by Knaster in \cite{Knaster1922}.  From \Cref{PFraisse} we also get an elementary Fra\"iss\'e theoretic proof of Bing's classic result from \cite{Bing1951} that there is really only one pseudoarc.

\begin{corollary}
    All pseudoarcs are homeomorphic.
\end{corollary}

\begin{proof}
    For any chainable continuum $X$, we have a sequence $(C_n)$ of chains which is both decreasing and coinitial (w.r.t.~refinement) among all open covers of $X$.  We can delete individual points from the open sets of these covers if necessary to ensure that no open set appears in more than one cover in the sequence.  Deleting some covers at the start if necessary, we may also assume that $C_0$ is not the trivial cover consisting of just the whole space.  Deleting the ends of some of these paths if necessary, we may further assume the covers are minimal and hence that the inclusion relation between any two covers is co-injective.  Thus we obtain an $\mathbf{P}$-poset $\mathbb{P}=\bigcup_{n\in\omega}C_n$ ordered by $\subseteq$ such that $\mathbb{P}_n=C_n$, for all $n\in\omega$.  By \cite[Proposition 2.9]{BaBiVi}, $X$ is homeomorphic to $\mathsf{S}\mathbb{P}$.

    Any two pseudoarcs are thus homeomorphic to $\mathsf{S}\mathbb{P}$ and $\mathsf{S}\mathbb{Q}$, for some $\mathbf{P}$-posets $\mathbb{P}$ and $\mathbb{Q}$ satisfying the equivalent conditions in \Cref{PFraisse}.  Thus $\mathbb{P}$ and $\mathbb{Q}$ are both $\mathbf{P}$-sub-Fra\"iss\'e and hence we have sequences of relations as in \Cref{BackAndForth} defining a bi-thin arrow-sequence, which in turn yields a homeomorphism between $\mathsf{S}\mathbb{P}$ and $\mathsf{S}\mathbb{Q}$.
\end{proof}

Consequently, following standard practice, we will now refer to ``the'' pseudoarc.

\subsection{Homogeneity}

Our next goal is to give an elementary Fra\"iss\'e theoretic proof of another classic result of Bing from \cite{Bing1948} stating that the pseudoarc is homogeneous, i.e.~there is a homeomorphism of the pseudoarc mapping any given point to any other.

First observe that subamalgamable morphisms are wide in $\mathbf{P}$, by \Cref{SubFraisseImpliesSubAmalgamation} and \Cref{TangledWide}.  To prove homogeneity we will need to strengthen this, which we do using the theory developed in \cite{IrwinSolecki2006} and \cite{BaBiVi2}.  Specifically, let us call ${\sqsupset}\in\mathbf{P}_Q^R$ \emph{end-preserving} if, for every $r\in\mathsf{E}(R)$, we have $q\in\mathsf{E}(Q)$ with $q^\sqsubset=\{r\}$.  We denote the wide subcategories of $\mathbf{P}$ and $\mathbf{F}$ with end-preserving morphisms by $\mathbf{EP}$ and $\mathbf{EF}$ respectively.

As usual, we say a category $\mathbf{K}$ \emph{has amalgamation} if its identity morphisms are amalgamable in $\mathbf{K}$ (from which it follows that all morphisms are amalgamable in $\mathbf{K}$).

\begin{theorem}\label{PAmalgamation}
    Each of the categories $\mathbf{F}$, $\mathbf{P}$, $\mathbf{EF}$ and $\mathbf{EP}$ has amalgamation.
\end{theorem}

\begin{proof}
    Amalgamation in $\mathbf{F}$ was proved in \cite[Theorem 3.1]{IrwinSolecki2006}, from which amalgamation in $\mathbf{P}$ follows, by \cite[Corollary 4.48]{BaBiVi2}.  To see that amalgamation holds in $\mathbf{EF}$, take any $e\in\mathbf{EF}_Q^P$ and $f\in\mathbf{EF}_R^P$.  Adding one vertex adjacent to each end of the graphs and extending the mappings to these new ends, we obtain proper $e'\in\mathbf{EF}_{Q'}^{P'}$ and $f'\in\mathbf{EF}_{R'}^{P'}$ with $P=P'\setminus\mathsf{E}(Q')$, $Q=Q'\setminus\mathsf{E}(Q')$, $R=R'\setminus\mathsf{E}(Q')$, $e=e'|_Q$ and $f=f'|_R$.  By amalgamation in $\mathbf{F}$, we have $g'\in\mathbf{F}_{S'}^{Q'}$ and $h'\in\mathbf{F}_{S'}^{R'}$ with $e'\circ g'=f'\circ h'$.  As $e'$ and $f'$ are proper, $g'(s)\in\mathsf{E}(Q')$ if and only if $h'(s)\in\mathsf{E}(R')$.  Thus restricting $g'$ and $h'$ to a subpath of $S'$ if necessary, we can ensure $g'$ and $h'$ are also proper.  Letting $S=S'\setminus\mathsf{E}(S')$, $g=g'|_S$ and $h=h'|_S$, it follows that $g\in\mathbf{EF}_S^Q$, $h\in\mathbf{EF}_S^R$ and $e\circ g=f\circ h$, proving amalgamation in $\mathbf{EF}$.  Amalgamation in $\mathbf{EP}$ then follows again from \cite[Corollary 4.48]{BaBiVi2}.
\end{proof}

Given any partition $\mathcal{P}$ of a graph $G$ into non-empty subsets, we define ${\sqcap_\mathcal{P}}\subseteq\mathcal{P}\times\mathcal{P}$ by
\[Q\mathrel{\sqcap_\mathcal{P}}R\qquad\Leftrightarrow\qquad\exists q\in Q\ \exists r\in R\ (q\mathrel{\sqcap_G}r).\]
This turns $\mathcal{P}$ into a graph such that the canonical quotient map $\mathsf{q}_\mathcal{P}$ defined by
\[\mathsf{q}_\mathcal{P}(g)=Q\qquad\Leftrightarrow\qquad g\in Q\]
is edge-preserving and hence a morphism from $G$ to $\mathcal{P}$, i.e.~$\mathsf{q}_\mathcal{P}\in\mathbf{G}_G^\mathcal{P}$.  Given any other morphism ${\sqsupset}\in\mathbf{G}_G^H$ we also have a relation ${\sqsupset_\mathcal{P}}\subseteq H\times\mathcal{P}$ defined by
\[h\sqsupset_\mathcal{P}Q\qquad\Leftrightarrow\qquad\forall q\in Q\ (h\sqsupset q).\]
Note ${{\sqsupset_\mathcal{P}}\circ{\mathsf{q}_\mathcal{P}}}\subseteq{\sqsupset}$ and, moreover, ${\sqsupset_\mathcal{P}}\in\mathbf{G}^H_{\mathcal{P}}$ if (and only if) $\sqsupset_\mathcal{P}$ is co-surjective.

The following is a combinatorial analog of \cite[Theorem 3]{Bing1948}.

\begin{proposition}\label{EndMove}
	For any tangled ${\sqsupset}\in\mathbf{P}_R^S$ and any $v\in R$, we have a partition $\mathcal{P}$ of $R$ forming a path such that ${\sqsupset_\mathcal{P}}$ is co-surjective, $\mathsf{q}_\mathcal{P}(v)\in\mathsf{E}(\mathcal{P})$ and
	\[\mathsf{E}(R)\ni r\sqsubset s\in\mathsf{E}(S)\qquad\Rightarrow\qquad\mathsf{q}_\mathcal{P}(r)\sqsubset s.\]
\end{proposition}

\begin{proof}
	If $\mathsf{e}(S)=0$ then we can just take $\mathcal{P}$ to be the trivial partition of $R$, i.e.~$\mathcal{P}=\{R\}$.  If $\mathsf{e}(S)=1$ then we can take any $s\in S=\mathsf{E}(S)$ and let $\mathcal{P}=\{s^\sqsupset,R\setminus s^\sqsupset\}$.  Now assume the result holds whenever $\mathsf{e}(R)<n$ and that now $\mathsf{e}(R)=n\geq2$.  Let $R_v$ be a minimal subpath of $R$ containing $v$ such that we have $s_v\in\mathsf{E}(S)$ with $\mathsf{E}(R_v)\subseteq s_v^\sqsupset$ and $r^\sqsubset=\{s_v\}$, for some $r\in R_v$.  As $\sqsupset$ is tangled, minimality implies that we have a proper subpath $T\subseteq S$ containing $s_v$ such that ${\sqsupset}|^T_{R_v}\in\mathbf{P}_{R_v}^T$, which is again tangled.  By assumption, we have a partition $\mathcal{P}_v$ of $R_v$ forming a path such that ${\sqsupset_{\mathcal{P}_v}}$ is co-surjective, $\mathsf{q}_{\mathcal{P}_v}(v)\in\mathsf{E}(\mathcal{P}_v)$ and
	\[\mathsf{E}(R_v)\in r\sqsubset t\in\mathsf{E}(T)\qquad\Rightarrow\qquad\mathsf{q}_{\mathcal{P}_v}(r)\sqsubset t.\]
    Take $E_v\in\mathcal{P}_v$ containing some $r_v\in\mathsf{E}(R_v)$ such that no $Q\in[\mathsf{q}_{\mathcal{P}_v}(v),E_v)$ contains an end of $\mathsf{E}(R_v)$.  As $\mathsf{E}(R_v)\ni r_v\sqsubset s_v\in\mathsf{E}(T)$, this means $E_v=\mathsf{q}_{\mathcal{P}_v}(r_v)\sqsubset s_v$, so if we let $P_{s_v}=s_v^\sqsupset\setminus\bigcup[\mathsf{q}_{\mathcal{P}_v}(v),E_v)$ then $E_v\subseteq P_{s_v}$.  Taking $E_v'\mathrel{\sqcap_{\mathcal{P}_v}}E_v$, it follows that $P_{s_v}\cap R_v$ contains elements adjacent to those of $E'_v$ but no other $Q\in[\mathsf{q}_{\mathcal{P}_v}(v),E_v)$.  Moreover, $P_{s_v}\setminus R_v$ can not contain an elements adjacent to any $Q\in[\mathsf{q}_{\mathcal{P}_v}(v),E_v)$ because none of these subsets contain ends of $R_v$.  Thus replacing $E_v$ in $\mathcal{P}_v$ with $P_{s_v}$ still defines a path.  For each $s\in S\setminus\{s_v\}$, let
\[P_s=\{r\in R\setminus\bigcup[\mathsf{q}_{\mathcal{P}_v}(v),E_v):r\in s^\sqsupset\subseteq[r,s_v]\}.\]
Then we see that $\mathcal{P}=[\mathsf{q}_{\mathcal{P}_v}(v),E_v)\cup\{P_s:s\in S\}$ is the desired partition of $R$.
\end{proof}

\begin{theorem}
    The pseudoarc is homogeneous.
\end{theorem}

\begin{proof}
    By \Cref{PAmalgamation} and \cite[Corollary 3.8]{Ku2014}, we have a graded $\mathbf{EP}$-Fra\"iss\'e poset $\mathbb{P}$, i.e.~such that the restrictions $\geq^m_n$ of the order relation to pairs of levels defines a Fra\"iss\'e sequence in $\mathbf{EP}$ in the sense of \cite{Ku2014} (and hence Fra\"iss\'e in $\mathbf{P}$ too, as every morphism of $\mathbf{P}$ is a left factor of a morphism in $\mathbf{EP}$).  In particular, $\mathbb{P}$ is an $\mathbf{EP}$-poset so we have a decreasing sequence $e_n\in\mathsf{E}(\mathbb{P}_n)$ of ends defining an element of the spectrum $E=\{e_n:n\in\omega\}\in\mathsf{S}\mathbb{P}$.  To prove homogeneity it suffices to construct a homeomorphism mapping $E$ onto any other given $S\in\mathsf{S}\mathbb{P}$.
    
    To do this first note that, for any sufficiently large $n_0\in\omega$ and $s_0\in S\cap\mathbb{P}_{n_0}$, we have $f_0\in\mathbf{F}^0_{n_0}$ with $f_0(s_0)=e_0$.  Take $n>n_0$ such that $S\cap\mathbb{P}_n\leq s_0$, so we have $f\in\mathbf{F}^{n_0}_n$ with $f\subseteq{\geq}$ and $f[S\cap\mathbb{P}_n]=\{s_0\}$.  By \Cref{PFraisse}, we have $n'>n$ such that $\geq^n_{n'}$ is tangled.  By \Cref{EndMove}, for any $s\in S\cap\mathbb{P}_{n'}$, we have $g\in\mathbf{F}^{\mathbb{P}_n}_R$ and $g'\in\mathbf{F}_{\mathbb{P}_{n'}}^R$ with $g\circ g'\subseteq{\geq^n_{n'}}$ and $g(s)\in\mathsf{E}(R)$.  In particular, $s\leq g(g'(s))\leq s_0$ and hence $f_0(f(g(g'(s))))=e_0$.  So $f_0\circ f\circ g$ preserves one end of $R$ and, by composing with any $h\in\mathbf{F}^R_{\mathbb{P}_m}$ mapping one end of $\mathbb{P}_m$ to $g'(s)$ and the other to some element of $(f_0\circ f\circ g)^{-1}[\mathsf{E}(\mathbb{P}_0)\setminus\{e_0\}]$, we obtain $f_0\circ f\circ g\circ h\in\mathbf{EF}_{\mathbb{P}_m}^{\mathbb{P}_0}$.  As $\mathbb{P}$ is $\mathbf{EP}$-Fra\"iss\'e, we may replace $h$ if necessary to further ensure that $f_0\circ f\circ g\circ h\subseteq{\geq^0_m}$.  In particular, $f_0\circ f\circ g\circ h(e_m)=e_0$.
    
    Let $R'$ be the path obtained from $R$ by adding an extra vertex $r'$ adjacent to $g'(s)$.  By \cite[Proposition 2.39]{BaBiVi2}, we have $n''>n'$ and $t\in S\cap\mathbb{P}_{n''}$ with $t\vartriangleleft_{\mathbb{P}_{n''}}s$.  This means we have $g''\in\mathbf{F}_{\mathbb{P}_{n''}}^{R'}$ such that $g''(t)=r'$ and $g''|_{\mathbb{P}_{n''}\setminus\{t\}}\subseteq(g'\circ{\geq_{n''}^{n'}})$.  Take $f'\in\mathbf{F}_{R'}^{\mathbb{P}_{n_0}}$ with $f'(r')=s_0$ and $f'|_R=f\circ g$ and hence $f'\circ g''\subseteq{\geq}$.  We can then take $m_1>m$ such that we have an edge $E\subseteq\mathbb{P}_{m_1}$ containing $e_{m_1}$ with $E\leq e_m$.  We then have a function $h'\in\mathbf{F}^{R'}_{\mathbb{P}_{m_1}}$ with $h'(e_{m_1})=r'$ and $h'|_{\mathbb{P}_{m_1}\setminus\{e_{m_1}\}}\subseteq{h\circ{\geq}}$, and hence $f_0\circ f'\circ h'\subseteq{\geq^0_{m_1}}$.  Let $g_0=f'\circ h'$ so $f_0\circ g_0\subseteq{\geq}$.

    As $\mathbb{P}$ is $\mathbf{P}$-Fra\"iss\'e, \cite[Proposition 3.1]{Ku2014} yields $n_1>n''$ and ${\sqsupset}\in\mathbf{P}_{\mathbb{P}_{n_1}}^{\mathbb{P}_{m_1}}$ such that ${h'\circ{\sqsupset}}={g''\circ{\geq}}$.  Take $s_1\in S\cap\mathbb{P}_{n_1}$ with $s_1\leq t$.  As $g''(t)=r'$, we must then have $p\sqsupset s_1$ with $h'(p)=r'$.  As $h'|_{\mathbb{P}_{m_1}\setminus\{e_{m_1}\}}\subseteq{h\circ{\geq}}$, the only possibility is $p=e_{m_1}$.  We can then take $f_1\in\mathbf{F}_{\mathbb{P}_{n_1}}^{\mathbb{P}_{m_1}}$ with $f_1\subseteq{\sqsupset}$ and $f_1(s_1)=e_{m_1}$.  Then $h'\circ f_1\subseteq g''\circ{\geq}$ and hence $g_0\circ f_1\subseteq f'\circ h'\circ f_1\subseteq f'\circ g''\circ{\geq}\subseteq{\geq}$.  Continuing in this way, we obtain functions $(f_k)$ and $(g_k)$ defining a bi-thin arrow sequence as in \Cref{BackAndForth} which thus yields a homeomorphism $\phi$ on $\mathsf{S}\mathbb{P}$.  Moreover, $\phi$ maps $S$ to $E$, as $(s_k)\subseteq S$, $(e_{m_k})\subseteq E$ and $f_k(s_k)=e_{m_k}$, for all $k\in\omega$.
\end{proof}

\subsection{A Dense Conjugacy Class}

For any graphs $G$ and $H$, the canonical edge-relation on $G\times H$ is the largest one making the coordinate projections edge-preserving, i.e.~where
\[\langle g,h\rangle\mathrel{\sqcap}\langle g',h'\rangle\qquad\Leftrightarrow\qquad g\mathrel{\sqcap}g'\quad\text{and}\quad h\mathrel{\sqcap}h'.\]
When a relation is both surjective and co-surjective, we call it \emph{bi-surjective}.

\begin{proposition}
    For any $Q\in\mathbf{P}$ and ${\leftarrow}\subseteq Q\times Q$, the following are equivalent.
    \begin{enumerate}
        \item\label{ConnectedDi} $\leftarrow$ is bi-surjective and connected $($as a subgraph of $Q\times Q)$.
        \item\label{StrongDi} ${\leftarrow}\subseteq{{\sqsupset}\circ{\sqin}}$ and ${=_R}\subseteq{\sqsubset}\circ{\leftarrow}\circ{\sqni}$, for some ${\sqsupset},{\sqni}\in\mathbf{P}_R^Q$.
        \item\label{VStrongDi} ${\leftarrow}=f\circ g^{-1}$, for some $f,g\in\mathbf{F}_R^Q$.
    \end{enumerate}
\end{proposition}

\begin{proof}\
    \begin{itemize}
        \item[\eqref{ConnectedDi}$\Rightarrow$\eqref{VStrongDi}] If $\leftarrow$ is connected then we have a surjective edge-preserving function $f$ from a path $R$ onto $\leftarrow$, i.e.~$f\in\mathbf{F}_R^{\leftarrow}$.  As $\leftarrow$ is bi-surjective, we can compose with the coordinate projections $p_1$ and $p_2$ to get $f_1=p_1\circ f\in\mathbf{F}_R^Q$ and $f_2=p_2\circ f\in\mathbf{F}_R^Q$, necessarily with ${\leftarrow}=f_1\circ f_2^{-1}$.
        
        \item[\eqref{VStrongDi}$\Rightarrow$\eqref{StrongDi}] Just note that if $f,g\in\mathbf{F}_R^Q$ and ${\leftarrow}=f\circ g^{-1}$ then ${=_R}\subseteq f^{-1}\circ f\circ g^{-1}\circ g={\sqsubset}\circ{\leftarrow}\circ{\sqni}$.

        \item[\eqref{StrongDi}$\Rightarrow$\eqref{ConnectedDi}] To start with, say ${=_R}\subseteq{\sqsubset}\circ{\leftarrow}\circ{\sqni}$.  For each $q\in Q$, the co-injectivity of $\sqsupset$ yields $r\in R$ with $r^\sqsubset=\{q\}$.  As $r\mathrel{{\sqsubset}\circ{\leftarrow}\circ{\sqni}}r$, we must therefore have $q'\in Q$ with $r\sqsubset q\leftarrow q'\sqni r$.  This shows that $\leftarrow$ is surjective while co-surjectivity of $\leftarrow$ follows likewise from the co-injectivity of $\sqni$.

        Now assume further that ${\leftarrow}\subseteq{{\sqsupset}\circ{\sqin}}$.  For any $\langle q_2,q_1\rangle,\langle q_2',q_1'\rangle\in{\leftarrow}$, we then have $s,s'\in R$ with $q_2\sqsupset s\sqin q_1$ and $q_2'\sqsupset s'\sqin q_1'$.  For each $r\in(s,s')$, we may pick $q_2^r\sqsupset r$ and $q_1^r\sqni r$ with $q_2^r\leftarrow q_1^r$, as ${=_R}\subseteq{\sqsubset}\circ{\leftarrow}\circ{\sqni}$.  As $\sqsupset$ and $\sqni$ are edge-preserving, so is the function $f$ from $[s,t]$ to $\leftarrow$ defined by $f(s)=\langle q_2,q_1\rangle$, $f(s')=\langle q_2',q_1'\rangle$ and $f(r)=\langle q_2^r,q_1^r\rangle$, for all $r\in(s,s')$.  Thus $\leftarrow$ is connected. \qedhere
    \end{itemize}
\end{proof}

Given graphs $G$ and $H$, let us define a smaller \emph{strict} edge-relation $\overline{\sqcap}$ on $G\times H$ by
\[\langle g,h\rangle\mathrel{\overline\sqcap}\langle g',h'\rangle\qquad\Leftrightarrow\qquad (g=g'\text{ and }h\mathrel{\sqcap}h')\text{ or }(g\mathrel{\sqcap}g'\text{ and }h=h').\]
We call a subset of $G\times H$ \emph{strictly connected} if it is connected as a subgraph of $(G\times H,\overline{\sqcap})$.

\begin{proposition}\label{StrictDigraphs}
    The strictly connected digraphs are co-initial in $\underline{\mathbf{P}}$.
\end{proposition}

\begin{proof}
    Say $\leftarrow$ is a $\mathbf{P}$-subfactorisable relation on $Q\in\mathbf{P}$ so ${=_R}\subseteq{{\sqsubset}\circ{\leftarrow}\circ{\sqni}}$, for some ${\sqsupset},{\sqni}\in\mathbf{P}_R^Q$.  Setting ${\twoheadleftarrow}={\leftarrow}\cap(\sqsupset\circ\sqin)$ we immediately see that ${\twoheadleftarrow}\subseteq{\sqsupset\circ\sqin}$ and again ${=_R}\subseteq{{\sqsubset}\circ{\twoheadleftarrow}\circ{\sqni}}$.  As $\mathbf{P}$ has amalgamation, every subfactorisable relation is automatically widely subfactorisable, by \Cref{WideSubfactorisability}.  Thus $(Q,\twoheadleftarrow)$ is a digraph and ${=_Q}\in\underline{\mathbf{P}}^\leftarrow_{\twoheadleftarrow}$, showing that connected digraphs are co-initial in $\underline{\mathbf{P}}$.  But if $\twoheadleftarrow$ is connected then ${\in_Q}\circ{\twoheadleftarrow}\circ{\ni_Q}$ is strictly connected, where as before ${\in_Q}\in\mathbf{F}_{\mathsf{F}Q}^Q$ is the membership relation restricted to $Q$ and its cliques $\mathsf{F}Q$.  As ${\in_Q}\in\underline{\mathbf{P}}^{\twoheadleftarrow}_{{\in_Q}\circ{\twoheadleftarrow}\circ{\ni_Q}}$, this shows that strictly connected digraphs are also co-initial in $\underline{\mathbf{P}}$.
\end{proof}

\begin{theorem}
    The autohomeomorphisms of the pseudoarc have a dense conjugacy class.
\end{theorem}

\begin{proof}
    As there are no restrictions on the morphisms in $\mathbf{P}$, we see that every autohomeomorphism of $\mathsf{S}\mathbb{P}$, for any given $\mathbf{P}$-Fra\"iss\'e poset $\mathbb{P}$, is automatically $\mathbf{P}$-compatible, i.e.~$\mathbf{C}_\mathbf{P}^\times$ is already the entire autohomeomorphism group.  By \Cref{CoinitialFromDirected} and \Cref{DirectedVsDenseCC}, it thus suffices to show that $\underline{\mathbf{P}}$ is directed.  By \Cref{StrictDigraphs}, it then suffices to show that strictly connected digraphs are directed.

    Accordingly, say we have strictly connected bi-surjective relations ${\leftarrow}\subseteq P_m\times P_m$ and ${\twoheadleftarrow}\subseteq P_n\times P_n$ on paths of length $m$ and $n$ respectively.  Let $l=(m+1)(n+1)-1$ and define $f\in\mathbf{F}_{P_l}^{P_m}$ and $g\in\mathbf{F}_{P_l}^{P_n}$ so that when $0\leq i\leq m$ and $0\leq j\leq n$,
    \begin{align*}
        f(p^l_{i(n+1)+j})&=p^m_i.\\
        g(p^l_{i(n+1)+j})&=\begin{cases}p^n_j&\text{if }i\text{ is even}\\p^n_{n-j}&\text{if }i\text{ is odd}.\end{cases}
    \end{align*}
    It suffices to show that the relation on $P_l$ given by ${\leftarrowtail}=(f^{-1}\circ{\leftarrow}\circ f)\cap(g^{-1}\circ{\twoheadleftarrow}\circ g)$ is also strictly connected and bi-surjective as it is then immediate that $f\in\underline{\mathbf{P}}^\leftarrow_{\leftarrowtail}$ and $g\in\underline{\mathbf{P}}^\twoheadleftarrow_{\leftarrowtail}$.

    For surjectivity, take any $i\leq m$ and $j\leq n$.  The surjectivity of $\leftarrow$ yields $i'\leq m$ with $f(p^l_{i(n+1)+j})=p^m_i\leftarrow p^m_{i'}=f(p^l_{i'(n+1)+j})$.  If $i$ and $i'$ are even then the surjectivity of $\twoheadleftarrow$ yields $j'\leq n$ with $g(p^l_{i(n+1)+j})=p^n_j\twoheadleftarrow p^n_{j'}=g(p^l_{i'(n+1)+j'})$ so $p^l_{i(n+1)+j}\leftarrowtail p^l_{i'(n+1)+j'}$.  Similarly, if $i$ is even but $i'$ is odd then again the surjectivity of $\twoheadleftarrow$ yields $j'\leq n$ such that $g(p^l_{i(n+1)+j})=p^n_j\twoheadleftarrow p^n_{n-j'}=g(p^l_{i'(n+1)+j'})$ and hence $p^l_{i(n+1)+j}\leftarrowtail p^l_{i(n+1)+j'}$.  The other cases are proved in the same way showing that $\leftarrowtail$ is surjective.  The co-surjectivity of $\leftarrowtail$ likewise follows from the co-surjectivity of $\leftarrow$ and $\twoheadleftarrow$.

    As $\leftarrow$ is strictly connected, we have $h\in\mathbf{F}_Q^\leftarrow$, viewing $\leftarrow$ as a subgraph of $P_m\times P_m$ with the strict edge relation.  So we have $h_1\in\mathbf{F}_Q^{P_m}$ and $h_2\in\mathbf{F}_Q^{P_m}$ with $h(q)=\langle h_1(q),h_2(q)\rangle$, for all $q\in Q$.  As $\twoheadleftarrow$ is strictly connected, for any $q\in Q$ so is the (non-empty) subgraph
    \[{\leftarrowtail_q}={\leftarrowtail}\cap[p^l_{h_1(q)(n+1)},p^l_{(h_1(q)+1)(n+1)})\times[p^l_{h_2(q)(n+1)},p^l_{(h_2(q)+1)(n+1)}).\]
    As ${\leftarrowtail}={\bigcup_{q\in Q}\leftarrowtail_q}$, to show that $\leftarrowtail$ is strictly connected, it then suffices to show that ${\leftarrowtail_{q_+}}\cup{\leftarrowtail_{q_-}}$ is strictly connected, for any adjacent $q_+,q_-\in Q$.  This is immediate if $h(q_+)=h(q_-)$.  Otherwise, $h(q_+)$ and $h(q_-)$ agree on one coordinate and differ on another.  Assume $h_1(q_-)=p^m_{i_1}$, $h_1(q_+)=p^m_{i_1+1}$ and $h_2(q_+)=p^m_{i_2}=h_2(q_-)$, for some even $i_1,i_2\leq m$.  Then we have $j\leq n$ with
    \[g(p^l_{(i_1+1)(n+1)})=g(p^l_{i_1(n+1)+n})=p^n_n\twoheadleftarrow p^n_j=g(p^l_{i_2(n+1)+j}),\]
    which means that $p^l_{(i_1+1)(n+1)}\leftarrowtail_{q_+}p^l_{i_2(n+1)+j}$ and $p^l_{i_1(n+1)+n}\leftarrowtail_{q_-}p^l_{i_2(n+1)+j}$.  As
    \[\langle p^l_{(i_1+1)(n+1)},p^l_{i_2(n+1)+j}\rangle\mathrel{\overline{\sqcap}}\langle p^l_{i_1(n+1)+n},p^l_{i_2(n+1)+j}\rangle,\]
    ${\leftarrowtail_{q_+}}\cup{\leftarrowtail_{q_-}}$ is indeed connected.  The other cases are proved similarly.
\end{proof}

\bibliography{HD}{}
\bibliographystyle{alphaurl}

\end{document}